\documentclass[11pt,a4paper]{amsart}
\usepackage{amsrefs}
\usepackage[T1]{fontenc}
\usepackage[latin1]{inputenc}
\usepackage[body={15cm, 24cm}]{geometry}
\usepackage{graphicx,mathrsfs,amssymb,longtable, color}

\newtheorem{theorem}{\textbf{Theorem}}[section]
\newtheorem{proposition}[theorem]{\textbf{Proposition}}
\newtheorem{lemma}[theorem]{\textbf{Lemma}}
\newtheorem{corollary}[theorem]{\textbf{Corollary}}

\theoremstyle{definition}
\newtheorem{definition}[theorem]{\textbf{Definition}}

\theoremstyle{remark}

\newtheorem{example}[theorem]{Example}
\newtheorem{remark}[theorem]{Remark}

\allowdisplaybreaks\let\inf\undefined
\DeclareMathOperator*{\inf}{inf\vphantom{p}}
\numberwithin{equation}{section}

\title[Sparse Mean-Field Control Problems]{Generalized Dynamic Programming Principle and Sparse Mean-Field Control Problems}

\author[G. Cavagnari]{Giulia Cavagnari}
\address{\hspace{-0.5em}\begin{tabular}{ll}Giulia Cavagnari:&Department of Mathematics ``F. Casorati'',\\& University of Pavia\\ &Via Ferrata 5, I-27100 Pavia, Italy.\end{tabular}}
\email{giulia.cavagnari@unipv.it}

\author[A. Marigonda]{Antonio Marigonda}
\address{\hspace{-0.5em}\begin{tabular}{ll}Antonio Marigonda:&Department of Computer Science,\\& University of Verona\\ &Strada Le Grazie 15, I-37134 Verona, Italy.\end{tabular}}
\email{antonio.marigonda@univr.it}

\author[B. Piccoli]{Benedetto Piccoli}
\address{\hspace{-0.5em}\begin{tabular}{ll}Benedetto Piccoli:&Department of Mathematical Sciences,\\& Rutgers University - Camden\\ &311 N. 5th Street
Camden, NJ 08102, USA.\end{tabular}}
\email{piccoli@camden.rutgers.edu}

\date{\today}

\thanks{Declaration of interest: none.}

\begin{document}

\begin{abstract}
In this paper we study optimal control problems in Wasserstein spaces, 
which are suitable to describe macroscopic dynamics of multi-particle systems.
The dynamics is described by a parametrized continuity equation, in which the Eulerian velocity field is affine w.r.t. some variables.
Our aim is to minimize a cost functional which includes
a control norm, thus enforcing a \emph{control sparsity} constraint. 
More precisely, we consider a nonlocal restriction on the total amount of control that can be used depending on the overall state of the evolving mass. 
We treat in details two main cases: an instantaneous constraint
on the control applied to the evolving mass and 
a cumulative constraint, which depends also on the amount of control
used in previous times.
For both constraints, we prove the existence of optimal trajectories for general cost functions and that the value function is viscosity solution of a suitable
Hamilton-Jacobi-Bellmann equation. 
Finally, we discuss an abstract Dynamic Programming Principle, providing further applications in the Appendix.
\end{abstract}

\maketitle

\section{Introduction}
The investigation of optimal control problems in the space of measures is attracting an increasing interest by the mathematical community in the last years, due to the potential
applications in the study of complex systems, or multi-agent systems. See for example \cite{CPT} for crowd dynamics models, and \cite{ChapRUC} for social network analysis.

\par\medskip\par
Many physical and biological phenomena can be modelled by 
\emph{multi-particle system} with a large number of particles.
At the \emph{microscopic} level, the behaviour of each particle is determined by local and nonlocal interactions. As the number of particles increases, the complexity of the system grows very fast.
It turns out that only a \emph{macroscopic (statistical)} description of the state of such a system can be actually provided,
in the sense that an observer can deduce the state of the system only by measuring averages of suitable quantities (as, e.g., in statistical mechanics).
\par\medskip\par
Assuming that there are neither creation nor loss of agents during the evolution, it is natural to describe the state of the system by a time-dependent probability 
measure $\mu_t$ on $\mathbb R^d$, in the sense that for every Borel set $A\subseteq\mathbb R^d$ and $t\ge 0$ the quantity $\mu_t(A)$ corresponds to the fraction 
of the total agents that are cointained in $A$ at time $t$.
\par\medskip\par
An alternative point of view is the following: suppose that the state of the system is expressed by a time-dependent real valued map $\Phi_t$ defined on the set of continuous and
bounded functions on $\mathbb R^d$. In this sense, $\Phi_t(\varphi)\in\mathbb R$ expresses the result of the observer's measurement of the quantity $\varphi(\cdot)$ on the system at time $t$.
If we assume that $\Phi_t$ is positive, linear and continuous, 
by Riesz representation theorem, we have that $\Phi_t$ - and hence the state of the system - can be uniquely represented by a Borel measure $\mu_t$ on $\mathbb R^d$, in the sense that 
\[\Phi_t(\varphi)=\int_{\mathbb R^d}\varphi(x)\,d\mu_t(x).\]
\par\medskip\par
In \cites{CMNP,CM,CMO} time-optimal control problems in the space of probability measures $\mathscr P(\mathbb R^d)$ are addressed,
by considering systems without interactions among the agents. The authors were able to extend to this framework some classical results, 
among which an Hamilton-Jacobi-Bellman (briefly HJB) equation solved by the minimum-time function in a suitable viscosity sense.
However, a full characterization of the minimum time function in terms of uniqueness of the solution of the corresponding HJB equation is still missing.
\par\medskip\par
In \cite{Cav}, the author focuses on results concerning conditions for attainability and regularity of the minimum time function in $\mathscr P(\mathbb R^d)$.
A set of tangent distributions in the 2-Wasserstein space is introduced in \cite{Aver} for the study of viability results, 
i.e. weak invariance properties of subsets of probability measures for mean-field type controlled evolutions, with a dynamics taking into account
possible interactions among the agents.
\par\medskip\par
The study of viscosity solutions for general Hamilton-Jacobi equations in the space of measures (see e.g. \cites{CQ,GNT,GT}) deeply involves 
the definition of suitable notions of sub/super-differentials (see \cites{AGS,Carda}). 
Recently, in \cite{MQ}, a Mayer problem without interactions was studied. For such a problem, under strong regularity assumptions,
the authors provided a full characterization of the value function as the unique viscosity solution of a suitable HJB equation in the class
of Lipschitz continuous functions. The comparison principle used to establish the uniqueness of the viscosity solution 
has been extended in the forthcoming \cite{JMQ18}, where much milder assumptions are required. 
\par\medskip\par
The characterization of the value function as the unique solution of a PDE in Wasserstein spaces is currently one of the most challenging issues in 
control problems for time-evolving measures.
From the point of view of the necessary conditions, in \cite{BR} the authors provide a Pontryagin Maximum principle for a problem with interactions encoded in the dynamics by means of a nonlocal vector field, 
under strong regularity on the control function.
Finally, in \cite{FLOS} the authors investigate $\Gamma$-convergence results for a controlled nonlocal dynamics, 
proving thus the consistency of the mean field model in a measure-theoretic setting with the corresponding one in finite dimension.
\par\medskip\par
In this paper we study an optimal control problem for a multi-particle system subject to the influence of an external controller,
who aims to minimize a general cost function satisfying some basic properties.
The evolution of the system starting from an initial configuration $\mu_0\in\mathscr P(\mathbb R^d)$ is described by an absolutely continuous curve $t\mapsto \mu_t$ in the space of probability  measures,
and, recalling that the total mass is preserved along the evolution, the macroscopic dynamics will be expressed by the continuity equation
\[\begin{cases}
\partial_t\mu_t+\mathrm{div}\,(v_t\mu_t)=0,\\
\mu_{|t=0}=\mu_0.
\end{cases}\]
The link between the macroscopic dynamics and the controller's influence on each individual at the microscopic level is expressed by the constraint
on the Borel vector field $v$, requiring that 
\[v_t(x)\in F(x)\textrm{ for a.e. $t\in[0,T]$ and $\mu_t$-a.e. $x\in\mathbb R^d$},\]
where the set-valued map
\[F(x):=\left\{f(x,u):=f_0(x)+A(x)\,u\,:\,u\in U\right\}\]
represents the admissible velocities for the particles that transit through the point $x\in\mathbb R^d$,
$U\subset\mathbb R^m$ is a given convex and compact \emph{control set}, with $0\in U$, 
and $A(x)$ is a $d\times m$ matrix whose columns are the evaluations at $x\in\mathbb R^d$ of continuous vector fields $f_i:\mathbb R^d\to\mathbb R^d$, $i=1,\dots m$, 
with  uniform linear growth.
\par\medskip\par
The starting point of the problem we discuss is similar to the \emph{sparse control strategy} for multi-agent systems addressed in \cites{CFPT,FPR}, where 
the authors study the evolution of the entire population in the case of a control action concentrated only on a small portion of the total number of agents. 
\par\medskip\par
In this paper, we consider a similar interaction constraint but, instead of dealing with a nonlocal dynamics, we consider it as part of the cost functional to be minimized.
More precisely,
we define the \emph{control magnitude density}  $\Psi:\mathbb R^d\times\mathbb R^d\to [0,+\infty]$ by 
\[\Psi(x,v):=\begin{cases}\displaystyle\inf\left\{|u|:\,u\in U,\,f(x,u)=v\right\},&\textrm{ if $v\in F(x)$},\\ \\+\infty,&\textrm{ otherwise,}\end{cases}\]
which, under suitable assumptions, turns out to be continuous on $\mathrm{Graph}\,F$.
This map gives the minimum norm for a control generating an admissible velocity $v$ at $x$. In particular, the integral
\[\int_{\mathbb R^d}\Psi(x,v_t(x))\,d\mu_t,\]
can be considered as a measure of the \emph{effort} of the controller, combining the magnitude of the control used to implement the vector field $v_t(\cdot)$ 
and the quantity $\mu_t$ of particles to be controlled at time $t$.
The above integral can thus be naturally used to impose \emph{control sparsity constraints}, 
in the form of upper bounds on the controller's effort to drive the mass of particles. 
\par\medskip\par
To this aim, the constraint we will consider in this paper will be upper bounds on the $L^\infty$ norm or on the $L^1$ norm
of the map 
\begin{equation}\label{eq:costIntro}
t\mapsto\int_{\mathbb R^d}\Psi(x,v_t(x))\,d\mu_t.
\end{equation}
In the case of $L^\infty$ bound, at every instant of time the controller must choose how to use the maximum amount of effort available to control the population:
we notice that a weak action distributed on a large number of particles, or a strong action on a few number of individuals may require the same effort.
In the case of $L^1$ bound, the effects are cumulated in time, thus all the past history of the evolution must be taken into account.
\par\medskip\par
We stress out that the two kinds of interaction considered are at the \emph{macroscopic} level, not at the microscopic one. Indeed, they do not involve directly the behavior of each individual w.r.t. the others (like in the case of attraction/repulsion potentials). Instead, we are interested in the particles' \emph{collective evolution} and on the \emph{effort} required to generate it.
\par\medskip\par
Considering both our cases of interest, the main goals of this paper are the following:
\begin{itemize}
\item to study compactness results for the set of feasible trajectories;
\item to prove the existence of optimal trajectories for general cost functions;
\item to provide necessary conditions in the form of an HJB equation solved by the value function in a suitable viscosity sense.
\end{itemize}
In order to treat in a unified way the two cases, in Section \ref{sec:GF}, we provide an abstract framework for optimal control problems
proving a general Dynamic Programming Principle. Further possible applications of this framework are discussed in Appendix \ref{sec:appendix}.
\medskip

The paper is organized as follows: in Section \ref{sec:prelim} we recall some basic notion about Wasserstein spaces and fix the notation; 
in Section \ref{sec:rainbow} we outline the problem providing some preliminary results; in Section \ref{sec:GF} we prove a general 
Dynamic Programming Principle; in Section \ref{sec:sparse} we analyze the two control-sparsity problems ($L^\infty$-time averaged sparsity 
and $L^1$-time averaged sparsity cases) proving existence of optimal trajectories for general cost functions. 
In Section \ref{sec:HJB}, we study for each of the considered cases an Hamilton-Jacobi-Bellman equation solved 
by the value function in some suitable viscosity sense. In Section \ref{sec:examples}, we show that this theory can be applied to a more specific cost functional leading to the minimum time function. In Section \ref{sec:appl} we provide an application motivating our theoretical study and particularly our interest on the cost \eqref{eq:costIntro}.  Finally, in Appendix \ref{sec:appendix} we discuss other applications 
of the framework outlined in Section \ref{sec:GF}, while in Appendix \ref{sec:moments} we recall some estimates borrowed from \cite{CMNP,Cav}.

\section{Preliminaries and notation}\label{sec:prelim}

As main references concerning optimal transport and measure theory, the reader may refer to \cites{AGS,Santa,Vil}. \par\medskip\par
We will use the following notation.\par\medskip\par
{\small\begin{longtable}{ll}
$B(x,r)$&the open ball of radius $r$ of a normed space $X$,\\ &i.e., $B(x,r):=\{y\in X:\,\|y-x\|_X<r\}$;\\
$\overline K$&the closure of a subset $K$ of a topological space $X$;\\
$\mathrm{Id}_X(\cdot)$&the identity function of $X$,\\ &i.e. $\mathrm{Id}_X(x)=x$ for all $x\in X$;\\
$I_K(\cdot)$&the indicator function of $K$,\\ &i.e. $I_K(x)=0$ if $x\in K$, $I_K(x)=+\infty$ if $x\notin K$;\\
$\chi_K(\cdot)$&the characteristic function of $K$,\\ &i.e. $\chi_K(x)=1$ if $x\in K$, $\chi_K(x)=0$ if $x\notin K$;\\
$C^0_b(X;Y)$&the set of continuous bounded function from a Banach space $X$ to $Y$, \\
&endowed with $\|f\|_{\infty}=\displaystyle\sup_{x\in X}|f(x)|$ (if $Y=\mathbb R$, $Y$ will be omitted);\\
$C^0_c(X;Y)$&the set of compactly supported functions of $C^0_b(X;Y)$, \\
&with the topology induced by $C^0_b(X;Y)$;\\
$\Gamma_I$&the set of continuous curves from a real interval $I$ to $\mathbb R^d$;\\
$\Gamma_T$&the set of continuous curves from $[0,T]$ to $\mathbb R^d$;\\
$e_t$&the evaluation operator $e_t:\mathbb R^d\times\Gamma_I$\\ &defined by $e_t(x,\gamma)=\gamma(t)$ for all $t\in I$;\\
$\mathscr P(X)$&the set of Borel probability measures on a separable metric space $X$,\\
&endowed with the weak$^*$ topology induced by $C^0_b(X)$;\\
$\mathscr M(\mathbb R^d;\mathbb R^d)$&the set of vector-valued Borel measures on $\mathbb R^d$ with values in $\mathbb R^d$,\\
&endowed with the weak$^*$ topology induced by $C^0_c(\mathbb R^d;\mathbb R^d)$;\\
$\mathrm{supp}\,\mu$&the support of the measure $\mu$;\\
$|\nu|$&the total variation of a measure $\nu\in \mathscr M(\mathbb R^d;\mathbb R^d)$;\\
$\ll$&the absolutely continuity relation between measures defined on the same\\
&$\sigma$-algebra;\\
$\mathrm{m}_p(\mu)$&the $p$-th moment of a probability measure $\mu\in\mathscr P(X)$,\\ &i.e., $\mathrm{m}_p(\mu)=\|\mathrm{Id}_X\|^p_{L^p_{\mu}}$;\\
$r\sharp\mu$&the push-forward of the measure $\mu\in\mathscr P(X)$ by the Borel map $r$\\ &(see Definition \ref{def:push});\\
$\mu\otimes\eta_x$&the product measure of $\mu\in\mathscr P(X)$ with the Borel family of measures\\ &$\{\eta_x\}_{x\in X}$ (see Section 5.3 in \cite{AGS});\\
$\pi_i$&the $i$-th projection map $\pi_i(x_1,\dots,x_N)=x_i$;\\
$\pi_{ij}$&the $i,j$-th projection map $\pi_{ij}(x_1,\dots,x_N)=(x_i,x_j)$;\\
$\mathscr P_p(X)$&the subset of the elements $\mathscr P(X)$ with finite $p$-moment, \\
$W_p(\mu,\nu)$&the $p$-Wasserstein distance between $\mu$ and $\nu$ (see Definition \ref{def:was});\\
$\Pi(\mu,\nu)$&the set of admissible transport plans from $\mu$ to $\nu$ (see Definition \ref{def:was});\\
&endowed with the $p$-Wasserstein distance;\\
$\mathscr L^d$&the Lebesgue measure on $\mathbb R^d$;\\
$\dfrac{\nu}{\mu}$&the Radon-Nikodym derivative of the measure $\nu$ w.r.t. the measure $\mu$;\\
$\mathrm{Lip}(f)$&the Lipschitz constant of a function $f$.\\
\end{longtable}}

\begin{definition}[Pushforward measure]\label{def:push}
Given two separable metric spaces $X,Y$, $\mu\in\mathscr P(X)$, and a Borel map $r:X\to Y$, 
we define the \emph{push forward measure} $r\sharp\mu\in\mathscr P(Y)$ by 
$r\sharp\mu(B):=\mu(r^{-1}(B))$ for all Borel sets $B\subseteq Y$, or equivalently,
\[\int_X f(r(x))\,d\mu(x)=\int_Y f(y)\,d r\sharp\mu(y),\]
for every bounded (or $r\sharp\mu$-integrable) Borel function $f:Y\to\mathbb R$. 
\end{definition}

We refer to Chapter 5, Section 2 of \cite{AGS} for the main properties of the pushforward measures.

The Wasserstein distance and its basic properties are recalled below.

\begin{definition}[Wasserstein distance]\label{def:was}
Given $\mu_1,\mu_2\in\mathscr P(\mathbb R^d)$, $p\ge 1$, we define the $p$-\emph{Wasserstein distance} between $\mu_1$ and $\mu_2$
by setting
\begin{equation}\label{eq:wasserstein}
W_p(\mu_1,\mu_2):=\left(\inf\left\{\iint_{\mathbb R^d\times\mathbb R^d}|x_1-x_2|^p\,d\pi(x_1,x_2)\,:\,\pi\in \Pi(\mu_1,\mu_2)\right\}\right)^{1/p}\,,
\end{equation}
where the set of \emph{admissible transport plans} $\Pi(\mu_1,\mu_2)$ is defined by
\begin{align*}
\Pi(\mu_1,\mu_2):=&\left\{\pi\in \mathscr P(\mathbb R^d\times\mathbb R^d):\, \begin{array}{c}\pi(A_1\times \mathbb R^d)=\mu_1(A_1),\\ \pi(\mathbb R^d\times A_2)=\mu_2(A_2),\end{array}\right.\\
&\hspace{2cm}\left.\phantom{\int}\text{ for all }\mu_i\text{-measurable sets }A_i,\,i=1,2\right\}.
\end{align*}
\end{definition}

\begin{proposition}\label{prop:wassconv}
$\mathscr P_p(\mathbb R^d)$ endowed with the $p$-Wasserstein metric $W_p(\cdot,\cdot)$ is a complete separable metric space.
Moreover, given a sequence $\{\mu_n\}_{n\in\mathbb N}\subseteq\mathscr P_p(\mathbb R^d)$ and $\mu\in\mathscr P_p(\mathbb R^d)$,
we have that the following are equivalent
\begin{enumerate}
\item $\displaystyle\lim_{n\to\infty}W_p(\mu_n,\mu)=0,$
\item $\mu_n\rightharpoonup^*\mu$ and $\{\mu_n\}_{n\in\mathbb N}$ has uniformly integrable $p$-moments.
\end{enumerate}
\end{proposition}
\begin{proof}
See Proposition 7.1.5 in \cite{AGS}.
\end{proof}

Concerning disintegration results for measures, widely used in this paper, we refer the reader to Section 5.3 in \cite{AGS}.
The following result is Theorem 5.3.1 in \cite{AGS}.

\begin{theorem}[Disintegration]
Given a measure $\mu\in\mathscr P(\mathbb X)$ and a Borel map $r:\mathbb X\to X$, there exists a family of probability
measures $\{\mu_x\}_{x\in X}\subseteq \mathscr P(\mathbb X)$,
uniquely defined for $r\sharp \mu$-a.e. $x\in X$, such that $\mu_x(\mathbb X\setminus r^{-1}(x))=0$ for $r\sharp \mu$-a.e. $x\in X$, and
for any Borel map $\varphi:X\times Y\to [0,+\infty]$ we have
\[\int_{\mathbb X}\varphi(z)\,d\mu(z)=\int_X \left[\int_{r^{-1}(x)}\varphi(z)\,d\mu_x(z)\right]d(r\sharp \mu)(x).\]
We will write $\mu=(r\sharp \mu)\otimes \mu_x$.
If $\mathbb X=X\times Y$ and $r^{-1}(x)\subseteq\{x\}\times Y$ for all $x\in X$, we can identify each measure $\mu_x\in\mathscr P(X\times Y)$
with a measure on $Y$.
\end{theorem}

\section{Setting of the problem and preliminary results}\label{sec:rainbow}

We study an optimal control problem in the space of probability measures with a \emph{control sparsity} constraint. We develop separately two specific constraints: an $L^\infty$-time 
averaged sparsity condition (Section \ref{sec:linftymu}) and an $L^1$-time averaged sparsity constraint (Section \ref{sec:l1mu}). In both cases, the sparsity-interaction term is encoded in the cost functional to be minimized.
In the first case, at a.e. time instant we impose an upper bound on the magnitude of control to be used on the evolving mass. In the second case, this constraint is $L^1$ in time.
Both these problems are strongly motivated by applications to multi-particle systems.
\par\medskip\par
In this section we introduce and discuss some preliminary properties regarding the dynamics and the objects that will be used in Section \ref{sec:sparse} and also in Appendix \ref{sec:appendix} to describe some control sparsity constraints. 

\begin{definition}[Standing assumptions]\label{def:setting}
Let $I\subseteq\mathbb R$ be a nonempty compact interval, $U\subseteq \mathbb R^m$ be a convex compact subset, with $0\in U$, called the \emph{control set}. Let $f_i\in C^0(\mathbb R^d;\mathbb R^d)$, $i=0,\dots, m$ satisfying the following conditions
\begin{enumerate}
\item \emph{growth condition:}
\[C:=\sup_{x\in\mathbb R^d}\left\{\dfrac{1}{|x|+1}\sum_{i=0}^m|f_i(x)|\right\}<+\infty;\]
\item \emph{rank condition:} the $d\times m$ matrix $A(x):=\left(f_1(x), f_2(x),\dots f_m(x)\right)$ has rank independent of $x\in\mathbb R^d$.
\end{enumerate}
We define the set-valued map $F:\mathbb R^d\rightrightarrows\mathbb R^d$ by setting 
\[F(x):=\left\{f(x,u):=f_0(x)+A(x)\,u\,:\,\,u\in U\right\}.\]
This multifunction $F$ will govern the controlled dynamics in terms of a differential inclusion, as described in Definition \ref{def:adm-traj}.
The \emph{graph} of $F(\cdot)$ is the set
\[\mathrm{Graph}\,F:=\{(x,v)\in\mathbb R^d\times\mathbb R^d:\,v\in F(x)\}.\]
We define the \emph{control magnitude density} $\Psi:\mathbb R^d\times\mathbb R^d\to [0,+\infty]$ by 
\[\Psi(x,v)=\begin{cases}\displaystyle\inf\left\{|u|:\,u\in U,\,f(x,u)=v\right\},&\textrm{ if $v\in F(x)$},\\ +\infty,&\textrm{ otherwise.}\end{cases}\]
\end{definition}

\begin{lemma}\label{lemma:bprop}
In the framework of Definition \ref{def:setting}, the following properties hold.
\begin{enumerate}
\item The set-valued map $F(\cdot)$ is continuous w.r.t. the Hausdorff metric, and it has nonempty, compact and convex values at every $x\in\mathbb R^d$. In particular, $\mathrm{Graph}\,F$ is closed. Moreover,
$F$ satisfies the linear growth condition, i.e., there exists a constant $D>0$ such that $F(x)\subseteq \overline{B(0,D(|x|+1))}$ for every $x\in\mathbb R^d$.
\item The map $\Psi:\mathbb R^d\times\mathbb R^d\to [0,+\infty]$ satisfies the following properties
\begin{itemize}
\item[$(i)$] for all $v\in F(x)$, there exists $w\in\mathbb R^m$ such that  
\[\Psi(x,v)=|A(x)^+(v-f_0(x))+(I-A(x)^+\,A(x))\,w|,\] where $A(x)^+$ is the Moore-Penrose pseudo-inverse of $A(x)$, and $I$ is the $m\times m$ identity matrix. Moreover $\Psi(\cdot,\cdot)$ is continuous on its domain, i.e. on $\mathrm{Graph}\,F$;
\item[$(ii)$] $\Psi(\cdot,\cdot)$ is l.s.c. in $\mathbb R^d\times\mathbb R^d$;
\item[$(iii)$] $v\mapsto\Psi(x,v)$ is convex for any $x\in\mathbb R^d$.
\end{itemize}
\end{enumerate}
\end{lemma}
\begin{proof}
Item (1) follows directly from the definition, recalling the continuity of $f_i$, $i=1,\dots, m$, the compactness of $U$ and the standing assumptions. We can take $D=C\cdot R_U$, where we define $R_U:=\max\{|u|:u\in U\}$.

We prove $(i)$. Let $x\in\mathbb R^d$ be fixed, $v\in F(x)$ and let us denote with $A(x)^+$ the Moore-Penrose pseudo-inverse of $A(x)$, which exists and is unique (see \cite{Pen}). 
Then, by hypothesis, there exists a solution $u$ of $A(x)\,u=v-f_0(x)$ and by pseudo-inverse properties, we can characterize any such a solution by $u=A(x)^+\,(v-f_0(x))+(I-A(x)^+\,A(x))\,w$, $w\in\mathbb R^m$.

Let us define the map $g:\mathbb R^d\times\mathbb R^d\times\mathbb R^m\to\mathbb R^m$, $g(x,v,w):=A(x)^+\,(v-f_0(x))+(I-A(x)^+\,A(x))\,w$. 
We have that $g$ is continuous, indeed by the rank condition in Definition \ref{def:setting} and pseudo-inverses properties, since $x\mapsto A(x)$ is continuous, so is $x\mapsto A(x)^+$.

Thus, by Proposition 1.4.14 in \cite{AuF}, we have that the multifunction $M:\mathbb R^d\times\mathbb R^d\rightrightarrows\mathbb R^m$, defined by $M(x,v):=\{g(x,v,w)\,:\,w\in \overline{B_{R_U}(0)}\}$ is continuous.
Hence, so is the set-valued function $Q:\mathbb R^d\times\mathbb R^d\rightrightarrows\mathbb R^m$, $Q(x,v):=M(x,v)\cap U$. By Corollary 9.3.3 in \cite{AuF}, we get the continuity of the minimal norm selection of $Q$, thus the continuity of $\Psi$ on its domain, i.e. on $\mathrm{Graph}\,F$, since
\[\Psi(x,v)=\min_{z\in Q(x,v)}|z|,\]
when $(x,v)\in\mathrm{Graph}\,F$.

Finally, $(ii)$ follows by the continuity of $\Psi$ on $\mathrm{dom}\,\Psi\equiv\mathrm{Graph}\,F$ and closedness of $\mathrm{Graph}\,F$.

We pass now to the proof of $(iii)$. Let $(x,v_j)\in\mathbb R^d\times\mathbb R^d$, $j=1,2$. If $(x,v_j)\notin \mathrm{Graph}\,F$, for some $j=1,2$, the convexity inequality is trivially satisfied, since $\Psi$ takes
the value $+\infty$, so we assume that $(x,v_j)\in \mathrm{Graph}\,F$, $j=1,2$. 
We notice that, since for all $u_j\in Q(x,v_j)$, $j=1,2$, we have $v_j-f_0(x)=A(x)\,u_j$,
then $\lambda u_1+(1-\lambda)u_2\in Q(x,\lambda v_1+(1-\lambda)v_2)$ for all $\lambda\in [0,1]$.
Recalling the triangular inequality, we have for all $u_j\in Q(x,v_j)$, $j=1,2$,
\begin{align*}
\Psi(x,\lambda v_1+(1-\lambda)v_2)=&\min\left\{|u|:\, u\in Q(x,\lambda v_1+(1-\lambda)v_2)\right\}\\ 
\le&|\lambda u_1+(1-\lambda)u_2|\le \lambda|u_1|+(1-\lambda)|u_2|,
\end{align*}
By taking the minimum on $u_j\in Q(x,v_j)$, $j=1,2$, we obtain that $\Psi(x,\cdot)$ is convex.
\end{proof}

Following the same line as in \cites{Cav,CMNP,CMO,CM,MQ}, we define the set of \emph{admissible trajectories} as follows.

\begin{definition}[Admissible trajectories]\label{def:adm-traj}
Let $I\subseteq\mathbb R$ be a compact and nonempty interval. In the setting of Definition~\ref{def:setting}, we say that $(\boldsymbol\mu,\boldsymbol\nu)$ is an \emph{admissible trajectory} defined on $I$, and we write $(\boldsymbol\mu,\boldsymbol\nu)\in\mathcal A_I$, if the time-depending Borel probability measure $\boldsymbol\mu=\{\mu_t\}_{t\in I}\subseteq\mathscr P(\mathbb R^d)$ and the time-depending Borel vector-valued measure $\boldsymbol\nu=\{\nu_t\}_{t\in I}\subseteq\mathscr M(\mathbb R^d;\mathbb R^d)$ satisfy the following properties
\begin{itemize}
\item[(A1)] \emph{continuity equation}: we have $\partial_t\mu_t+\mathrm{div}\,\nu_t=0$ in the sense of distributions in $I\times \mathbb R^d$;
\item[(A2)] \emph{velocity constraint}: $|\nu_t|\ll \mu_t$ for a.e. $t\in I$ and the Radon-Nikodym derivative satisfies $\dfrac{\nu_t}{\mu_t}(x)\in F(x)$, 
for $\mu_t$-a.e. $x\in\mathbb R^d$ and a.e. $t\in I$. Equivalently, we ask
\[\int_I \mathcal J_F(\mu_t,\nu_t)\,dt<+\infty,\]
where $\mathcal J_F:\mathscr P(\mathbb R^d)\times\mathscr M(\mathbb R^d;\mathbb R^d)\to [0,+\infty]$ is defined by
\begin{align*}
\mathcal J_F(\mu,E)&:=\begin{cases}
\displaystyle\int_{\mathbb R^d} I_{F(x)}\left(\dfrac{E}{\mu}(x)\right)\,d\mu(x),
&\textrm{if }|E|\ll \mu,\\ \\
+\infty,&\textrm{otherwise}.\end{cases}
\end{align*}
\end{itemize}
Given $\bar\mu\in\mathscr P(\mathbb R^d)$, we say that $(\boldsymbol\mu,\boldsymbol\nu)\in\mathcal A_I$ is an admissible trajectory \emph{starting from $\bar\mu$} if $\mu_{|t=\min I}=\bar\mu$, and we write $(\boldsymbol\mu,\boldsymbol\nu)\in\mathcal A_I(\bar\mu)$.
It can be proved (see \cite{AGS}) that every $\boldsymbol\mu$ such that $(\boldsymbol\mu,\boldsymbol\nu)\in\mathcal A_I$ for some $\boldsymbol\nu$ admits a narrowly continuous representative. Thus from now on we will always refer to it.
\end{definition}

\begin{remark}\label{rem:basic estim}
Let $\bar\mu\in\mathscr P_p(\mathbb R^d)$ with $p\ge 1$, and $(\boldsymbol\mu,\boldsymbol\nu)\in\mathcal A_I(\bar\mu)$. 
Then, due to the growth assumption on $F(\cdot)$, it is possible to bound the moments of $\mu_t$ in terms of the moments of $\mu_0$ (see Proposition \ref{prop:moments}), 
and hence we have $\boldsymbol\mu\subseteq\mathscr P_p(\mathbb R^d)$.
\end{remark}

Recalling the Superposition Principle (Theorem 8.2.1 in \cite{AGS}) and its extension to differential inclusions provided in Theorem 1 in \cite{CMPbulgaro}, we refer to the following.

\begin{definition}[Probabilistic representations]\label{def:representation}
Let $\boldsymbol\mu=\{\mu_t\}_{t\in I}\subseteq\mathscr P_p(\mathbb R^d)$ be an absolutely continuous trajectory and $\boldsymbol\nu=\{\nu_t\}_{t\in I}$ a family of Borel vector-valued measures such that $\partial_t\mu_t+\mathrm{div}\,\nu_t=0$, $t\in I$. We say that a probability measure $\boldsymbol\eta\in\mathscr P(\mathbb R^d\times\Gamma_I)$
\begin{enumerate}
\item\label{item1:repres} \emph{represents the pair $(\boldsymbol\mu,\boldsymbol\nu)$}, if $\boldsymbol\eta$ is concentrated on the couples $(x,\gamma)\in \mathbb R^d\times\Gamma_I$ where $\gamma$ satisfies $\dot\gamma(t)=\displaystyle\frac{\nu_t}{\mu_t}(\gamma(t))$, $\gamma(\min I)=x$, and $\mu_t=e_t\sharp\boldsymbol\eta$ for all $t\in I$;
\item\label{item2:repres} \emph{represents $\boldsymbol\mu$}, if $\boldsymbol\eta$ is concentrated on the couples $(x,\gamma)\in \mathbb R^d\times\Gamma_I$ where $\gamma$ satisfies $\dot\gamma(t)\in F(\gamma(t))$, $\gamma(\min I)=x$, and $\mu_t=e_t\sharp\boldsymbol\eta$ for all $t\in I$.
\end{enumerate}
\end{definition}

We notice that in general if $(\boldsymbol\mu,\boldsymbol\nu)\in\mathcal A_I$, then $\boldsymbol\mu$ can have more than a representation (see \cite{CMPbulgaro} for an example of this situation).

Recalling Theorem 8.2.1 in \cite{AGS}, and Theorem 1 in \cite{CMPbulgaro}, we have that
\begin{lemma}[Equivalence]\label{lemma:equiv-def}
In the setting of Definition~\ref{def:setting}.
\begin{enumerate}
\item\label{item1:equiv-def} Every $(\boldsymbol\mu,\boldsymbol\nu)\in\mathcal A_I$ admits a representation $\boldsymbol\eta\in\mathscr P(\mathbb R^d\times\Gamma_I)$ according to Definition \ref{def:representation}\eqref{item1:repres}.
\item\label{item2:equiv-def} Every $\boldsymbol\eta\in\mathscr P(\mathbb R^d\times\Gamma_I)$ concentrated on the couples $(x,\gamma)\in \mathbb R^d\times\Gamma_I$, where $\dot\gamma(t)\in F(\gamma(t))$ and $\gamma(0)=x$,
represents an admissible trajectory $(\boldsymbol\mu,\boldsymbol\nu)\in\mathcal A_I$,  with $\boldsymbol\nu=\{\nu_t\}_{t\in I}$ with $\nu_t=v_t\mu_t$ and
\[v_t(y)=\int_{e_t^{-1}(y)}\dot\gamma(t)\,d\eta_{t,y}(x,\gamma),\]
for a.e. $t\in I$ and $\mu_t$-a.e. $x\in\mathbb R^d$, where the Borel family of measures $\{\eta_{t,y}\}_{y\in\mathbb R^d}\subseteq\mathscr P(\mathbb R^d\times\Gamma_I)$ is the disintegration of $\boldsymbol\eta$ w.r.t. 
the evaluation operator $e_t$.
\end{enumerate}
\end{lemma}

\begin{remark}
Lemma \ref{lemma:equiv-def} allows us to consider equivalently an admissible trajectory defined as in Definition \ref{def:adm-traj}, or a probability measure $\boldsymbol\eta\in\mathscr P(\mathbb R^d\times\Gamma_T)$
satisfying the property of Lemma \ref{lemma:equiv-def} (2). 
\end{remark}

\begin{definition}\label{def:reprSetA}
Given $\bar\mu\in\mathscr P(\mathbb R^d)$, we define
\begin{equation*}
\mathcal R_I^{\mathcal A}(\bar\mu):=\left\{\boldsymbol\eta\in\mathscr P(\mathbb R^d\times\Gamma_I)\,:\,\exists (\boldsymbol\mu,\boldsymbol\nu)\in\mathcal A_I(\bar\mu) \textrm{ s.t. } \boldsymbol\eta \textrm{ represents }(\boldsymbol\mu,\boldsymbol\nu)\right\}.
\end{equation*}
\end{definition}

 \begin{remark}
 We stress that in general
 \begin{equation*}
 \mathcal R_I^{\mathcal A}(\bar\mu)\subsetneq\left\{\boldsymbol\eta\in\mathscr P(\mathbb R^d\times\Gamma_I)\,:\,\exists (\boldsymbol\mu,\boldsymbol\nu)\in\mathcal A_I(\bar\mu) \textrm{ s.t. }
 \boldsymbol\eta \textrm{ represents }\boldsymbol\mu\right\}.
 \end{equation*}
 Indeed, in the left-hand set we are requiring $\boldsymbol\eta$ to be a representation for the pair $(\boldsymbol\mu,\boldsymbol\nu)$ as in Definition~\ref{def:representation}(\ref{item1:repres}), while in the right-hand side we are exploiting the various possibilities for the construction of a representation recalled in Definition~\ref{def:representation}(\ref{item2:repres}). We borrow the following clarifying example from \cite{CMPbulgaro}.
 \end{remark}
 
\begin{example}
In $\mathbb R^2$, let 
\begin{itemize}
\item[$\cdot$] $\mathscr A=\{\gamma_{x,y}(\cdot)\}_{(x,y)\in\mathbb R^2}\subseteq AC([0,2])$ where $\gamma_{x,y}(t)=(x+t,y-t\,\mathrm{sgn}\,y)$ for any $(x,y)\in\mathbb R^2$, $t\in [0,2]$, with  $\mathrm{sgn}(0)=0$;
\item[$\cdot$] $F:\mathbb R^2\rightrightarrows\mathbb R^2$, $F(x,y)\equiv [-1,1]\times[-1,1]$ for all $(x,y)\in\mathbb R^2$;
\item[$\cdot$] $\mu_0=\dfrac12\delta_0\otimes\mathscr L^1_{|[-1,1]}\in\mathscr P(\mathbb R^2)$, $\boldsymbol\eta=\mu_0\otimes \delta_{\gamma_{x,y}}\in\mathscr P(\mathbb R^2\times\Gamma_2)$, $\boldsymbol\mu=\{\mu_t\}_{t\in[0,2]}$ with $\mu_t=e_t\sharp\boldsymbol\eta$;
\item[$\cdot$] $Q$ be the open square of vertice $\left\{(0,0), (1,0), (1/2,\pm 1/2)\right\}$. 
\end{itemize}
By construction we have that
\begin{itemize}
\item[$\cdot$] $F$ is in the form of Definition \ref{def:setting} (by taking for instance as $f_0$ the null function, $A(x)$ be the $2\times 2$ identity matrix and $\mathbb R^d\ni U=[-1,1]\times[-1,1]$) and $\dot\gamma(t)\in F(\gamma(t))$ for all $\gamma\in\mathscr A$ and $t\in]0,2[$. 
\item[$\cdot$] $\boldsymbol\mu$ is an admissible trajectory and we denote with $\boldsymbol\nu=\{\nu_t\}_{t\in [0,2]}$ its driving family of Borel vector-valued measures.
\end{itemize}
Denoted by $v_t=\dfrac{\nu_t}{\mu_t}$ the mean vector field, this implies $v_t(x,y)=(1,0)$ for all $(x,y)\in Q\setminus(\mathbb R\times \{0\})$ and $t=x$.
Now, consider the associated characteristics $\dot{\tilde{\gamma}}_y(t)=v_t(\tilde\gamma_y(t))$, $\tilde\gamma_y(0)=(0,y)$, $y\in[-1,1]$, and let us build $\tilde{\boldsymbol\eta}=\mu_0\otimes\delta_{\tilde\gamma_y}\in\mathscr P(\mathbb R^d\times\Gamma_{[0,2]})$.\\
We notice that, by construction, $\tilde{\boldsymbol\eta}$ represents $(\boldsymbol\mu,\boldsymbol\nu)$ (in particular it represents $\boldsymbol\mu$), while $\boldsymbol\eta$ represents $\boldsymbol\mu$ but not the pair $(\boldsymbol\mu,\boldsymbol\nu)$, since it is not constructed on the mean vector field $\displaystyle\frac{\nu_t}{\mu_t}$.
\end{example}

In Theorem 3 in \cite{CMNP} the authors give sufficient conditions providing compactness of the set $\mathcal A_I(\bar\mu)$ w.r.t the uniform convergence of curves in $W_p$, with $\bar\mu\in\mathscr P_p(\mathbb R^d)$, $p\ge1$. While Proposition 1 in \cite{CMPbulgaro} states the compactness of the set $\mathcal R_I^{\mathcal A}(\bar\mu)$.

\medskip

\begin{lemma}[Norm-minimal control density]\label{lemma:MinContr}
Given $(\boldsymbol\mu,\boldsymbol\nu)\in\mathcal A_I$, there exists a Borel map $u:I\times\mathbb R^d\to U$, defined $\mu_t$-a.e. $x\in\mathbb R^d$ and a.e. $t\in I$, such that $\boldsymbol\nu=\{\nu_t=v_t\mu_t\}_{t\in I}$, with
\begin{align*}
&v_t(x)=f_0(x)+A(x)\,u(t,x),\quad\textrm{for }\mu_t\textrm{-a.e. }x\textrm{ and a.e. }t\in I,\\
&\int_{\mathbb R^d}\Psi\left(x,\dfrac{\nu_t}{\mu_t}(x)\right)\,d\mu_t(x)=\int_{\mathbb R^d}|u(t,x)|\,d\mu_t(x),\quad\textrm{for a.e. }t\in I.
\end{align*}
We will call $u(t,x)$ the \emph{norm-minimal control density} associated with the admissible trajectory $(\boldsymbol\mu,\boldsymbol\nu)$.
\end{lemma}
\begin{proof}
By assumption, there exists $v_t\in L^1_{\mu_t}$ such that $\boldsymbol{\nu}=\{\nu_t=v_t\mu_t\}_{t\in I}$, $v_t(x)\in F(x)$ for $\mu_t$-a.e. $x\in \mathbb R^d$ and a.e. $t\in I$. Then, by Lemma \ref{lemma:bprop} for $\mu_t$-a.e. $x$ and a.e. $t$ there exists a unique minimum-norm solution $u(t,x)\in U$ for $v_t(x)=f_0(x)+A(x)\,u$. It is defined by $u(t,x):= A(x)^+\,(v_t(x)-f_0(x))+(I-A(x)^+\,A(x))\,w$, for some $w\in\mathbb R^m$ and it satisfies $\Psi(x,v_t(x))=|u(t,x)|$ for $\mu_t$-a.e. $x$ and a.e. $t$. By construction, $u:I\times\mathbb R^d\to U$ is a well-defined Borel map for $\mu_t$-a.e. $x$ and a.e. $t$.
\end{proof}

\medskip

The following result allows us to prove the existence of an admissible trajectory with given (admissible) initial velocity and satisfying further properties which will be used later on in Section \ref{sec:HJB} to provide an HJB result in our framework.

\begin{lemma}\label{Lemma:initialVel1}
Let $\mu_0\in\mathscr P_2(\mathbb R^d)$, $T>0$, $v_0:\mathbb R^d\to\mathbb R^d$ be a Borel map which is also a $L^2_{\mu_0}$-selection of $F(\cdot)$.
Then for all $\beta>0$, there exists an admissible pair $(\boldsymbol\mu,\boldsymbol\nu)\in\mathcal A_{[0,T]}(\mu_0)$ and a representation $\boldsymbol\eta$ for the pair $(\boldsymbol\mu,\boldsymbol\nu)$ in the sense of Definition~\ref{def:representation}(\ref{item1:repres})
satisfying 
\begin{enumerate}
\item for all $p\in L^2_{\mu_0}(\mathbb R^d)$ 
\[\lim_{t\to0^+}\int_{\mathbb R^d\times\Gamma_{[0,T]}} \langle p\circ e_0(x,\gamma) ,\frac{e_t(x,\gamma)-e_0(x,\gamma)}{t}\rangle\,d\boldsymbol\eta(x,\gamma)=\int_{\mathbb R^d}\langle p(x),v_0(x)\rangle\,d\mu_0(x);\]
\item $\nu_t\rightharpoonup^* v_0\mu_0$ as $t\to 0^+$;
\item for all $t\in [0,T]$ we have
\[\int_{\mathbb R^d} \Psi\left(x,\dfrac{\nu_t}{\mu_t}(x)\right)\,d\mu_t(x)\le \int_{\mathbb R^d} \Psi\left(x,v_0(x)\right)\,d\mu_0(x);\]
\item the following bound holds
\[\int_0^T\int_{\mathbb R^d} \Psi\left(x,\dfrac{\nu_t}{\mu_t}(x)\right)\,d\mu_t(x)dt\le \beta.\]
\end{enumerate}
\end{lemma}
\begin{proof}
Let $u_0:\mathbb R^d\to U$ be a Borel map such that 
$v_0(x)=f(x,u_0(x))$ for all $x\in\mathbb R^d$ and $\Psi(x,v_0(x))=|u_0(x)|$. Notice that such a map $u_0$ exists by the same argument used in the proof of Lemma \ref{lemma:MinContr}.
Define the map $G:\mathbb R^d\rightrightarrows\Gamma_T$
\[G(\bar x)=\left\{\gamma\in AC([0,T]):\,\gamma(t)=\bar x+\int_0^t f\left(\gamma(s),u_0(\bar x)e^{-s}\right)\,ds, \,\,\forall t\in[0,T]\right\}.\]
According to Theorem 8.2.9 p.315 in \cite{AuF}, to prove the measurability of this map it is sufficient to notice that
the map $g:\mathbb R^d\times \Gamma_T\to \Gamma_T$ defined by 
\[g(\bar x,\gamma)(t):=\gamma(t)-\bar x+\int_0^t f\left(\gamma(s),u_0(\bar x)e^{-s}\right)\,ds,\]
is a Carath\'eodory map, i.e., $x\mapsto g(x,\gamma)$ is Borel for every $\gamma\in\Gamma_T$ and $\gamma\mapsto g(x,\gamma)$ is continuous for every $x\in\mathbb R^d$.\par\medskip\par
By Theorem 8.1.3 p. 308 in \cite{AuF}, since $G(\cdot)$ is Borel, it admits a Borel selection $x\mapsto \gamma_x\in G(x)$.
Define $\boldsymbol\mu=\{\mu_t\}_{t\in [0,T]}$ by setting $\mu_t=e_t\sharp\boldsymbol\eta$ where
\[\boldsymbol\eta=\mu_0\otimes\delta_{\gamma_x}\in\mathscr P(\mathbb R^d\times \Gamma_T).\]
According to Theorem 1 in \cite{CMPbulgaro}, we have that $(\boldsymbol\mu,\boldsymbol\nu=\{\nu_t\}_{t\in [0,T]})\in\mathcal A_{[0,T]}$,
where $\nu_t$ is defined by
\[\dfrac{\nu_t}{\mu_t}(y)=:v_t(y)=\int_{e_t^{-1}(y)}\dot \gamma(t) \,d\eta_{t,y}(x,\gamma),\]
for a.e. $t\in [0,T]$ and $\mu_t$-a.e. $y\in\mathbb R^d$, where we used the disintegration $\boldsymbol\eta=\mu_t\otimes\eta_{t,y}$.\par\medskip\par

\medskip

We prove (1) following a similar procedure as for the proof of Proposition 2.5 in \cite{MQ}. By Proposition \ref{prop:moments}, we have that $e_0,\frac{e_t-e_0}{t}\in L^2_{\boldsymbol\eta}$ for all $t\in[0,T]$. Thus, for all $p\in L^2_{\mu_0}(\mathbb R^d)$, by the definition of $\boldsymbol\eta$, we have
\begin{equation*}
\lim_{t\to0^+}\int_{\mathbb R^d\times\Gamma_{[0,T]}}\langle p\circ e_0(x,\gamma),\frac{e_t(x,\gamma)-e_0(x,\gamma)}{t}\rangle\,d\boldsymbol\eta(x,\gamma)=\lim_{t\to0^+}\int_{\mathbb R^d}\langle p(x),\frac{\gamma_x(t)-\gamma_x(0)}{t}\rangle\,d\mu_0(x).
\end{equation*}
By Dominated Convergence Theorem, we obtain
\begin{align*}
\lim_{t\to0^+}\int_{\mathbb R^d}\langle p(x),\frac{\gamma_x(t)-\gamma_x(0)}{t}\rangle\,d\mu(x)&=\int_{\mathbb R^d}\langle p(x),\lim_{t\to0^+}\frac{\gamma_x(t)-\gamma_x(0)}{t}\rangle\,d\mu_0(x)\\
&=\int_{\mathbb R^d}\langle p(x),v_0(x)\rangle\,d\mu_0(x),
\end{align*}
thanks to the uniform bound on $\displaystyle\frac{e_t-e_0}{t}$ in terms of the $2$-moment of $\mu_0$ (see Proposition \ref{prop:moments}).

\medskip

We prove (2). By the definition of $\boldsymbol\eta$, we have
\[v_t(y)=\int_{e_t^{-1}(y)}f(y,u_0(\gamma(0))e^{-t})\,d\eta_{t,y}(x,\gamma).\]
For any $\varphi\in C^0_b(\mathbb R^d;\mathbb R^d)$, we then have
\begin{equation}\label{eq:InitVel0}
\int_{\mathbb R^d}\langle \varphi(y),v_t(y)\rangle\,d\mu_t(y)=\int_{\mathbb R^d\times\Gamma_T}\varphi(\gamma(t))\cdot f(\gamma(t),u_0(\gamma(0))e^{-t})\,d\boldsymbol\eta(x,\gamma).
\end{equation}
We observe that we can use the Dominated Convergence Theorem, indeed
\begin{align*}
f(\gamma(t),u_0(\gamma(0))e^{-t})&= (1+|\gamma(t)|)\dfrac{f(\gamma(t),u_0(\gamma(0))e^{-t})}{1+|\gamma(t)|}\\
&\le C (1+|\gamma(t)|)\cdot \max\{1,\mathrm{diam}\, U\},
\end{align*}
which is $\boldsymbol\eta$-integrable since we can estimate the $2$-moment of $\mu_t$ in terms of the $2$-moment of $\mu_0$ by Proposition \ref{prop:moments}.
Thus, by passing to the limit under the integral sign in \eqref{eq:InitVel0}, we obtain
\[\lim_{t\to 0^+}\int_{\mathbb R^d}\langle \varphi(y),v_t(y)\rangle\,d\mu_t(y)=\int_{\mathbb R^d}\langle\varphi(x),f(x,u_0(x))\rangle\,d\mu_0=\int_{\mathbb R^d}\langle\varphi(x),v_0(x)\rangle\,d\mu_0(x).\]

\medskip

We prove (3). Recalling the affine structure of $f$, we have
\[v_t(y)=f\left(y,\int_{e_t^{-1}(y)}u_0(\gamma(0))e^{-t}\,d\eta_{t,y}(x,\gamma)\right),\]
thus
\begin{align*}
\int_{\mathbb R^d}\Psi(y,v_t(y))\,d\mu_t(y)\le&\int_{\mathbb R^d}\int_{e_t^{-1}(y)}|u_0(\gamma(0))|e^{-t}\,d\eta_{t,y}(x,\gamma)\,d\mu_t(x)\\
\le&\int_{\mathbb R^d\times\Gamma_T}|u_0(\gamma(0))|\,d\boldsymbol\eta(x,\gamma)=\int_{\mathbb R^d}|u_0(x)|\,d\mu_0(x)\\
=&\int_{\mathbb R^d}\Psi(x,v_0(x))\,d\mu_0(x).
\end{align*}
The last formula shows that the map
\[t\mapsto \int_{\mathbb R^d}\Psi(y,v_t(y))\,d\mu_t(y)\]
belongs to $L^1([0,T])$. In particular, for every $\beta>0$ there exists $\tau>0$ such that 
\[\int_{0}^{\tau}\int_{\mathbb R^d}\Psi(y,v_t(y))\,d\mu_t(y)\,dt\le \beta.\]
We then consider any solution $\boldsymbol{\tilde\mu}=\{\tilde \mu_t\}_{t\in [\tau,T]}$ of the equation
\[\begin{cases}
\partial_t\tilde\mu_t+\mathrm{div}(f(x,0)\tilde\mu_t)=0,\\
\tilde\mu_\tau=\mu_{\tau}.
\end{cases}\]
By Lemma 4.4 in \cite{DNS}, the juxtaposition of $\boldsymbol\mu$ restricted to $[0,\tau]$ with $\boldsymbol{\tilde\mu}$, and the juxtaposition of the corresponding families of Borel vector-valued measures $\boldsymbol\nu$ restricted to $[0,\tau]$ with $\boldsymbol{\tilde\nu}=\{f(\cdot,0)\tilde\mu_t\}_{t\in[\tau,T]}$, yields an admissible trajectory satisfying (4).
\end{proof}

\section{General dynamic programming principle}\label{sec:GF}

In this section we present an abstract Dynamic Programming Principle which holds in quite general frameworks: this will allow us to treat the optimal control problems proposed in Section \ref{sec:sparse} 
and in Appendix \ref{sec:appendix} in a unified way.
The proposed structure establish a common framework to check the validity of a Dynamic Programming Principle for problems of different nature.  

\begin{definition}\label{def:GFobjects}
A \emph{generalized control system} is a quadruplet $(X,\Sigma,c,c_f)$ where $X$, $\Sigma$ are nonempty sets, 
$c_f:X\to [0,+\infty]$, and $c:X\times X\times \Sigma\to [0,+\infty]$, is a map satisfying the following properties
\begin{enumerate}
\item[$\boldsymbol{(C_1)}$] for every $x,y,z\in X$, $\sigma_1,\sigma_2\in \Sigma$, there exists $\sigma'\in\Sigma$ such that
\[c(x,z,\sigma')\le c(x,y,\sigma_1)+c(y,z,\sigma_2).\]
\item[$\boldsymbol{(C_2)}$] for every $x,z\in X$, $\sigma\in \Sigma$, there exist $y'\in X$, $\sigma'_1,\sigma'_2\in\Sigma$ such that
\[c(x,z,\sigma)\ge c(x,y',\sigma'_1)+c(y',z,\sigma'_2).\]
\end{enumerate}
Given $x\in X$ we define the \emph{reachable set from} $x$ by
\[\mathscr R_x:=\left\{y\in X:\,\inf_{\sigma\in \Sigma}c(x,y,\sigma)<+\infty\right\},\]
and if $y\in\mathscr R_x$ we say that $y$ can be reached from $x$.
Notice that if $y\in\mathscr R_x$ and $z\in\mathscr R_y$, property $\boldsymbol{(C_1)}$ implies that $z\in \mathscr R_x$, hence the position
\[R_{\Sigma}:=\left\{(x,y)\in X\times X:\,y\in\mathscr R_x\right\},\]
defines a transitive relation $R_{\Sigma}$ on $X$.\par\medskip\par
If we define
\[X_\Sigma:=\{x\in X:\,(x,x)\in R_{\Sigma}\},\]
we have that the restriction of $R_{\Sigma}$ on $X_{\Sigma}$ is a partial order on $\Sigma$.
Equivalently, we have that $x\in X_{\Sigma}$ if and only if there exists $\sigma\in \Sigma$ such that $c(x,x,\sigma)<+\infty$, i.e., $x\in\mathscr R_x$.
\par\medskip\par
Given $x\in X$ and $y\in\mathscr R_x$, we define the set of \emph{admissible transitions} from $x$ to $y$ by
\[\mathscr A(x,y):=\left\{\sigma\in \Sigma:\, c(x,y,\sigma)<+\infty\right\}.\]
and if $\sigma\in \mathscr A(x,y)$, we call $c(x,y,\sigma)$ the \emph{cost of the admissible transition} $\sigma$.
We call $c_f(y)$ the \emph{exit cost} at the state $y$.\par\medskip\par
We define the \emph{value function} $V:X\to[0,+\infty]$ by setting
\[V(x)=\inf_{\substack{y\in X\\ \sigma\in\Sigma}}\left\{c(x,y,\sigma)+c_f(y)\right\},\]
and if $V(x)<+\infty$, we have
$V(x)=\displaystyle\inf_{y\in \mathscr R_x}\inf_{\sigma\in\mathscr A(x,y)}\left\{c(x,y,\sigma)+c_f(y)\right\}$.\par
\end{definition}

We prove a Dynamic Programming Principle for this general framework.

\begin{theorem}[Dynamic Programming Principle]\label{thm:DPP}
For every $x\in X$ we have
\[V(x)=\inf_{\substack{y\in X\\ \sigma\in\Sigma}}\left\{c(x,y,\sigma)+V(y)\right\}.\]
\end{theorem}
\begin{proof}
Set $W(x)=\displaystyle\inf_{\substack{y\in X\\ \sigma\in\Sigma}}\left\{c(x,y,\sigma)+V(y)\right\}$.
\begin{enumerate}
\item We prove that $V(x)\ge W(x)$. If $V(x)=+\infty$ there is nothing to prove. So assume $V(x)<+\infty$.
For all $\varepsilon>0$ there exist $y_\varepsilon\in X$, $\sigma_\varepsilon\in \Sigma$ such that
\[V(x)+\varepsilon\ge c(x,y_\varepsilon,\sigma_\varepsilon)+c_f(y_\varepsilon).\]
According to $\boldsymbol{(C_2)}$, there are $y'_\varepsilon\in X$, $\sigma'_{\varepsilon,1},\sigma'_{\varepsilon,2}\in\Sigma$
such that
\begin{align*}
V(x)+\varepsilon\ge&c(x,y_\varepsilon,\sigma_\varepsilon)+c_f(y_\varepsilon)\ge c(x,y'_\varepsilon,\sigma'_{\varepsilon,1})+c(y'_\varepsilon,y_\varepsilon,\sigma'_{\varepsilon,2})+c_f(y_\varepsilon)\\
\ge&c(x,y'_\varepsilon,\sigma'_{\varepsilon,1})+V(y'_\varepsilon)\ge \inf_{\substack{y\in X\\ \sigma\in\Sigma}}\left\{c(x,y,\sigma)+V(y)\right\}=W(x),
\end{align*}
and we conclude by letting $\varepsilon\to 0^+$.
\item We prove that $V(x)\le W(x)$. If $W(x)=+\infty$ there is nothing to prove. So assume $W(x)<+\infty$. For all $\varepsilon>0$
there exist $y_\varepsilon,y'_\varepsilon\in X$, $\sigma_\varepsilon,\sigma'_\varepsilon\in \Sigma$ such that
\[W(x)+\varepsilon\ge c(x,y_\varepsilon,\sigma_\varepsilon)+V(y_\varepsilon)\ge c(x,y_\varepsilon,\sigma_\varepsilon)+c(y_\varepsilon,y'_\varepsilon,\sigma'_\varepsilon)+c_f(y'_\varepsilon)-\varepsilon.\]
According to $\boldsymbol{(C_1)}$, there exists $\sigma_{\varepsilon}''\in\Sigma$ such that
\begin{align*}
W(x)+\varepsilon\ge&c(x,y'_\varepsilon,\sigma''_\varepsilon)+c_f(y'_\varepsilon)-\varepsilon\\ \ge&\inf_{\substack{y\in X\\ \sigma\in\Sigma}}\left\{c(x,y,\sigma)+c_f(y)\right\}-\varepsilon=V(x)-\varepsilon,
\end{align*}
and we conclude by letting $\varepsilon\to 0^+$.
\end{enumerate}
\end{proof}

\begin{definition}[Generalized admissible trajectory]\label{def:adm}
Let $(I,\le_I)$ be a totally orderered set admitting a maximal element $b\in I$ and a minimal element $a\in I$ w.r.t. the order $\le_I$.
We endow $I$ with the order topology, and use the notation $I=[a,b]$.
Given $x\in X$, a \emph{generalized admissible trajectory} starting from $x$ defined on $I$ is a pair $(\gamma,\sigma)$ of maps $\gamma:I\to X$, $\sigma:I\to \Sigma$, 
satisfying 
\begin{enumerate}
\item $\gamma(a)=x$;
\item $c(x,\gamma(t),\sigma(t))<+\infty$ for all $t\in I$; 
\item\label{item3:admGF} for all $t_1,t_2\in I$ with $t_1\le_I t_2$ there exists $\sigma_{t_1\to t_2}\in \Sigma$
such that 
\[c(x,\gamma(t_2),\sigma(t_2))\ge c(x,\gamma(t_1),\sigma(t_1))+c(\gamma(t_1),\gamma(t_2),\sigma_{t_1\to t_2}) \]
\end{enumerate}
In particular, by taking $t=b$ in $(2)$, we must have $\gamma(b)\in\mathscr R_x$. Moreover, from $(3)$ we deduce that $c(\gamma(t_1),\gamma(t_2),\sigma_{t_1\to t_2})<+\infty$, so 
$\gamma(t_2)\in\mathscr R_{\gamma(t_1)}$ for all $t_1,t_2\in I$ with $a\le_I t_1\le_I t_2\le_Ib$.\par
\end{definition}
\begin{remark}\label{rem:zerostand}
We notice that if $(\gamma,\sigma)$ is a generalized admissible trajectory defined on $I$, and $\sigma_{a\to a}\in\Sigma$ satisfies Definition \ref{def:adm} (3) with $t_1=t_2=a$,
we can define $\sigma'(t)=\sigma(t)$ for $t\ne a$ and $\sigma'(a)=\sigma_{a\to a}$.
In this case, we have that $(\gamma,\sigma')$ is still a generalized admissible trajectory, and, from Definition \ref{def:adm} (3), recalling that $\sigma'(a)=\sigma_{a\to a}$, 
we have $c(x,x,\sigma'(a))=0$. Thus, without loss of generality, given a generalized admissible trajectory $(\gamma,\sigma)$ defined on $I=[a,b]$ we always assume that
$c(x,x,\sigma(a))=0$.
\end{remark}

\begin{definition}[Optimal transitions and optimal trajectories]\label{def:GFopt}
Given $x,y\in X$, $\sigma\in\Sigma$, we say that $\sigma$ is an \emph{optimal transition} from $x$ to $y$
if \[V(x)=c(x,y,\sigma)+V(y).\]
A generalized admissible trajectory $(\gamma,\sigma)$ defined on $I=[a,b]$ is called an \emph{optimal trajectory} if for all $t\in I$
we have that $\sigma(t)$ is an optimal transition from $\gamma(a)$ to $\gamma(t)$, i.e.,
\[V(\gamma(a))=c(\gamma(a),\gamma(t),\sigma(t))+V(\gamma(t)), \textrm{ for all }t\in I.\]
\end{definition}

\begin{corollary}[DPP for generalized admissible trajectories]\label{cor:DPPadm}
Let $x\in X$ and $(\gamma,\sigma)$ be a generalized admissible trajectory starting from $x$ defined on the totally ordered set $I$.
Then the map $h:I\to [0,+\infty]$ defined as 
\[h(t):=c(x,\gamma(t),\sigma(t))+V(\gamma(t)),\]
is monotone increasing, and it is constant if and only if the trajectory is optimal.
Moreover, if the trajectory is optimal, for all $t,s\in I$ with $t\le s$ we have that
any $\sigma_{t\to s}\in \Sigma$ satisfying $(3)$ in Definition \ref{def:adm} is an optimal transition from $\gamma(t)$ to $\gamma(s)$. 
\end{corollary}
\begin{proof}
Let $(\gamma,\sigma)$ be a generalized admissible trajectory, we prove that $h(\cdot)$ is monotone increasing: indeed,
recalling Theorem \ref{thm:DPP}, 
\[V(z_1)\le c(z_1,z_2,\sigma)+V(z_2),\textrm{ for all }\sigma\in\Sigma,\,z_1,z_2\in X,\] 
thus, choosing $z_1=\gamma(t)$, $z_2=\gamma(s)$, $\sigma=\sigma_{t\to s}$, from Definition \ref{def:adm} (3), we have
\begin{equation}\label{eq:dpp}
V(\gamma(t))-V(\gamma(s))\le c(\gamma(t),\gamma(s),\sigma_{t\to s})\le c(x,\gamma(s),\sigma(s))-c(x,\gamma(t),\sigma(t)),
\end{equation}
hence
\[h(t)=V(\gamma(t))+c(x,\gamma(t),\sigma(t))\le V(\gamma(s))+c(x,\gamma(s),\sigma(s))=h(s),\]
as desired.
\begin{enumerate}
\item We prove that, if $h(\cdot)$ is constant, then the trajectory is optimal.
If $h$ is constant, we have 
\[h(t)=V(\gamma(t))+c(x,\gamma(t),\sigma(t))=V(\gamma(s))+c(x,\gamma(s),\sigma(s))=h(s),\]
and so
\[V(\gamma(t))-V(\gamma(s))=c(x,\gamma(s),\sigma(s))-c(x,\gamma(t),\sigma(t)).\]
In particular, all the inequalities in \eqref{eq:dpp} are fulfilled as equality,
thus $\sigma_{t\to s}$ is an optimal transition between $\gamma(t)$ and $\gamma(s)$.
By taking $t=a$, and recalling that we can always assume that $c(x,x,\sigma(a))=0$ (see Remark \ref{rem:zerostand}),
we have that $(\gamma,\sigma)$ is optimal.
\item We prove that, if the trajectory is optimal, then $h(\cdot)$ is constant.
Since the trajectory is optimal, we have 
\[V(x)=c(x,\gamma(t),\sigma(t))+V(\gamma(t)),\hspace{1cm}V(x)=c(x,\gamma(s),\sigma(s))+V(\gamma(s)),\]
hence 
\[V(\gamma(t))-V(\gamma(s))=c(x,\gamma(s),\sigma(s))-c(x,\gamma(t),\sigma(t)),\]
thus $h(\cdot)$ is constant, and again all the inequalities in \eqref{eq:dpp} are fulfilled as equality,
so $\sigma_{t\to s}$ is an optimal transition between $\gamma(t)$ and $\gamma(s)$.
\end{enumerate}
\end{proof}

\begin{corollary}\label{cor:DPPadm2}
Let $(\gamma,\sigma)$ be a generalized admissible trajectory defined on the totally ordered set $I=[a,b]$.
\begin{enumerate}
\item\label{item1corGF} if $\displaystyle\inf_{\sigma\in\Sigma}c(\gamma(b),\gamma(b),\sigma)=0$, we have $V(\gamma(b))\le c_f(\gamma(b))$.
\item if $(\gamma,\sigma)$ is optimal and $V(\gamma(b))=c_f(\gamma(b))$ then 
$\gamma(b)\in X$ and $\sigma(b)\in \Sigma$ realize the infimum in the definition of $V(\gamma(a))$, i.e.,
\[V(\gamma(a))=c(\gamma(a),\gamma(b),\sigma(b))+c_f(\gamma(b)).\]
\item\label{item3:corDPP2} if $\displaystyle\inf_{\sigma\in\Sigma}c(\gamma(b),\gamma(b),\sigma)=0$ and $\gamma(b)\in X$ and $\sigma(b)\in \Sigma$ realize the infimum in the definition of $V(\gamma(a))$, i.e.,
\[V(\gamma(a))=c(\gamma(a),\gamma(b),\sigma(b))+c_f(\gamma(b)),\]
then $(\gamma,\sigma)$ is optimal.
\end{enumerate}
\end{corollary}
\begin{proof}
\begin{enumerate}
\item[]
\item By assumption, for all $\varepsilon\ge 0$ there exists $\sigma_\varepsilon\in\Sigma$ with
\[V(\gamma(b))\le c(\gamma(b),\gamma(b),\sigma_\varepsilon)+c_f(\gamma(b))\le \varepsilon+c_f(\gamma(b)),\] 
and we conclude by letting $\varepsilon\to 0^+$ to obtain $V(\gamma(b))\le c_f(\gamma(b))$.
\item Recalling Theorem \ref{thm:DPP}, we have 
\[V(\gamma(a))=c(\gamma(a),\gamma(b),\sigma(b))+V(\gamma(b))=c(\gamma(a),\gamma(b),\sigma(b))+c_f(\gamma(b)).\]
\item Conversely, since $V(\gamma(b))\le c_f(\gamma(b))$ by item \ref{item1corGF}, we have
\[V(\gamma(a))=c(\gamma(a),\gamma(b),\sigma(b))+c_f(\gamma(b))\ge c(\gamma(a),\gamma(b),\sigma(b))+V(\gamma(b)).\]
but, according to the Theorem \ref{thm:DPP}, the opposite inequality holds, and so
\[V(\gamma(a))=c(\gamma(a),\gamma(b),\sigma(b))+V(\gamma(b)).\]
Recalling that $c(\gamma(a),\gamma(a),\sigma(a))=0$, by Corollary \ref{cor:DPPadm} we obtain for all $t\in [a,b]$
\begin{align*}c(\gamma(a),\gamma(a),\sigma(a))+V(\gamma(a))&\le c(\gamma(a),\gamma(t),\sigma(t))+V(\gamma(t))\\ &\le c(\gamma(a),\gamma(b),\sigma(b))+V(\gamma(b)),\end{align*}
and since the first and the last terms are equal, we conclude that for all $t\in [a,b]$
we have $V(\gamma(a))= c(\gamma(a),\gamma(t),\sigma(t))+V(\gamma(t))$,
and so the trajectory is optimal by Corollary \ref{cor:DPPadm}.
\end{enumerate}
\end{proof}

This completes the proof of the Dynamic Programming Principle.

\section{Control sparsity problems}\label{sec:sparse}
In this section we use the notation and the setting introduced in Section \ref{sec:rainbow} to formulate and analyze two problems involving a control-sparsity constraint. 
For both of them we will implement the following strategy:
\begin{itemize}
\item we describe the \emph{control sparsity} constraint that will be included in the cost functional to be minimized;
\item we prove a compactness property of the set of \emph{feasible} trajectories, i.e. admissible trajectories satisfying the control sparsity constraint;
\item we use the results of Section \ref{sec:GF} to prove a Dynamic Programming for the value function of the problem;
\item we prove the existence of an optimal trajectory;
\item we characterize the set of initial velocities for a feasible trajectory.
\end{itemize}
The last step is essential in order to provide necessary conditions in form of an Hamilton-Jacobi-Bellman equation satisfied by the value function (see Section \ref{sec:HJB}).

\subsection{The $L^\infty$-time averaged feasibility case}\label{sec:linftymu}\

Let $\alpha\ge 0$ be fixed, $p\ge 1$. Referring to the notation introduced in Section \ref{sec:GF}, we set
\[X=\mathscr P_p(\mathbb R^d),\quad \Sigma=\displaystyle\bigcup_{\substack{I\subseteq \mathbb R\\ I\textrm{ compact interval}}}\left[AC(I;\mathscr P_p(\mathbb R^d))\times\mathrm{Bor}(I;\mathscr M(\mathbb R^d;\mathbb R^d))\right].\] 
Observe that the set $\mathcal A_I$ of admissible trajectories starting by a measure in $\mathscr P_p(\mathbb R^d)$ (see Definition \ref{def:adm-traj} and Proposition \ref{prop:moments}) is a subset of $\Sigma$.\par
On the set $\Sigma_I:=AC(I;\mathscr P_p(\mathbb R^d))\times\mathrm{Bor}(I;\mathscr M(\mathbb R^d;\mathbb R^d))$ we will consider the topology of sequentially a.e. pointwise $w^*$-convergence, i.e.,
given $\{\rho^n\}_{n\in\mathbb N}:=\{(\boldsymbol\mu^n,\boldsymbol\nu^n)\}_{n}\subseteq\Sigma_I$, and $\rho:=(\boldsymbol\mu,\boldsymbol\nu)\in\Sigma_I$, we say that $\rho^n\rightharpoonup^*\rho$ 
if and only if $(\mu^n_t,\nu^n_t)\rightharpoonup^*(\mu_t,\nu_t)$ for a.e. $t\in I$.

\begin{definition}[$L^\infty$-time feasible trajectories]\label{def:Linftyfeas}
Given $\rho=(\boldsymbol\mu,\boldsymbol\nu)\in \mathcal A_I$, we define
the map $\theta_\rho:I\to [0,+\infty]$ by setting
\begin{equation*}
\theta_\rho(s):=\int_{\mathbb R^d}\Psi\left(x,\dfrac{\nu_s}{\mu_s}(x)\right)\,d\mu_s(x),
\end{equation*}
where $\Psi$ is the \emph{control magnitude density}.
Given $\bar\mu\in\mathscr P_p(\mathbb R^d)$, we set 
\begin{align*}
\mathcal F^\infty_I(\bar\mu):=&\{\rho\in \mathcal A_I(\bar\mu):\,\theta_\rho(s)\le \alpha \textrm{ for a.e. }s\in I\}=\{\rho\in \mathcal A_I(\bar\mu):\,\|\theta_\rho\|_{L^{\infty}(I)}\le \alpha\}\\
\mathcal R_I^\infty(\bar\mu):=&\left\{\boldsymbol\eta\in\mathscr P(\mathbb R^d\times\Gamma_I)\,:\,\exists (\boldsymbol\mu,\boldsymbol\nu)\in\mathcal F^\infty_I(\bar\mu) \textrm{ s.t. } \boldsymbol\eta \textrm{ represents }(\boldsymbol\mu,\boldsymbol\nu)\right\},
\end{align*}
and we define the set of \emph{$\alpha$-feasible trajectories} defined on $I$ by
\[\mathcal F^\infty_I:=\bigcup_{\bar\mu\in\mathscr P_p(\mathbb R^d)}\mathcal F^\infty_I(\bar\mu)\subseteq \Sigma_I.\]
Finally, notice that 
\[\|\theta_\rho\|_{L^{\infty}(I)}\le \alpha\textrm{ if and only if }\int_I \mathcal E(\mu_t,\nu_t)\,dt<+\infty,\]
where  $\mathcal E:\mathscr P(\mathbb R^d)\times\mathscr M(\mathbb R^d;\mathbb R^d)\to [0,+\infty]$ is defined by
\begin{align*}
\mathcal E(\mu,E)&:=\begin{cases}
\displaystyle I_{[0,\alpha]}\left(\int_{\mathbb R^d}\Psi\left(x,\frac{E}{\mu}(x)\right)\,d\mu(x)\right),
&\textrm{if }|E|\ll \mu,\\ \\
+\infty,&\textrm{otherwise}.\end{cases}
\end{align*}
\end{definition}

\begin{remark}
The quantity $\theta_\rho(s)$ represents the total magnitude of control acting on the mass at time $s$. 
Thus, the feasibility constraint imposes a restriction on the amount of control to be used w.r.t. the portion of controlled mass: in particular, at every instant of time the controller
must decide if it is more convenient to control all the mass with a reduced amount of control, or to act on few individuals with a greater amount of control (\emph{control sparsity}). 
In some cases (depending on the cost functional) the two strategies are actually equivalent. We refer to the surveys \cites{CFPT,FPR} for some applications of a sparse control strategy in the framework of multi-agent systems.
\end{remark}

The main topological properties of the set of feasible trajectories are summarized as follows, and are natural extensions of the same properties proved for the admissibility set $\mathcal A_I$, 
respectively in Proposition 3 and Theorem 3 in \cite{CMNP} and in Proposition 1 in \cite{CMPbulgaro}.

\begin{proposition}\label{prop:topprop}
Let $I\subseteq \mathbb R$ be a compact nonempty interval, $p\ge 1$, $\mu_0\in\mathscr P_p(\mathbb R^d)$, $C_1\ge 0$. Then 
\begin{enumerate}
\item $\mathcal F^\infty_I$ is closed w.r.t. the topology of $\Sigma_I$;
\item for any $\mathscr B\subseteq \mathcal F^\infty_I$, $C_1>0$ such that for all $(\boldsymbol\mu,\boldsymbol\nu)\in\mathscr B$ with $\boldsymbol\mu=\{\mu_t\}_{t\in I}$ it holds
$\mathrm m_p(\mu_0)\le C_1$, we have that the closure of $\mathscr B$ in $\Sigma_I$ is contained in $\mathcal F^\infty_I$;
\item $\mathcal F^\infty_I(\mu_0)$ is compact in the topology of $\Sigma_I$
\item $\mathcal R^\infty_I(\mu_0)$ is compact in the narrow topology.
\end{enumerate}
\end{proposition}
\begin{proof}
The proof is essentially based on the variational characterization of the feasibility constraint.\par\medskip\par
\emph{Step 1: }The functional $\mathscr F:\mathscr P(\mathbb R^d)\times\mathscr M(\mathbb R^d;\mathbb R^d)\to[0,+\infty]$ defined by
\begin{equation}\label{eq:feasEreduced}
\mathscr F(\mu,E):=\begin{cases}
\displaystyle\int_{\mathbb R^d}\Psi\left(x,\frac{E}{\mu}(x)\right)\,d\mu(x),
&\textrm{if }|E|\ll \mu,\\ \\
+\infty,&\textrm{otherwise}\end{cases}
\end{equation}
is l.s.c. w.r.t. w$^*$-convergence.\par\medskip\par
\emph{Proof of Step 1: }By Lemma~\ref{lemma:bprop}, the function $\Psi:\mathbb R^d\times\mathbb R^d\to [0,+\infty]$ is l.s.c. and  $\Psi(x,\cdot)$ is convex for any $x\in\mathbb R^d$ with bounded domain.
So adopting the notation in \cite{But}, we have that $\Psi_\infty(x,v)=0$ if $v=0$ and $\Psi_\infty(x,v)=+\infty$ if $v\neq 0$, where $\Psi_\infty(x,\cdot)$ 
denotes the \emph{recession function} for $\Psi(x,\cdot)$. By l.s.c. of $F$, there exists a continuous selection $z:\mathbb R^d\to\mathbb R^d$ of $F$ (Michael's Theorem). 
Thus, by continuity of $\Psi(\cdot,\cdot)$ in $\mathrm{Graph}\,F$ (see Lemma \ref{lemma:bprop}), we have that $x\mapsto\Psi(x,z(x))$ is continuous and finite. 
We conclude by Lemma 2.2.3, p. 39, Theorem 3.4.1, p.115, and Corollary 3.4.2 in \cite{But} or Theorem 2.34 in \cite{AFP}.\hfill$\diamond$\par\medskip\par
\emph{Step 2: }Let $\rho^n:=(\boldsymbol\mu^n,\boldsymbol\nu^n)\in\mathcal F^\infty_I$ for all $n\in\mathbb N$, $\rho:=(\boldsymbol\mu,\boldsymbol\nu)\in\Sigma_I$
be such that $\rho^n$ converges to $\rho$ in $\Sigma_I$. Then $\rho\in\mathcal F^\infty_I$.\par\medskip\par
\emph{Proof of Step 2: }By convexity and l.s.c. of the indicator function $I_{[0,\alpha]}(\cdot)$ and l.s.c. of $\mathscr F(\cdot,\cdot)$, we have that the functional $\mathcal E(\cdot,\cdot)$ is l.s.c w.r.t. w$^*$-convergence. 
The l.s.c. of the functional $\mathcal J_F(\cdot,\cdot)$, defined in Definition~\ref{def:adm-traj}$(A2)$, was already proved in Lemma 3 in \cite{CMNP}.
By Proposition 3 in \cite{CMNP} we have $\rho\in\mathcal A_I$.
Now, fix $t\in I$ such that $(\mu^n_t,\nu^n_t)\rightharpoonup^*(\mu_t,\nu_t)$ and $\mathcal E(\mu^n_t,\nu^n_t)+\mathcal J_F(\mu^n_t,\nu^n_t)=0$ for all $n\in\mathbb N$. By l.s.c. of $\mathcal E$, $\mathcal J_F$ and the fact that $\mathcal E,\mathcal J_F\ge 0$, we have  
\[\displaystyle 0\le \mathcal E(\mu_t,\nu_t)+\mathcal J_F(\mu_t,\nu_t)\le \liminf_{n\to +\infty} \mathcal E(\mu_t^n,\nu_t^n)+\mathcal J_F(\mu_t^n,\nu_t^n)=0.\]
By applying Fatou's lemma we deduce that \[\int_I\left(\mathcal E(\mu_t,\nu_t)+\mathcal J_F(\mu_t,\nu_t)\right)\,dt=0,\] hence $\rho\in\mathcal F^\infty_I$.\hfill$\diamond$\par\medskip\par
All the assertion now follows recalling that uniformly boundedness of the moments along a sequence implies existence of a narrowly convergent subsequence (see e.g. Chapter 5 in \cite{AGS}).
\end{proof}

We pass now to the analysis of the value function and of the Dynamic Programming Principle in this setting.
To this aim, we will refer to the abstract results proved in Section \ref{sec:GF}.

\begin{definition}[Concatenation and restriction]
\begin{enumerate}
\item[]
\item Let $I_i=[a_i,b_i]\subset\mathbb R$, $i=1,2$, with $b_1=a_2$, and $I:=I_1\cup I_2$. Let $\rho^i=(\boldsymbol\mu^i,\boldsymbol\nu^i)\in\mathcal F^\infty_{I_i}$ with $\mu^1_{b_1}=\mu^2_{a_2}$. The \emph{concatenation} $\rho^1\star\rho^2=(\boldsymbol\mu,\boldsymbol\nu)$ of $\rho^1$ and $\rho^2$ is defined by setting $\mu_t=\mu_t^i$ and $\nu_t=\nu_t^i$ when $t\in I_i$ for $i=1,2$, and noticing that we can assume $\nu^1_{b_1}=\nu^2_{a_2}$ by changing the driving vector field in a $\mathscr L^1$-negligible set. We recall that this implies that $\rho^1\star\rho^2\in\mathcal F^\infty_I$. Indeed, by Lemma 4.4 in \cite{DNS} we have that the set of solutions of the continuity equation is closed w.r.t. juxtaposition operations. The admissibility property of the resulting trajectory follows straightforwardly, as observed also in Theorem 6 in \cite{CMNP} and so does the feasibility one.
\item Let $\rho=(\boldsymbol\mu,\boldsymbol\nu)\in\mathcal F^\infty_I$. The \emph{restriction} $\rho_{|I_1}$ of $\rho$ to a compact and nonempty interval $I_1\subset I$, where $\rho_{|I_1}=(\boldsymbol\mu^1=\{\mu^1_t\}_{t\in I_1}, \boldsymbol\nu^1=\{\nu^1_t\}_{t\in I_1})$ is defined by setting $\mu^1_t:=\mu_t$ and $\nu^1_t:=\nu_t$ for all $t\in I_1$. Clearly we have $\rho_{|I_1}\in\mathcal F^\infty_{I_1}$. 
\end{enumerate}
\end{definition}

\begin{remark}\label{rem:extFeas}
We can always extend an $\alpha$-feasible trajectory $(\boldsymbol\mu^{(1)},\boldsymbol\nu^{(1)})\in\mathcal F^\infty_{[a,b]}$ to an $\alpha$-feasible
trajectory defined in the extended time-interval $[a,c]$, for any $0\le a\le b\le c$. It is sufficient to take $\bar v_t(\cdot)=f(\cdot,0)$ for
all $t\in[b,c]$ and to consider the solution $\boldsymbol\mu^{(2)}=\{\mu^{(2)}_t\}_{t\in[b,c]}$ of the continuity equation
$\partial_t\mu_t+\mathrm{div}(\nu_t^{(2)})=0$ for $t\in[b,c]$, with $\mu_{|t=b}=\mu^{(1)}_{|t=b}$, $\nu_t^{(2)}=\bar v_t\mu_t$.
We have that $(\boldsymbol\mu,\boldsymbol\nu):=(\boldsymbol\mu^{(1)},\boldsymbol\nu^{(1)})\star(\boldsymbol\mu^{(2)},\boldsymbol\nu^{(2)})$ is an $\alpha$-feasible trajectory on $[a,c]$.
\end{remark}

\begin{definition}\label{def:LinftyGF}
Let $c$ and $c_f$ be as in Definition \ref{def:GFobjects}, satisfying the following additional properties
\begin{enumerate}
\item[$\boldsymbol{(C_3)}$] $c(\mu^{(1)},\mu^{(2)}, \hat\rho)<+\infty$ if and only if $\hat\rho=(\boldsymbol{\hat\mu},\boldsymbol{\hat\nu})\in\mathcal F^\infty_I(\mu^{(1)})$, with $\hat\mu_{|t=\max I}=\mu^{(2)}$ for some compact and nonempty interval $I\subset\mathbb R$;
\item[$\boldsymbol{(C_4)}$] let $0\le a\le b\le c$, $\rho=(\boldsymbol{\mu},\boldsymbol{\nu})\in\mathcal F^\infty_{[a,c]}$. Then $c:X\times X\times \Sigma\to[0,+\infty]$ is superadditive by restrictions, i.e.
\[c(\mu_a,\mu_c,\rho)\ge c(\mu_a,\mu_b,\rho_{|[a,b]})+c(\mu_b,\mu_c,\rho_{|[b,c]}).\]
\end{enumerate}
Let $\bar\mu\in\mathscr P_p(\mathbb R^d)$, $I\subset\mathbb R$ nonempty and compact interval, and $\rho=(\boldsymbol\mu,\boldsymbol\nu)\in\mathcal F^\infty_I(\bar\mu)$. We define the set $\mathscr G^\infty_{\rho}$ made of the pairs $(\gamma,\sigma)$ defined as follows
\begin{enumerate}
\item $\gamma:I\to X$, $\gamma(t):=\mu_t$ for all $t\in I$;
\item $\sigma:I\to\Sigma$, $\sigma(t):=\rho_{|[\min I,t]}$ for all $t\in I$.
\end{enumerate}
Finally, we define the set
\[\mathscr G_I^\infty(\bar\mu):=\left\{(\gamma,\sigma)\in\mathscr G_{\rho}^\infty\,:\,\rho\in\mathcal F^\infty_I(\bar\mu)\right\}.\]
\end{definition}

\begin{theorem}[DPP for $L^\infty$-time feasibility case]\label{thm:DPPinfty}
Let $V:\mathscr P_p(\mathbb R^d)\to [0,+\infty]$ be as in Definition \ref{def:GFobjects}. For any $\mu_0\in\mathscr P_p(\mathbb R^d)$ we have
\[V(\mu_0)=\inf_{\substack{\rho=(\boldsymbol\mu,\boldsymbol\nu)\in\mathcal F^\infty_I(\mu_0)\\ I\subseteq \mathbb R \textrm{ compact interval}}}\left\{c(\mu_0,\mu_{|t=\max I},\rho)+V(\mu_{|t=\max I})\right\}.\]
\end{theorem}
\begin{proof}
The proof follows by Theorem \ref{thm:DPP} and $\boldsymbol{(C_3)}$.
\end{proof}

\begin{remark}\label{rem:LinftyGF}
Any $(\gamma,\sigma)\in\mathscr G_{\rho}^\infty$, s.t. $\rho\in\mathcal F^\infty_I(\bar\mu)$, is generalized admissible from $\bar\mu$, according to Definition \ref{def:adm}. 
It is sufficient to take $\sigma_{t_1\to t_2}:=\rho_{|[t_1,t_2]}$ for any $0\le t_1\le t_2\le T$, and observe that $\boldsymbol{(C_4)}$ implies item $(3)$ in Definition \ref{def:adm}. 
Thus, Corollaries \ref{cor:DPPadm}, \ref{cor:DPPadm2} hold in this framework. 
Furthermore, since we can indentify any $\hat\rho\in\mathcal F^\infty_{[t,t]}$ with its restriction $\hat\rho_{|[t,t]}$, then by $\boldsymbol{(C_4)}$ we have that $c(\gamma(t),\gamma(t),\sigma_{t\to t})=0$ for all $t\in I$.
\end{remark}

\begin{proposition}[Existence of minimizers]\label{prop:existenceinfty}
Assume properties $\boldsymbol{(C_1)}-\boldsymbol{(C_4)}$. Let $p\ge 1$ and $\mu_0\in\mathscr P_p(\mathbb R^d)$. If $c(\mu_0,\cdot,\cdot)$ and $c_f$ are l.s.c. w.r.t. $w^*$-convergence and $V(\mu_0)<+\infty$, then there exist $I\subset \mathbb R$ nonempty and compact interval and an optimal trajectory $(\gamma,\sigma)\in\mathscr G^\infty_I(\mu_0)$, according to Definition \ref{def:GFopt}.
\end{proposition}
\begin{proof}
By finiteness of $V(\mu_0)$ and $\boldsymbol{(C_3)}$, we have that for all $n\in\mathbb N$, $n>0$, there exist $(\mu^n,\rho^n)\in X\times\Sigma$ and $I^n\subset \mathbb R$ non empty compact interval such that $\rho^n=(\boldsymbol\mu^n,\boldsymbol\nu^n)\in\mathcal F^\infty_{I^n}(\mu_0)$ with $\mu^n_{|t=\max I^n}=\mu^n$ and
\[V(\mu_0)+\frac{1}{n}\ge c(\mu_0,\mu^n,\rho^n)+c_f(\mu^n).\]

Let $(\gamma^n,\sigma^n)\in\mathscr G^\infty_{\rho^n}$.
Without loss of generality, we can assume $I^n=[0,T_n]$ for all $n\in\mathbb N$. Let $T=\liminf_{n\to+\infty}T_n$, then there exists a subsequence such that $T=\lim_{k\to+\infty}T_{n_k}$ and $T_{n_k}\ge T-\dfrac{1}{k}$ for all $k\ge1$.

Let us consider the restrictions $\rho^{n_k}_{|[0,T-\frac{1}{k}]}\in\mathcal F^\infty_{[0,T-\frac{1}{k}]}(\mu_0)$ and any of their extensions in $[T-\frac{1}{k},T]$ preserving the feasibility constraint (see Remark \ref{rem:extFeas}). 
Denote with $\hat\rho^{nk}:=(\boldsymbol{\hat\mu}^{nk},\boldsymbol{\hat\nu}^{nk})\in \mathcal F^\infty_{[0,T]}(\mu_0)$ such an object. By compactness of $\mathcal F^\infty_{[0,T]}(\mu_0)$ proved 
in Proposition \ref{prop:topprop}, $\hat\rho^{nk}\rightharpoonup^*\hat\rho:=(\boldsymbol{\hat\mu},\boldsymbol{\hat\nu})\in\mathcal F^\infty_{[0,T]}(\mu_0)$, i.e. $(\hat\mu^{nk}_t,\hat\nu^{nk}_t)\rightharpoonup^*(\hat\mu_t,\hat\nu_t)$ for a.e. $t\in[0,T]$. 
Furthermore, similarly to the proof of Theorem 3 in \cite{CMNP}, by the standing assumption in Definition \ref{def:setting} and Remark \ref{rem:basic estim}, for any $k$ there exists a sequence $t_i\to T$ such that $\hat\nu^{nk}_{t_i}\rightharpoonup^*\hat\nu^{nk}_T$ for $i\to+\infty$, and $\hat\mu^{nk}_{t_i}\rightharpoonup^*\hat\mu^{nk}_T$ by absolute continuity of $\boldsymbol{\hat\mu}^{nk}$.

By a diagonal argument, we have that $\hat\rho^{nk}_{|[0,T-\frac{1}{k}]}\rightharpoonup^*\hat\rho$, up to subsequences. Property \eqref{item3:admGF} in Definition \ref{def:adm} leads to
\[c(\mu_0,\mu^{n_k},\rho^{n_k})\ge c(\mu_0,\hat\mu^{nk}_{|t=T-\frac{1}{k}},\hat\rho^{nk}_{|[0,T-\frac{1}{k}]}).\]
By passing to the limit up to subsequences we get
\[V(\mu_0)=c(\mu_0,\hat\mu,\hat\rho)+c_f(\hat\mu).\]
Thus, $(\gamma,\sigma)\in\mathscr G^\infty_{\hat\rho}$ is optimal by Corollary \ref{cor:DPPadm2}\eqref{item3:corDPP2}.
\end{proof}

In Section \ref{sec:examples} we see a concrete example of cost and value functions satisfying $\boldsymbol{(C_1)}-\boldsymbol{(C_4)}$.

\medskip

The last part of this subsection is devoted to the characterization of the set of initial velocities for feasible trajectories in $\mathscr P_2(\mathbb R^d)$. 
This is a fundamental ingredient to deal with the formulation of an Hamilton-Jacobi-Bellman equation for this problem, which is discussed in Section \ref{sec:HJBinfty}.

\begin{definition}\label{def:ivsetLinfty}
Let $\mu\in\mathscr P_2(\mathbb R^d)$, $0\le s\le T$.
\begin{enumerate}
\item Given $\boldsymbol\eta\in\mathcal R^\infty_{[s,T]}(\mu)$, we define
\begin{equation*}
\mathscr V_{[s,T]}(\boldsymbol\eta):=\left\{w_{\boldsymbol\eta}\in L^2_{\boldsymbol\eta}(\mathbb R^d\times\Gamma_{[s,T]};\mathbb R^d)\,:\,
\begin{array}{l}\exists\{t_i\}_{i\in\mathbb N}\subseteq]s,T[, \textrm{ with }t_i\to s^+ \textrm{ and}\\ \\\displaystyle\frac{e_{t_i}-e_s}{t_i-s}\rightharpoonup w_{\boldsymbol\eta} \textrm{ weakly in }L^2_{\boldsymbol\eta}(\mathbb R^d\times\Gamma_{[s,T]};\mathbb R^d)\end{array}\right\},
\end{equation*}
and $\mathscr V_{[s,T]}(\boldsymbol\eta)\neq\emptyset$ as already observed in Definition 11 in \cite{CMNP}.
\item We set
\begin{equation*}
\mathscr V^\infty_{[s,T]}(\mu):=\left\{x\mapsto\int_{\Gamma^x_{[s,T]}}w_{\boldsymbol\eta}(x,\gamma)\,d\eta_x\,:\,\boldsymbol\eta\in\mathcal R^\infty_{[s,T]}(\mu),\,
w_{\boldsymbol\eta}\in \mathscr V_{[s,T]}(\boldsymbol\eta)\right\},
\end{equation*}
where we denoted with $\{\eta_x\}_{x\in\mathbb R^d}$ the disintegration of $\boldsymbol\eta$ w.r.t. the map $e_s$.
\item We define the set
\begin{equation*}
\mathscr Z^\infty(\mu):=\left\{v\in L^2_\mu(\mathbb R^d;\mathbb R^d)\,:\,\begin{array}{l}v(x)\in F(x) \textrm{ for }\mu\textrm{-a.e. }x\in\mathbb R^d \\\textrm{and }\displaystyle\int_{\mathbb R^d}\Psi(x,v(x))\,d\mu(x)\le\alpha\end{array}\right\}.
\end{equation*}
\end{enumerate}
\end{definition}

\begin{lemma}[Initial velocity]\label{lem:carattVelinfty}
Let $\mu\in\mathscr P_2(\mathbb R^d)$, $0\le s\le T$. Then $\mathscr V^\infty_{[s,T]}(\mu)\equiv\mathscr Z^\infty(\mu)$.
\end{lemma}
\begin{proof}
From items $(1)$ and $(3)$ in Lemma \ref{Lemma:initialVel1}, we deduce that $\mathscr Z^\infty(\mu)\subseteq\mathscr V^\infty_{[s,T]}(\mu)$. 

\medskip

Let us now prove the other inclusion, i.e. $\mathscr V^\infty_{[s,T]}(\mu)\subseteq\mathscr Z^\infty(\mu)$. Without loss of generality, let us consider $s=0$. Let $\boldsymbol\eta\in\mathcal R^\infty_{[0,T]}(\mu)$. For any $(\boldsymbol\mu,\boldsymbol\nu)$ represented by $\boldsymbol\eta$, we have in particular that $(\boldsymbol\mu,\boldsymbol\nu)\in\mathcal A_{[0,T]}$, then $x\mapsto\int_{\Gamma^x_{[0,T]}}w_{\boldsymbol\eta}(x,\gamma)\,d\eta_x$, with $w_{\boldsymbol\eta}\in\mathscr V_{[0,T]}(\boldsymbol\eta)$, is an $L^2_\mu$-selection of $F$ by Lemma 5 in \cite{CMNP} and convexity of $F(x)$. Hence, it remains to prove that
\[\int_{\mathbb R^d}\Psi\left(x,\int_{\Gamma^x_{[0,T]}}w_{\boldsymbol\eta}(x,\gamma)\,d\eta_x\right)\,d\mu(x)\le\alpha.\] 
By feasibility of $(\boldsymbol\mu,\boldsymbol\nu)$, for a.e. $t\in[0,T]$ we have
\[\int_{\mathbb R^d\times\Gamma_{[0,T]}}\Psi(\gamma(t),\dot\gamma(t))\,d\boldsymbol\eta(x,\gamma)=\int_{\mathbb R^d}\Psi\left(x,\frac{\nu_t}{\mu_t}(x)\right)\,d\mu_t(x)\le\alpha.\]

By hypothesis and Lemma \ref{lemma:bprop}, $\Psi(\cdot,\cdot)$ is uniformly continuous on $\mathrm{Graph}\,F$ when the first variable ranges in a compact set $K\subset\mathbb R^d$.

Then, let us fix $n\in\mathbb N$ and consider the closed ball $\overline{B_n(0)}=\overline{B(0,n)}\subset\mathbb R^d$. We have
\begin{align*}
\int_{\overline{B_n(0)}}\Psi\left(x,\int_{\Gamma^x_{[0,T]}}w_{\boldsymbol\eta}(x,\gamma)\,d\eta_x\right)&\,d\mu(x)\le\\
&\le\liminf_{t\to0^+}\int_{\overline{B_n(0)}}\Psi\left(x,\int_{\Gamma^x_{[0,T]}}\frac{e_t-e_0}{t}\,d\eta_x\right)\,d\mu(x)\\
&\le\liminf_{t\to0^+}\int_{\overline{B_n(0)}\times\Gamma_{[0,T]}}\Psi\left(e_0,\frac{e_t-e_0}{t}\right)\,d\boldsymbol\eta(x,\gamma)\\
&=\liminf_{t\to0^+}\int_{\overline{B_n(0)}\times\Gamma_{[0,T]}}\Psi\left(\gamma(0),\frac{1}{t}\int_0^t\dot\gamma(r)\,dr\right)\,d\boldsymbol\eta,
\end{align*}
where we used l.s.c. of $\Psi$ in the second variable and Fatou's Lemma for the first inequality, and Jensen's inequality together with convexity of $\Psi$ in the second variable for the second inequality.
Moreover,
\begin{align*}
\liminf_{t\to0^+}\int_{\overline{B_n(0)}\times\Gamma_{[0,T]}}&\Psi\left(\gamma(0),\frac{1}{t}\int_0^t\dot\gamma(r)\,dr\right)\,d\boldsymbol\eta\le\\
&\le\liminf_{t\to0^+}\frac{1}{t}\int_0^t\int_{\overline{B_n(0)}\times\Gamma_{[0,T]}}\Psi(\gamma(0),\dot\gamma(r))\,d\boldsymbol\eta\,dr\\
&\le\liminf_{t\to0^+}\frac{1}{t}\int_0^t\int_{\overline{B_n(0)}\times\Gamma_{[0,T]}}\Psi(\gamma(r),\dot\gamma(r))\,d\boldsymbol\eta\,dr+\varepsilon\\
&\le\alpha+\varepsilon,
\end{align*}
for any $\varepsilon$ small enough. Where we used again Jensen's inequality and convexity of $\Psi$ in the second variable for the first inequality. Finally we used uniform continuity of $\Psi(\cdot,\cdot)$ in the compact set $\overline{B_n(0)}$, together with uniform continuity of the set of all $\gamma\in\Gamma_{[0,T]}$ s.t. $(\gamma(0),\gamma)\in\mathrm{supp}\,\boldsymbol\eta$. Indeed, by the standing assumption in Definition \ref{def:setting} and compactness of $U$, for all $(x,\gamma)\in\mathrm{supp}\,\boldsymbol\eta$, we have
\begin{align*}
|\gamma(t)-\gamma(0)|&\le\int_0^t|\dot\gamma(s)|\,ds\le C\int_0^t|\gamma(s)|\,ds+Ct\\
&\le C\int_0^t|\gamma(s)-\gamma(0)|\,ds+Ct\,(1+|\gamma(0)|),
\end{align*}
and so, by Gronwall's inequality, and recalling that $\gamma(0)\in\overline{B(0,n)}$,
\begin{equation*}
|\gamma(t)-\gamma(0)|\le Ct\,(1+|\gamma(0)|) e^{Ct}\le CT(1+n)e^{CT}.
\end{equation*}

We conclude by letting $\varepsilon\to0^+$ in the former estimate, and noticing that this holds for all $n\in\mathbb N$, thus by passing to the limit we have
\[\int_{\mathbb R^d}\Psi\left(x,\int_{\Gamma^x_{[0,T]}}w_{\boldsymbol\eta}(x,\gamma)\,d\eta_x\right)\,d\mu(x)=\sup_{n\in\mathbb N}\int_{\overline{B_n(0)}}\Psi\left(x,\int_{\Gamma^x_{[0,T]}}w_{\boldsymbol\eta}(x,\gamma)\,d\eta_x\right)\,d\mu(x)\le\alpha.\]
\end{proof}

\subsection{The $L^1$-time averaged feasibility case}\label{sec:l1mu}\

Let $\alpha\ge 0$ be fixed, $p\ge 1$. Referring to the notation of Section \ref{sec:GF}, we take
\[X=\mathscr P_p(\mathbb R^d)\times[0,+\infty[,\quad \Sigma=\displaystyle\bigcup_{\substack{I\subseteq \mathbb R\\ I\textrm{ compact interval}}}\left[AC(I;\mathscr P_p(\mathbb R^d))\times \mathrm{Bor}(I;\mathscr M(\mathbb R^d;\mathbb R^d))\right]\times[0,+\infty[.\] 
On the set $\Sigma_I:=AC(I;\mathscr P_p(\mathbb R^d))\times \mathrm{Bor}(I;\mathscr M(\mathbb R^d;\mathbb R^d))\times[0,+\infty[$ we will consider the topology given
by the sequentially a.e. $w^*$-convergence and the convergence in $\mathbb R$. More precisely, given $\{\rho^n\}_{n\in\mathbb N}:=\{(\boldsymbol\mu^n,\boldsymbol\nu^n,\omega^n)\}_{n\in\mathbb N}\subseteq\Sigma_I$, 
and  $\rho:=(\boldsymbol\mu,\boldsymbol\nu,\omega)\in\Sigma_I$, we say that \emph{$\rho^n$ converges in $w^*/\mathbb R$ to $\rho$}, and we write $\rho^n\rightharpoonup^{*/\mathbb R}\rho$, if and only if 
$(\mu^n_t,\nu^n_t)\rightharpoonup^*(\mu_t,\nu_t)$ for a.e. $t\in I$, and $\omega^n\to\omega$ in $\mathbb R$.

\begin{definition}[$L^1$-time feasible trajectories]
Given $\rho=(\boldsymbol\mu,\boldsymbol\nu,\omega)\in \mathcal A_I\times[0,+\infty[$, we define
the map $\omega_{\rho}:I\to [0,+\infty]$ by setting
\begin{equation*}
\omega_{\rho}(t):=
\begin{cases}
\displaystyle\omega+\int_{\min I}^t\theta_\rho(s)\,ds,&\textrm{ if }\theta_\rho\in L^1([\min I, t]);\\ \\ 
+\infty,&\textrm{ otherwise},
\end{cases}
\end{equation*}
where $\theta_\rho:I\to [0,+\infty]$ is defined as in Definition \ref{def:Linftyfeas}.\par\medskip\par
Given $(\bar\mu,\bar\omega)\in\mathscr P_p(\mathbb R^d)\times[0,+\infty[$, we set 
\begin{align*}
\mathcal F^1_I(\bar\mu,\bar\omega):=&\{\rho\in \mathcal A_I(\bar\mu)\times[0,+\infty[:\,\omega_\rho(\max I)\le \alpha,\,\omega_\rho(\min I)=\bar\omega\}\\
\mathcal R_I^1(\bar\mu,\bar\omega):=&\left\{\boldsymbol\eta\in\mathscr P(\mathbb R^d\times\Gamma_I)\,:\,\exists (\boldsymbol\mu,\boldsymbol\nu,\omega)\in\mathcal F^1_I(\bar\mu,\bar\omega) \textrm{ s.t. } \boldsymbol\eta \textrm{ represents }(\boldsymbol\mu,\boldsymbol\nu)\right\}.
\end{align*}
and we define the set of \emph{$\alpha$-feasible trajectories} defined on $I$ by
\[\mathcal F^1_I:=\bigcup_{\substack{\bar\mu\in\mathscr P_p(\mathbb R^d)\\ \bar\omega\ge 0}}\mathcal F^1_I(\bar\mu,\bar\omega)\subseteq \Sigma_I.\]
\end{definition}

The counterpart of Proposition \ref{prop:topprop} is the following.

\begin{proposition}\label{prop:topprop2}
Let $I\subseteq \mathbb R$ be a compact nonempty interval, $p\ge 1$, $\mu_0\in\mathscr P_p(\mathbb R^d)$, $C_1\ge 0$, $\omega\in[0,\alpha]$. Then 
\begin{enumerate}
\item $\mathcal F^1_I$ is closed w.r.t. the topology of $\Sigma_I$;
\item for any $\mathscr B\subseteq \mathcal F^\infty_I$, $C_1>0$ such that for all $(\boldsymbol\mu,\boldsymbol\nu)\in\mathscr B$ with $\boldsymbol\mu=\{\mu_t\}_{t\in I}$ it holds
$\mathrm m_p(\mu_0)\le C_1$, we have that the closure of $\mathscr B$ in $\Sigma_I$ is contained in $\mathcal F^\infty_I$;
\item $\mathcal F^1_I(\mu_0,\omega)$ is compact in the topology of $\Sigma_I$
\item $\mathcal R^1_I(\mu_0,\omega)$ is compact in the narrow topology.
\end{enumerate}
\end{proposition}
\begin{proof}
It is enough to notice that, given $\rho:=(\boldsymbol\mu,\boldsymbol\nu,\omega)\in\Sigma_I$, we have $\omega_\rho(\max I)\le \alpha$
if and only if 
$\mathcal E'(\rho)<+\infty$,
where  $\mathcal E':AC(I,\mathscr P(\mathbb R^d))\times \mathrm{Bor}(I,\mathscr M(\mathbb R^d;\mathbb R^d))\times [0,+\infty[\to [0,+\infty]$ is defined by
{\footnotesize
\begin{align*}
\mathcal E'(\boldsymbol\mu,\boldsymbol\nu,\omega)&:=\begin{cases}
\displaystyle I_{[0,\alpha]}\left(\omega+\int_I\int_{\mathbb R^d}\Psi\left(x,\frac{\nu_s}{\mu_s}(x)\right)\,d\mu_s(x)\,ds\right),
&\textrm{if }|\nu_s|\ll \mu_s \textrm{ for a.e. }s\in I, \\&\textrm{and }\theta_\rho\in L^1(I)\\ \\
+\infty,&\textrm{otherwise}.\end{cases}
\end{align*}}
Moreover, by applying Fatou's lemma and recalling the l.s.c. of the functional $\mathscr F$ in the proof of Proposition \ref{prop:topprop}, we have that
the functional $\mathscr G:AC(I,\mathscr P(\mathbb R^d))\times \mathrm{Bor}(I,\mathscr M(\mathbb R^d;\mathbb R^d))\times [0,+\infty[\to[0,+\infty]$ defined by
{\footnotesize
\begin{align*}
\mathscr G(\boldsymbol\mu,\boldsymbol\nu,\omega)&:=\begin{cases}
\displaystyle\omega+\int_I\int_{\mathbb R^d}\Psi\left(x,\frac{\nu_s}{\mu_s}(x)\right)\,d\mu_s(x)\,ds,
&\textrm{if }|\nu_s|\ll \mu_s \textrm{ for a.e. }s\in I, \\&\textrm{and }\theta_\rho\in L^1(I)\\ \\
+\infty,&\textrm{otherwise}.\end{cases}
\end{align*}}
is l.s.c. w.r.t. $w^*/\mathbb R$-convergence. Thus the functional $\mathcal E'$ is l.s.c. The other assertions follows by an easy adaption of the proof of Proposition \ref{prop:topprop}.
\end{proof}

With the following definition, we notice that Remark \ref{rem:extFeas} can be easily applied also in this setting.
\begin{definition}[Concatenation and restriction]
\begin{enumerate}
\item[]
\item Let $I_i=[a_i,b_i]\subset\mathbb R$, $i=1,2$, with $b_1=a_2$, and $I:=I_1\cup I_2$. Let $\rho^i=(\boldsymbol\mu^i,\boldsymbol\nu^i,\omega^i)\in\mathcal F^1_{I_i}$ with $\mu^1_{b_1}=\mu^2_{a_2}$, $\omega_{\rho^1}(b_1)=\omega^2$. The \emph{concatenation} $\rho^1\star\rho^2=(\boldsymbol\mu,\boldsymbol\nu,\omega^1)$ of $\rho^1$ and $\rho^2$ is defined by setting $\mu_t=\mu_t^i$ and $\nu_t=\nu_t^i$ when $t\in I_i$ for $i=1,2$. We recall that this implies that $\rho^1\star\rho^2\in\mathcal F^1_I$, with a similar reasoning as for the $L^\infty$-time feasibility condition setting.
\item Let $\rho=(\boldsymbol\mu,\boldsymbol\nu,\omega)\in\mathcal F^1_I$. The \emph{restriction} $\rho_{|I_1}$ of $\rho$ to a compact and nonempty interval $I_1\subseteq I$, where $\rho_{|I_1}=(\boldsymbol\mu^1=\{\mu^1_t\}_{t\in I_1}, \boldsymbol\nu^1=\{\nu^1_t\}_{t\in I_1},\omega^1)$, is defined by setting $\mu^1_t:=\mu_t$ and $\nu^1_t:=\nu_t$ for all $t\in I_1$, $\omega^1=\omega_\rho(\min I_1)$. Clearly we have $\rho_{|I_1}\in\mathcal F^1_{I_1}$. 
\end{enumerate}
\end{definition}

\medskip

As done in Definition \ref{def:LinftyGF} for the problem of Section \ref{sec:linftymu}, we now want to reconduct this framework to the general one of Definition \ref{def:adm} in order to gain the general results proved in Section \ref{sec:GF}.

\begin{definition}\label{def:L1GF}
Let $c$ and $c_f$ be as in Definition \ref{def:GFobjects}, satisfying the following additional properties
\begin{enumerate}
\item[$\boldsymbol{(C'_3)}$] $c(\mu^{(1)},\omega^{(1)},\mu^{(2)}, \omega^{(2)}, \hat\rho)<+\infty$ if and only if $\hat\rho=(\boldsymbol{\hat\mu},\boldsymbol{\hat\nu},\hat\omega)\in\mathcal F^1_I(\mu^{(1)},\omega^{(1)})$, with $\hat\mu_{|t=\max I}=\mu^{(2)}$ and $\omega_{\hat\rho}(\max I)=\omega^{(2)}$ for some compact and nonempty interval $I\subset\mathbb R$;
\item[$\boldsymbol{(C'_4)}$] let $0\le a\le b\le c$, $\rho=(\boldsymbol{\mu},\boldsymbol{\nu},\omega)\in\mathcal F^1_{[a,c]}$. Then $c:X\times X\times \Sigma\to[0,+\infty]$ is superadditive by restrictions, i.e.
\[c(\mu_a,\omega,\mu_c,\omega_\rho(c),\rho)\ge c(\mu_a,\omega,\mu_b,\omega_\rho(b),\rho_{|[a,b]})+c(\mu_b,\omega_\rho(b),\mu_c,\omega_\rho(c),\rho_{|[b,c]}).\]
\end{enumerate}
Let $\bar\mu\in\mathscr P_p(\mathbb R^d)$, $\bar\omega\in[0,\alpha]$, $I\subset\mathbb R$ nonempty and compact interval, and $\rho=(\boldsymbol\mu,\boldsymbol\nu,\omega)\in\mathcal F^1_I(\bar\mu,\bar\omega)$. We define the set $\mathscr G^1_{\rho}$ made of the pairs $(\gamma,\sigma)$ defined as follows
\begin{enumerate}
\item $\gamma:I\to X$, $\gamma(t):=(\mu_t,\omega_\rho(t))$ for all $t\in I$;
\item $\sigma:I\to\Sigma$, $\sigma(t):=\rho_{|[\min I,t]}$ for all $t\in I$.
\end{enumerate}
Finally, we define the set
\[\mathscr G_I^1(\bar\mu,\bar\omega):=\left\{(\gamma,\sigma)\in\mathscr G_{\rho}^1\,:\,\rho\in\mathcal F^1_I(\bar\mu,\bar\omega)\right\}.\]
\end{definition}

\begin{theorem}[DPP for $L^1$-time feasibility case]\label{thm:DPP1}
Let $V:\mathscr P_p(\mathbb R^d)\times[0,+\infty[\to [0,+\infty]$ be as in Definition \ref{def:GFobjects}. For any $(\mu_0,\omega_0)\in\mathscr P_p(\mathbb R^d)\times[0,+\infty[$, we have
\[V(\mu_0,\omega_0)=\inf_{\substack{\rho=(\boldsymbol\mu,\boldsymbol\nu,\omega)\in\mathcal F^1_I(\mu_0,\omega_0)\\ I\subseteq \mathbb R \textrm{ compact interval}}}\left\{c(\mu_0,\omega_0,\mu_{|t=\max I},\omega_\rho(\max I),\rho)+V(\mu_{|t=\max I},\omega_\rho(\max I))\right\}.\]
\end{theorem}
\begin{proof}
Coming from Theorem \ref{thm:DPP} and $\boldsymbol{(C'_3)}$.
\end{proof}

We notice that Remark \ref{rem:LinftyGF} holds also in this setting, thus we gain Corollaries \ref{cor:DPPadm}, \ref{cor:DPPadm2}.

\begin{proposition}[Existence of minimizers]\label{prop:existence1}
Assume properties $\boldsymbol{(C_1)},\boldsymbol{(C_2)},\boldsymbol{(C'_3)},\boldsymbol{(C'_4)}$. Let $p\ge 1$ and $\bar\mu\in\mathscr P_p(\mathbb R^d)$, $\bar\omega\in[0,\alpha]$. If $c(\bar\mu,\bar\omega,\cdot,\cdot,\cdot)$ and $c_f$ are l.s.c. w.r.t. $w^*/\mathbb R$-convergence and $V(\bar\mu,\bar\omega)<+\infty$, then there exist $I\subset \mathbb R$ nonempty and compact interval and an optimal trajectory $(\gamma,\sigma)\in\mathscr G^1_I(\bar\mu,\bar\omega)$, according to Definition \ref{def:GFopt}.
\end{proposition}
\begin{proof}
Analogous to the proof of Proposition \ref{prop:existenceinfty}.
\end{proof}

\medskip

We pass now to analyze the set of initial velocities in this framework. This result is used later on in Section \ref{sec:HJB1} where an Hamilton-Jacobi-Bellman equation is studied.

\begin{definition}\label{def:ivsetL1}
Let $(\mu,\omega)\in\mathscr P_2(\mathbb R^d)\times[0,\alpha[$, $0\le s\le T$.
\begin{enumerate}
\item Given $\boldsymbol\eta\in\mathcal R_{[s,T]}(\mu,\omega)$, we define the set $\mathscr V_{[s,T]}(\boldsymbol\eta)$ as in Definition~\ref{def:ivsetLinfty}.
\item We set
\begin{equation*}
\mathscr V^1_{[s,T]}(\mu,\omega):=\left\{x\mapsto\int_{\Gamma^x_{[s,T]}}w_{\boldsymbol\eta}(x,\gamma)\,d\eta_x\,:\,\boldsymbol\eta\in\mathcal R^1_{[s,T]}(\mu,\omega),\,
w_{\boldsymbol\eta}\in \mathscr V_{[s,T]}(\boldsymbol\eta)\right\},
\end{equation*}
where we denoted with $\{\eta_x\}_{x\in\mathbb R^d}$ the disintegration of $\boldsymbol\eta$ w.r.t. the map $e_s$.
\item We define the set
\begin{equation*}
\mathscr Z^1(\mu):=
\left\{v\in L^2_\mu(\mathbb R^d;\mathbb R^d)\,:\,v(x)\in F(x) \textrm{ for }\mu\textrm{-a.e. }x\in\mathbb R^d\right\}.
\end{equation*}
\end{enumerate}
\end{definition}

\begin{lemma}[Initial velocity]\label{lem:carattVel1}
Let $(\mu,\omega)\in\mathscr P_2(\mathbb R^d)\times[0,\alpha[$, $0\le s\le T$. Then $\mathscr V^1_{[s,T]}(\mu,\omega)\equiv\mathscr Z^1(\mu)$.
\end{lemma}
\begin{proof}
The proof follows immediately by items $(1)$ and $(4)$ in Lemma \ref{Lemma:initialVel1}, for the inclusion $\mathscr Z^1(\mu)\subseteq\mathscr V^1_{[s,T]}(\mu,\omega)$. 
It is sufficient to take $\beta=\alpha-\omega>0$. The other inclusion comes straightforwardly by Lemma 5 in \cite{CMNP}.
\end{proof}

\begin{remark}
Notice that if $\omega=\alpha$, then $f(\cdot,0)$ is the only admissible velocity driving a feasible trajectory for positive times starting from $(\mu,\omega)$.
\end{remark}

\section{Hamilton-Jacobi-Bellman equation}\label{sec:HJB}
In this section, we determine an Hamilton-Jacobi-Bellman equation for the $L^\infty$-time averaged and the $L^1$-time averaged feasibilty cases of Section \ref{sec:sparse}. 
We consider general cost functions satisfying the properties $\boldsymbol{(C_1)}-\boldsymbol{(C_4)}$ already introduced, and the following property $\boldsymbol{(C_5)}$ regarding a limiting behavior. 
Concerning the definition of sub/super-differentials that we choose to adopt in the space of probability measures, we refer to the recent survey \cite{MQ}.
Our aim is to provide \emph{necessary conditions} in the sense of the formulation of an Hamilton-Jacobi-Bellman equation in the space of measures solved by the value function in a suitable viscosity sense.

\medskip

Referring to Section \ref{sec:GF} and Definition \ref{def:adm}, we ask the following further condition on the cost function $c$.
\begin{itemize}
\item[$\boldsymbol{(C_5)}$] There exists a map $h:X\to\mathbb R$ such that for any $x\in X$, $t>0$ and any $(\gamma,\sigma)$ generalized admissible trajectory from $x$ defined on $[0,t]$, we have
\[\lim_{t\to0^+}\frac{c(x,\gamma(t),\sigma(t))}{t}=h(x).\]
\end{itemize}

\subsection{$L^\infty$-feasibility case}\label{sec:HJBinfty}\

Throughout this section, we consider the framework described in Section \ref{sec:linftymu}, hence
\[X=\mathscr P_2(\mathbb R^d),\quad \Sigma=\displaystyle\bigcup_{\substack{I\subseteq \mathbb R\\ I\textrm{ compact interval}}}\left[AC(I;\mathscr P_2(\mathbb R^d))\times \mathrm{Bor}(I;\mathscr M(\mathbb R^d;\mathbb R^d))\right].\]

We give now a definition of viscosity sub/super-differentials used in this paper, inspired by Definition 3.2 in \cite{MQ}. 
\begin{definition}[Viscosity sub/super-differentials]
Let $w:\mathscr P_2(\mathbb R^d)\to\mathbb R$ be a map, $\bar\mu\in\mathscr P_2(\mathbb R^d)$, $\delta>0$. We say that $p\in L^2_{\bar\mu}(\mathbb R^d)$ belongs to the viscosity $\delta$-superdifferential of $w$ at $\bar\mu$, and we write $p\in D^+_\delta w(\bar\mu)$, if for all $\mu\in\mathscr P_2(\mathbb R^d)$ we have
\begin{equation*}
w(\mu)-w(\bar\mu)\le\int_{\mathbb R^d\times\mathbb R^d\times\mathbb R^d}\langle x_2, x_3-x_1\rangle\,d\tilde\mu(x_1,x_2,x_3)+
\delta W_2(\bar\mu,\mu)+o(W_2(\bar\mu,\mu)),
\end{equation*}
for all $\tilde\mu\in\mathscr P(\mathbb R^d\times\mathbb R^d\times\mathbb R^d)$ such that $\pi_{1,2}\sharp\tilde\mu=(\mathrm{Id}_{\mathbb R^d},p)\sharp\bar\mu$ and $\pi_{1,3}\sharp\tilde\mu\in\Pi(\bar\mu,\mu)$.

In a similar way, the set of the viscosity $\delta$-subdifferentials of $w$ at $\bar\mu$ is defined by $D^-_\delta w(\bar\mu)=-D^+_\delta(-w)(\bar\mu)$.
\end{definition}

We consider the same definition of viscosity sub/super-solutions given in \cite{MQ} as follows.

\begin{definition}[Viscosity solutions]\label{def:viscosol}
Consider the equation
\begin{equation}\label{eq:HJB}
\mathscr H(\mu,Dw(\mu))=0,
\end{equation}
for a given hamiltonian $\mathscr H:T^*(\mathscr P_2(\mathbb R^d))\to\mathbb R$, i.e. $\mathscr H(\mu,p)$ is defined for any $\mu\in\mathscr P_2(\mathbb R^d)$ and $p\in L^2_\mu(\mathbb R^d)$. We say that a function $w:\mathscr P_2(\mathbb R^d)\to\mathbb R$ is
\begin{itemize}
\item a \emph{viscosity subsolution} of \eqref{eq:HJB} if $w$ is u.s.c. and there exists a constant $C>0$ such that
\[\mathscr H(\mu,p)\ge-C\delta,\]
for all $\mu\in\mathscr P_2(\mathbb R^d)$, $p\in D^+_\delta w(\mu)$ and $\delta>0$.
\item a \emph{viscosity supersolution} of \eqref{eq:HJB} if $w$ is l.s.c. and there exists a constant $C>0$ such that
\[\mathscr H(\mu,p)\le C\delta,\]
for all $\mu\in\mathscr P_2(\mathbb R^d)$, $p\in D^-_\delta w(\mu)$ and $\delta>0$.
\item a \emph{viscosity solution} of \eqref{eq:HJB} if $w$ is both a viscosity subsolution and supersolution.
\end{itemize}
\end{definition}

We now prove that a value function with associated cost function satisfying $\boldsymbol{(C_1)}-\boldsymbol{(C_5)}$ is a viscosity solution of a suitable Hamilton-Jacobi-Bellman equation with Hamiltonian defined as follows.
\begin{definition}[Hamiltonian]\label{def:Hinfty}
We define the \emph{hamiltonian} $\mathscr H^\infty:T^*(\mathscr P_2(\mathbb R^d))\to\mathbb R$ by
\[\mathscr H^\infty(\mu,p):=h(\mu)+\inf_{v\in\mathscr Z^\infty(\mu)}\int_{\mathbb R^d}\langle p,v\rangle\,d\mu,\]
for any $(\mu,p)\in T^*(\mathscr P_2(\mathbb R^d))$, where the map $h$ is given in $\boldsymbol{(C_5)}$.
\end{definition}

\begin{theorem}[HJB]\label{thm:HJBinfty}
Assume properties $\boldsymbol{(C_1)}-\boldsymbol{(C_5)}$. Let $\mathcal B\subseteq\mathscr P_2(\mathbb R^d)$ open with uniformly bounded $2$-moments. Assume $c(\mu,\cdot,\cdot)$ and $c_f$ to be l.s.c. w.r.t. $w^*$-convergence, $\mu\in\mathcal B$, and the value function $V:\mathscr P_2(\mathbb R^d)\to\mathbb R$ to be continuous on $\mathcal B$. Then $V$ is a viscosity solution of $\mathscr H^\infty(\mu,DV(\mu))=0$ on $\mathcal B$, where $\mathscr H^\infty$ is defined as in \ref{def:Hinfty}.
\end{theorem}
\begin{proof}
\emph{Claim 1.} $V$ is a subsolution of  $\mathscr H^\infty(\mu,DV(\mu))=0$ on $\mathcal B$.
\smallskip

Let $\bar\mu\in\mathcal B$, $\delta>0$, $p\in D^+_\delta V(\bar\mu)$. Let $v\in\mathscr Z^\infty(\bar\mu)$, then by Lemma \ref{lem:carattVelinfty} there exists $(\boldsymbol\mu,\boldsymbol\nu)\in\mathcal F^\infty_{[0,T]}(\bar\mu)$, $\boldsymbol\eta$ representing $(\boldsymbol\mu,\boldsymbol\nu)$ such that
\[\lim_{t\to 0^+}\int_{\mathbb R^d\times\Gamma_T}\langle p\circ e_0,\frac{e_t-e_0}{t}\rangle\,d\boldsymbol\eta=\int_{\mathbb R^d}\langle p,v\rangle\,d\bar\mu.\]

By Theorem \ref{thm:DPPinfty},
\[V(\mu_t)-V(\bar\mu)+c(\bar\mu,\mu_t,\sigma(t))\ge 0,\]
for all $t\in[0,T]$, where $\sigma(t):=(\boldsymbol\mu,\boldsymbol\nu)_{|[0,t]}$. Notice that if we define $\tilde\mu=(e_0,p\circ e_0,e_t)\sharp\boldsymbol\eta$, we have $\pi_{1,2}\sharp\tilde\mu=(\mathrm{Id}_{\mathbb R^d},p)\sharp\bar\mu$ and $\pi_{1,3}\sharp\tilde\mu=(e_0,e_t)\sharp\boldsymbol\eta\in\Pi(\bar\mu,\mu_t)$.
Hence, $W_2(\bar\mu,\mu_t)\le\|e_t-e_0\|_{L^2_{\boldsymbol\eta}}$, which vanishes as $t\to 0^+$ by continuity of $t\mapsto e_t$ (see Proposition 2.3 in \cite{MQ}).

Thus, we can apply the definition of viscosity superdifferential with $\tilde\mu$ as before and have
\begin{align*}
0&\le V(\mu_t)-V(\bar\mu)+c(\bar\mu,\mu_t,\sigma(t))\\
&\le \int_{\mathbb R^d\times\mathbb R^d\times\mathbb R^d}\langle x_2, x_3-x_1\rangle\,d\tilde\mu(x_1,x_2,x_3)+
\delta\, W_2(\bar\mu,\mu_t)+o(W_2(\bar\mu,\mu_t))+c(\bar\mu,\mu_t,\sigma(t))\\
&\le\int_{\mathbb R^d\times\Gamma_T}\langle p\circ e_0,e_t-e_0\rangle\,d\boldsymbol\eta+\delta\,\|e_t-e_0\|_{L^2_{\boldsymbol\eta}}+o(\|e_t-e_0\|_{L^2_{\boldsymbol\eta}})+c(\bar\mu,\mu_t,\sigma(t)).
\end{align*}
Dividing by $t$, we have $\|\frac{e_t-e_0}{t}\|_{L^2_{\boldsymbol\eta}}\le K$, where $K$ is a suitable constant coming from Proposition \ref{prop:moments} and from hypothesis. Hence,
\begin{equation*}
-\delta K\le\int_{\mathbb R^d\times\Gamma_T}\langle p\circ e_0,\frac{e_t-e_0}{t}\rangle\,d\boldsymbol\eta+\frac{c(\bar\mu,\mu_t,\sigma(t))}{t},
\end{equation*}
and by letting $t\to 0^+$ we get
\begin{equation*}
-\delta K\le\int_{\mathbb R^d}\langle p(x),v(x)\rangle\,d\bar\mu(x)+h(\bar\mu).
\end{equation*}
We conclude by passing to the infimum on $v\in\mathscr Z^\infty(\bar\mu)$.

\medskip

\emph{Claim 2.} $V$ is a supersolution of  $\mathscr H^\infty(\mu,DV(\mu))=0$ on $\mathcal B$.

\smallskip

Take $\bar\mu\in\mathcal B$, $\delta>0$, $p\in D^-_\delta V(\bar\mu)$. By Proposition \ref{prop:existenceinfty}, there exist $T>0$, an optimal trajectory $(\gamma,\sigma)\in\mathscr G^\infty_\rho$ with $\rho=(\boldsymbol\mu,\boldsymbol\nu)\in\mathcal F^\infty_{[0,T]}(\bar\mu)$, and a representation $\boldsymbol\eta$ such that
\[V(\mu_t)-V(\bar\mu)+c(\bar\mu,\mu_t,\sigma(t))=0, \quad\textrm{for all }t\in[0,T].\]

We can take as before $\tilde\mu=(e_0,p\circ e_0,e_t)\sharp\boldsymbol\eta$, thus we have  $W_2(\bar\mu,\mu_t)\le\|e_t-e_0\|_{L^2_{\boldsymbol\eta}}$ and we obtain
\begin{align*}
0&=V(\mu_t)-V(\bar\mu)+c(\bar\mu,\mu_t,\sigma(t))\\
&\ge\int_{\mathbb R^d\times\Gamma_T}\langle p\circ e_0, e_t-e_0\rangle\,d\boldsymbol\eta-\delta\,\|e_t-e_0\|_{L^2_{\boldsymbol\eta}}-o(\|e_t-e_0\|_{L^2_{\boldsymbol\eta}})+c(\bar\mu,\mu_t,\sigma(t)).
\end{align*}
Dividing by $t$ and reasoning as in Claim 1, we get
\begin{equation*}
\delta K\ge\int_{\mathbb R^d\times\Gamma_T}\langle p\circ e_0,\frac{e_t-e_0}{t}\rangle\,d\boldsymbol\eta+\frac{c(\bar\mu,\mu_t,\sigma(t))}{t}.
\end{equation*}
Now, there exists a sequence $\{t_i\}_{i\in\mathbb N}\subseteq]0,T[$ and $w_{\boldsymbol\eta}\in\mathscr V_{[0,T]}(\boldsymbol\eta)$ s.t. $t_i\to 0^+$, $\frac{e_{t_i}-e_0}{t_i}$ weakly converge to $w_{\boldsymbol\eta}$ in $L^2_{\boldsymbol\eta}$, thus by letting $t_i\to 0^+$, thanks to Lemma \ref{lem:carattVelinfty}, there exists $v\in\mathscr Z^\infty(\bar\mu)$ s.t.
\[\delta K\ge\int_{\mathbb R^d}\langle p,v\rangle\,d\bar\mu + h(\bar\mu)\ge\mathscr H^\infty(\bar\mu,p).\]
\end{proof}

\subsection{$L^1$-feasibility case}\label{sec:HJB1}\

In this section, we consider the framework described in Section \ref{sec:l1mu}, hence
\[X=\mathscr P_2(\mathbb R^d)\times[0,+\infty[,\quad \Sigma=\displaystyle\bigcup_{\substack{I\subseteq \mathbb R\\ I\textrm{ compact interval}}}\left[AC(I;\mathscr P_2(\mathbb R^d))\times \mathrm{Bor}(I;\mathscr M(\mathbb R^d;\mathbb R^d))\right]\times[0,+\infty[.\]

Similarly to the previous case, we give the following.
\begin{definition}[Viscosity sub/super-differentials]
Let $w:\mathscr P_2(\mathbb R^d)\times[0,+\infty[\to\mathbb R$ be a map, $(\bar\mu,\bar\omega)\in\mathscr P_2(\mathbb R^d)\times[0,+\infty[$, $\delta>0$. We say that $(p_{\bar\mu},p_{\bar\omega})\in L^2_{\bar\mu}(\mathbb R^d)\times[0,+\infty[$ belongs to the viscosity $\delta$-superdifferential of $w$ at $(\bar\mu,\bar\omega)$, and we write $(p_{\bar\mu},p_{\bar\omega})\in D^+_\delta w(\bar\mu,\bar\omega)$, if for all $(\mu,\omega)\in\mathscr P_2(\mathbb R^d)\times[0,+\infty[$ we have
\begin{align*}
w(\mu,\omega)-w(\bar\mu,\bar\omega)&\le\int_{\mathbb R^d\times\mathbb R^d\times\mathbb R^d}\langle x_2, x_3-x_1\rangle\,d\tilde\mu(x_1,x_2,x_3)+p_{\bar\omega}\,|\omega-\bar\omega|+\\
&\qquad+\delta \sqrt{W_2^2(\bar\mu,\mu)+|\omega-\bar\omega|^2}+o(W_2(\bar\mu,\mu)+|\omega-\bar\omega|),
\end{align*}
for all $\tilde\mu\in\mathscr P(\mathbb R^d\times\mathbb R^d\times\mathbb R^d)$ such that $\pi_{1,2}\sharp\tilde\mu=(\mathrm{Id}_{\mathbb R^d},p_{\bar\mu})\sharp\bar\mu$ and $\pi_{1,3}\sharp\tilde\mu\in\Pi(\bar\mu,\mu)$.

In a similar way, the set of the viscosity $\delta$-subdifferentials of $w$ at $(\bar\mu,\bar\omega)$ is defined by $D^-_\delta w(\bar\mu,\bar\omega)=-D^+_\delta(-w)(\bar\mu,\bar\omega)$.
\end{definition}

We adopt the same definition of viscosity sub/super-solutions given in Definition \ref{def:viscosol}, with the natural adaptations for this context.

\begin{definition}[Hamiltonian]\label{def:H1}
We define the \emph{hamiltonian} $\mathscr H^1:T^*(\mathscr P_2(\mathbb R^d)\times[0,+\infty[)\to\mathbb R$ by
\[\mathscr H^1(\mu,\omega,p_\mu,p_\omega):=h(\mu,\omega)+\inf_{v\in\mathscr Z^1(\mu)}\left\{\int_{\mathbb R^d}\left[\langle p_\mu(x),v(x)\rangle+p_\omega\,\Psi(x,v(x))\right]\,d\mu(x)\right\},\]
for any $(\mu,\omega,p_\mu,p_\omega)\in T^*(\mathscr P_2(\mathbb R^d)\times[0,+\infty[)$, where the map $h$ is given in $\boldsymbol{(C_5)}$.
\end{definition}

\begin{lemma}\label{Lemma:initialVel1-part2}
Let $(\mu_0,\bar\omega)\in\mathscr P_2(\mathbb R^d)\times[0,\alpha[$, $T>0$. Let $(\boldsymbol\mu,\boldsymbol\nu,\bar\omega)\in\mathcal F^1_{[0,T]}(\mu_0,\bar\omega)$ be a feasible trajectory represented by $\boldsymbol\eta\in\mathscr P(\mathbb R^d\times\Gamma_T)$ in the sense of Definition \ref{def:representation}\eqref{item1:repres}. Then there exist $w:\mathbb R^d\to\mathbb R^d$ $L^2_{\mu_0}$-selection of $F(\cdot)$, and a sequence $t_k\in[0,T]$, $t_k\to0^+$, such that
\begin{enumerate}
\item for all $p\in L^2_{\mu_0}(\mathbb R^d)$ 
\[\lim_{k\to+\infty}\int_{\mathbb R^d\times\Gamma_{[0,T]}} \langle \frac{e_{t_k}(x,\gamma)-e_0(x,\gamma)}{t_k},p\circ e_0(x,\gamma)\rangle\,d\boldsymbol\eta(x,\gamma)=\int_{\mathbb R^d}\langle w(x),p(x)\rangle\,d\mu_0(x);\]
\item $\displaystyle\lim_{k\to+\infty}\frac{1}{t_k}\int_0^{t_k}\int_{\mathbb R^d}\Psi\left(x,\frac{\nu_s}{\mu_s}(x)\right)\,d\mu_s(x)\,ds\ge\int_{\mathbb R^d}\Psi(x,w(x))\,d\mu_0(x)$.
\end{enumerate}
\end{lemma}
\begin{proof}
Let $t_k\to 0^+$ be any sequence along which $\frac{e_{t_k}-e_0}{t_k}$ weakly converges in $L^2_{\boldsymbol\eta}$. Then, item $(1)$ follows by Lemma \ref{lem:carattVel1}.

Let us prove the second item. For any $t\in[0,T]$ and any $\varphi\in C^0_C(\mathbb R^d)$ we have
\begin{align*}
\langle \dfrac{e_t-e_0}{t},\varphi\circ e_0\rangle_{L^2_{\boldsymbol\eta}}=&\iint_{\mathbb R^d\times\Gamma_T}\langle \dfrac{\gamma(t)-\gamma(0)}{t},\varphi(x)\rangle\,d\boldsymbol\eta\\
=&\dfrac{1}{t}\int_0^t \iint_{\mathbb R^d\times\Gamma_T}\langle \dot\gamma(s),\varphi(\gamma(0))\rangle\,d\boldsymbol\eta\,ds\\
=&\dfrac{1}{t}\int_0^t \iint_{\mathbb R^d\times\Gamma_T}\langle \dot\gamma(s),\varphi(\gamma(s))\rangle\,d\boldsymbol\eta\,ds+\\
&+\dfrac{1}{t}\int_0^t \iint_{\mathbb R^d\times\Gamma_T}\langle \dot\gamma(s),\varphi(\gamma(0))-\varphi(\gamma(s))\rangle\,d\boldsymbol\eta\,ds  \\
\le&\dfrac{1}{t}\int_0^t \int_{\mathbb R^d}\iint_{\mathbb R^d\times\Gamma_T}\langle \dot\gamma(s),\varphi(\gamma(s))\rangle\,d\boldsymbol\eta(x,\gamma)\,ds+\\
&+\dfrac{1}{t}\int_0^t \iint_{\mathbb R^d\times\Gamma_T} C(1+|\gamma(s)|)|\varphi(\gamma(0))-\varphi(\gamma(s))|\,d\boldsymbol\eta\,ds\\
=&\dfrac{1}{t}\int_0^t \iint_{\mathbb R^d\times\Gamma_T}\langle \frac{\nu_s}{\mu_s}(\gamma(s)),\varphi(\gamma(s))\rangle\,d\boldsymbol\eta\,ds+\dfrac{1}{t}\int_0^t H(s)\,ds\\
=&\dfrac{1}{t}\int_0^t \int_{\mathbb R^d}\langle \frac{\nu_s}{\mu_s}(y),\varphi(y)\rangle\,d\mu_s(y)\,ds+\dfrac{1}{t}\int_0^t H(s)\,ds
\end{align*}
where $s\mapsto H(s)$ is the continuous function defined by 
\[H(s)=\iint_{\mathbb R^d\times\Gamma_T} C(1+|\gamma(s)|)|\varphi(\gamma(0))-\varphi(\gamma(s))|\,d\boldsymbol\eta\,ds.\]
With the very same argument, denoted with $v_s(y)=\frac{\nu_s}{\mu_s}(y)$, we can prove that
\begin{equation}\label{eq:convinitvel}
\left|\dfrac{1}{t}\int_0^t \int_{\mathbb R^d}\langle v_s(y),\varphi(y)\rangle\,\,d\mu_s(y)\,ds- \langle \dfrac{e_t-e_0}{t},\varphi\circ e_0\rangle_{L^2_{\boldsymbol\eta}}\right|\le \dfrac{1}{t}\int_0^t H(s)\,ds,
\end{equation}
and the right hand side tends to $0$ as $t\to 0$.
In particular, as a consequence of item $(1)$, we get
\begin{equation}\label{eq:critical}
\lim_{k\to+\infty}\dfrac{1}{t_k}\int_0^{t_k} \int_{\mathbb R^d}\langle v_s(y),\varphi(y)\rangle\,\,d\mu_s(y)\,ds=\int_{\mathbb R^d}\langle w(x),\varphi(x)\rangle\,d\mu_0(x).
\end{equation}

Now, as already observed in the proof of Proposition \ref{prop:topprop} (Step 1), we can apply Lemma 2.2.3(i) in \cite{But} 
to say that there exist $\{a_h\}_{h\in\mathbb N},\{b_h\}_{h\in\mathbb N}\subseteq C^0(\mathbb R^d;\mathbb R)$ such that $\Psi(x,v)=\sup_{h}[a_h(x)+\langle v,b_h(x)\rangle]$, 
for all $x,v\in\mathbb R^d$. Without loss of generality, we can assume $a_h, b_h$ to have compact support. Thus,
\begin{align*}
&\lim_{k\to+\infty} \dfrac{1}{t_k}\int_0^{t_k}\int_{\mathbb R^d}\Psi\left(x,\frac{\nu_s}{\mu_s}(x)\right)\,d\mu_s(x)\,ds\\
&=\lim_{k\to+\infty} \int_0^1\int_{\mathbb R^d}\Psi\left(x,\frac{\nu_{t_kw}}{\mu_{t_kw}}(x)\right)\,d\mu_{t_kw}(x)\,dw\\
&\ge\lim_{k\to+\infty}\int_0^1\int_{\mathbb R^d}\left(a_h(x)+\langle \frac{\nu_{t_kw}}{\mu_{t_kw}}(x),b_h(x)\rangle\right)\,d\mu_{t_kw}(x)\,dw\\
&=\int_{\mathbb R^d}\left(a_h(x)+\langle w,b_h(x)\rangle\right)\,d\mu_0(x),
\end{align*}
where the last passage follows by absolute continuity of $\boldsymbol\mu$ and by \eqref{eq:critical}.

By positivity of $\Psi$, we can consider $\{a_h\}_{h\in\mathbb N},\{b_h\}_{h\in\mathbb N}$ to be positive. Let $g_h(x):=a_h(x)+\langle w(x),b_h(x)\rangle$, and $\hat g_k:=\max\{g_h(x)\,:\,h\le k\}$. Now, since $\hat g_k$ is a non-decreasing sequence of measurable and non-negative functions with $\sup_h g_h(x)=\sup_k \hat g_k(x)$, then passing to the supremum and applying Beppo-Levi Theorem we have
\[\lim_{k\to+\infty} \dfrac{1}{t_k}\int_0^{t_k}\int_{\mathbb R^d}\Psi\left(x,\frac{\nu_s}{\mu_s}(x)\right)\,d\mu_s(x)\,ds\ge\int_{\mathbb R^d}\Psi(x,w)\,d\mu_0(x).\]
\end{proof}

\begin{theorem}[HJB]\label{thm:HJBL1}
Assume properties $\boldsymbol{(C_1)},\boldsymbol{(C_2)},\boldsymbol{(C'_3)},\boldsymbol{(C'_4)},\boldsymbol{(C_5)}$. Let $\mathcal B\subseteq\mathscr P_2(\mathbb R^d)$ open with uniformly bounded $2$-moments. Assume $c(\mu,\omega,\cdot,\cdot,\cdot)$ and $c_f$ to be l.s.c. w.r.t. $w^*/\mathbb R$-convergence, $\mu\in\mathcal B$ and $\omega\in[0,\alpha[$. Assume the value function $V:\mathscr P_2(\mathbb R^d)\times[0,+\infty[\to\mathbb R$ to be continuous on $\mathcal B\times[0,\alpha[$. Then $V$ is a viscosity solution of $\mathscr H^1(\mu,\omega,DV(\mu,\omega))=0$ on $\mathcal B\times[0,\alpha[$, where $\mathscr H^1$ is defined as in \ref{def:H1}.
\end{theorem}
\begin{proof}
\emph{Claim 1.} $V$ is a subsolution of  $\mathscr H^1(\mu,\omega,DV(\mu,\omega))=0$ on $\mathcal B\times[0,\alpha[$.
\smallskip

Let $(\bar\mu,\bar\omega)\in\mathcal B\times[0,\alpha[$, $\delta>0$, $(p_{\bar\mu},p_{\bar\omega})\in D^+_\delta V(\bar\mu,\bar\omega)$, and $v\in\mathscr Z^1(\bar\mu)$. By Lemma \ref{lem:carattVel1} and \ref{Lemma:initialVel1}, there exist $(\boldsymbol\mu,\boldsymbol\nu)\in\mathcal A_{[0,T]}(\bar\mu)$ and $\boldsymbol\eta$ representing $(\boldsymbol\mu,\boldsymbol\nu)$ such that $\rho:=(\boldsymbol\mu,\boldsymbol\nu,\bar\omega)\in\mathcal F^1_{[0,T]}(\bar\mu,\bar\omega)$, and items $(1),(3)$ of Lemma \ref{Lemma:initialVel1} are satisfied with $v=v_0$.

By Theorem \ref{thm:DPP1} we have
\[V(\mu_t,\omega_\rho(t))-V(\bar\mu,\bar\omega)+c(\bar\mu,\bar\omega,\mu_t,\omega_\rho(t),\sigma(t))\ge 0,\]
for all $t\in[0,T]$, where $\sigma(t):=\rho_{|[0,t]}$. Moreover, if we define $\tilde\mu=(e_0,p_{\bar\mu}\circ e_0,e_t)\sharp\boldsymbol\eta$, then $\pi_{1,2}\sharp\tilde\mu=(\mathrm{Id}_{\mathbb R^d},p_{\bar\mu})\sharp\bar\mu$ and $\pi_{1,3}\sharp\tilde\mu=(e_0,e_t)\sharp\boldsymbol\eta\in\Pi(\bar\mu,\mu_t)$.
Proceeding analogously to the proof of Theorem \ref{thm:HJBinfty} and recalling items $(1),(3)$ of Lemma \ref{Lemma:initialVel1}, we have
\begin{align*}
0&\le V(\mu_t,\omega_\rho(t))-V(\bar\mu,\bar\omega)+c(\bar\mu,\bar\omega,\mu_t,\omega_\rho(t),\sigma(t))\\
&\le \int_{\mathbb R^d\times\mathbb R^d\times\mathbb R^d}\langle x_2, x_3-x_1\rangle\,d\tilde\mu(x_1,x_2,x_3)+p_{\bar\omega}\,|\omega_\rho(t)-\bar\omega|+\\
&\qquad+\delta\, \sqrt{W_2^2(\bar\mu,\mu_t)+|\omega_\rho(t)-\bar\omega|^2}+
o(W_2(\bar\mu,\mu_t)+|\omega_\rho(t)-\bar\omega|)+c(\bar\mu,\bar\omega,\mu_t,\omega_\rho(t),\sigma(t))\\
&= \int_{\mathbb R^d\times\Gamma_T}\langle p_{\bar\mu}\circ e_0, e_t-e_0\rangle\,d\boldsymbol\eta+p_{\bar\omega}\,\int_0^t\int_{\mathbb R^d}\Psi\left(x,\frac{\nu_s}{\mu_s}(x)\right)\,d\mu_s(x)\,ds+\\
&\qquad+\delta\, \sqrt{W_2^2(\bar\mu,\mu_t)+|\omega_\rho(t)-\bar\omega|^2}+
o(W_2(\bar\mu,\mu_t)+|\omega_\rho(t)-\bar\omega|)+c(\bar\mu,\bar\omega,\mu_t,\omega_\rho(t),\sigma(t))\\
&\le \int_{\mathbb R^d\times\Gamma_T}\langle p_{\bar\mu}\circ e_0, e_t-e_0\rangle\,d\boldsymbol\eta+t\,p_{\bar\omega}\,\int_{\mathbb R^d}\Psi(x,v(x))\,d\bar\mu(x)+\\
&\qquad+\delta\, \sqrt{W_2^2(\bar\mu,\mu_t)+|\omega_\rho(t)-\bar\omega|^2}+
o(W_2(\bar\mu,\mu_t)+|\omega_\rho(t)-\bar\omega|)+c(\bar\mu,\bar\omega,\mu_t,\omega_\rho(t),\sigma(t)).
\end{align*}
Let us now divide by $t$ and recall that $\|\frac{e_t-e_0}{t}\|_{L^1_{\boldsymbol\eta}}\le K_1$, for some $K_1>0$, and also $\frac{\omega_\rho(t)-\bar\omega}{t}\le\int_{\mathbb R^d}\Psi(x,v(x))\,d\bar\mu(x)\le K_2$, for some $K_2>0$, by boundedness of $U$. Then, for some $K>0$ we have
\begin{align*}
-\delta K&\le\int_{\mathbb R^d\times\Gamma_T}\langle p_{\bar\mu}\circ e_0,\frac{e_t-e_0}{t}\rangle\,d\boldsymbol\eta + p_{\bar\omega}\,\int_{\mathbb R^d}\Psi(x,v(x))\,d\bar\mu(x)+\\
&\qquad+\frac{1}{t}o(\|e_t-e_0\|_{L^2_{\boldsymbol\eta}}+|\omega_\rho(t)-\bar\omega|)+\frac{1}{t}c(\bar\mu,\bar\omega,\mu_t,\omega_\rho(t),\sigma(t)).
\end{align*}
By letting $t\to 0^+$,
\[-\delta K\le \int_{\mathbb R^d}\langle p_{\bar\mu}(x),v(x)\rangle\,d\bar\mu+p_{\bar\omega}\,\int_{\mathbb R^d}\Psi(x,v(x))\,d\bar\mu + h(\bar\mu,\bar\omega),\]
and we conclude by passing to the infimum w.r.t. $v\in\mathscr Z^1(\bar\mu)$.

\emph{Claim 2.} $V$ is a supersolution of  $\mathscr H^1(\mu,\omega,DV(\mu,\omega))=0$ on $\mathcal B\times[0,\alpha[$.

\smallskip

Let $(\bar\mu,\bar\omega)\in\mathcal B\times[0,\alpha[$, $\delta>0$, $(p_{\bar\mu},p_{\bar\omega})\in D^-_\delta V(\bar\mu,\bar\omega)$. By Proposition \ref{prop:existence1}, there exist an optimal trajectory $\rho:=(\boldsymbol\mu,\boldsymbol\nu,\bar\omega)\in\mathcal F^1_{[0,T]}(\bar\mu,\bar\omega)$ and a representation $\boldsymbol\eta$ such that
\[V(\mu_t,\omega_\rho(t))-V(\bar\mu,\bar\omega)+c(\bar\mu,\bar\omega,\mu_t,\omega_\rho(t),\sigma(t))=0, \quad\textrm{for all }t\in[0,T].\]

As done for Claim 1, we can take $\tilde\mu=(e_0,p_{\bar\mu}\circ e_0,e_t)\sharp\boldsymbol\eta$, thus
\begin{align*}
0&=V(\mu_t,\omega_\rho(t))-V(\bar\mu,\bar\omega)+c(\bar\mu,\bar\omega,\mu_t,,\omega_\rho(t),\sigma(t))\\
&\ge\int_{\mathbb R^d\times\Gamma_T}\langle p_{\bar\mu}\circ e_0, e_t-e_0\rangle\,d\boldsymbol\eta+p_{\bar\omega}\,\int_0^t\int_{\mathbb R^d}\Psi\left(x,\frac{\nu_s}{\mu_s}(x)\right)\,d\mu_s(x)\,ds+\\
&\qquad-\delta\,\sqrt{W_2^2(\bar\mu,\mu_t)+|\omega_\rho(t)-\bar\omega|^2}-
o(W_2(\bar\mu,\mu_t)+|\omega_\rho(t)-\bar\omega|)+c(\bar\mu,\bar\omega,\mu_t,\omega_\rho(t),\sigma(t))\\
&\ge\int_{\mathbb R^d\times\Gamma_T}\langle p_{\bar\mu}\circ e_0, e_t-e_0\rangle\,d\boldsymbol\eta+p_{\bar\omega}\,\int_0^t\int_{\mathbb R^d}\Psi\left(x,\frac{\nu_s}{\mu_s}(x)\right)\,d\mu_s(x)\,ds+\\
&\qquad-\delta\,\sqrt{\|e_t-e_0\|^2_{L^2_{\boldsymbol\eta}}+\left(\int_0^t\int_{\mathbb R^d} \Psi\left(x,\frac{\nu_s}{\mu_s}(x)\right)\,d\mu_s\,ds\right)^2}-
o(\|e_t-e_0\|_{L^2_{\boldsymbol\eta}}+|\omega_\rho(t)-\bar\omega|)\\
&\qquad+c(\bar\mu,\bar\omega,\mu_t,\omega_\rho(t),\sigma(t)).
\end{align*}
Dividing by $t$ and reasoning as in Claim 1, we get
\begin{align*}
\delta K&\ge\int_{\mathbb R^d\times\Gamma_T}\langle p_{\bar\mu}\circ e_0,\frac{e_t-e_0}{t}\rangle\,d\boldsymbol\eta + p_{\bar\omega}\,\frac{1}{t}\int_0^t\int_{\mathbb R^d}\Psi\left(x,\frac{\nu_s}{\mu_s}(x)\right)\,d\mu_s(x)\,ds+\\
&\qquad-\frac{1}{t}o(\|e_t-e_0\|_{L^2_{\boldsymbol\eta}}+|\omega_\rho(t)-\bar\omega|)+\frac{1}{t}c(\bar\mu,\bar\omega,\mu_t,\omega_\rho(t),\sigma(t)).
\end{align*}
Now, by Lemma \ref{Lemma:initialVel1-part2} there exists a sequence $\{t_k\}_{k\in\mathbb N}\subseteq]0,T[$ and $v_0\in\mathscr Z^1(\bar\mu)$ s.t. by passing to the limit in the previous estimate, along the sequence $t_k$, we get
\[\delta K\ge\int_{\mathbb R^d}\langle p_{\bar\mu},v_0\rangle\,d\bar\mu +p_{\bar\omega}\int_{\mathbb R^d} \Psi(x,v_0(x))\,d\bar\mu+ h(\bar\mu,\bar\omega)\ge\mathscr H^1(\bar\mu,\bar\omega,p_{\bar\mu},p_{\bar\omega}).\]
\end{proof}

\section{A special case: the minimum time function}\label{sec:examples}

In this section we show a remarkable example where the theory proposed in this paper can be applied.

For simplicity of exposition, consider the framework outlined in Section \ref{sec:linftymu} (similarly, it is possible to consider the setting of Section \ref{sec:l1mu}). We thus take $X=\mathscr P_2(\mathbb R^d)$, $\Sigma=\displaystyle\bigcup_{\substack{I\subseteq \mathbb R\\ I\textrm{ compact interval}}}\left[AC(I;\mathscr P_2(\mathbb R^d))\times \mathrm{Bor}(I;\mathscr M(\mathbb R^d;\mathbb R^d))\right]$.
A closed nonempty \emph{target set} $\tilde S\subseteq\mathscr P_2(\mathbb R^d)$ is given.

We then define the cost functions $c:X\times X\times\Sigma\to [0,+\infty]$ and $c_f:X\to[0,+\infty]$ as follows
\begin{align*}
c(\mu^{(1)},\mu^{(2)},(\boldsymbol\mu,\boldsymbol\nu))=&\begin{cases}T,&\textrm{ if }(\boldsymbol\mu,\boldsymbol\nu)\in\mathcal F^\infty_{[0,T]}(\mu^{(1)})\textrm{ and }\mu_{|t=T}=\mu^{(2)},\\ +\infty,&\textrm{ otherwise},\end{cases}\\
c_f(\tilde\mu)=&\begin{cases}0,&\textrm{ if }\tilde\mu\in \tilde S,\\ +\infty, &\textrm{ if }\tilde\mu\notin \tilde S.\end{cases}
\end{align*}
Following the notation introduced in Section \ref{sec:GF}, we have that  $V(\mu)<+\infty$ if and only if there exist a feasible trajectory $(\boldsymbol\mu,\boldsymbol\nu)$ joining $\mu$ with the target set $\tilde S$.
In this case $V(\mu)$ is the infimum amount of time where such trajectories are defined, and so $V(\mu)=0$ if and only if $\mu\in \tilde S$.

\begin{remark}
Such a value function $V$ is the so called \emph{minimum time function} studied in \cites{Cav,CMNP,CM,CMO}. It is important to stress that in those references, no feasibility constraints are imposed, dealing just with admissibility properties for trajectories in $\mathscr P_2(\mathbb R^d)$. This paper provides thus a sort of extension of those results.
\end{remark}

We show now that conditions $\boldsymbol{(C_1)}-\boldsymbol{(C_5)}$ are satisfied.

\begin{itemize}
\item Check for $\boldsymbol{(C_1)}$. 
Given $\mu^{(i)}\in X$, $i=1,2,3$ and $(\boldsymbol\mu^{(j)},\boldsymbol\nu^{(j)})=(\{\mu^{(j)}_t\}_{[0,T_j]},\{\nu^{(j)}_t\}_{[0,T_j]})\in\Sigma$, $j=1,2$, we must prove that it is possible to construct $(\boldsymbol\mu',\boldsymbol\nu')=(\{\mu'_t\}_{[0,T']},\{\nu'_t\}_{[0,T']})\in\Sigma$
such that 
\[c(\mu^{(1)},\mu^{(3)},(\boldsymbol\mu',\boldsymbol\nu'))\le c(\mu^{(1)},\mu^{(2)},(\boldsymbol\mu^{(1)},\boldsymbol\nu^{(1)}))+c(\mu^{(2)},\mu^{(3)},(\boldsymbol\mu^{(2)},\boldsymbol\nu^{(2)})).\]
We notice that if $\mu^{(1)}_T\ne \mu^{(2)}$ or $\mu^{(2)}_0\ne \mu^{(2)}$ or $(\boldsymbol\mu^{(j)},\boldsymbol\nu^{(j)})$ are not feasible, there is nothing to prove, since the right hand side is $+\infty$.
Thus we assume $\mu^{(1)}_T=\mu^{(2)}=\mu^{(2)}_0$ and feasibility of $(\boldsymbol\mu^{(j)},\boldsymbol\nu^{(j)})$, $j=1,2$. Let us define $T'=T_1+T_2$ and 
$(\boldsymbol\mu',\boldsymbol\nu')=(\boldsymbol\mu^{(1)},\boldsymbol\nu^{(1)})\star(\boldsymbol\mu^{(2)},\boldsymbol\nu^{(2)})\in\Sigma$. Thus 
\[c(\mu^{(1)},\mu^{(3)},(\boldsymbol\mu',\boldsymbol\nu'))=T'=T_1+T_2=c(\mu^{(1)},\mu^{(2)},(\boldsymbol\mu^{(1)},\boldsymbol\nu^{(1)}))+c(\mu^{(2)},\mu^{(3)},(\boldsymbol\mu^{(2)},\boldsymbol\nu^{(2)})).\]
\item Check for $\boldsymbol{(C_2)}$.
Given $\mu^{(i)}\in X$, $i=1,2$, and $(\boldsymbol\mu,\boldsymbol\nu)=(\{\mu_t\}_{[0,T]},\{\nu_t\}_{[0,T]})\in\Sigma$, we have to prove that there exist $\mu'\in X$, $(\boldsymbol{\mu'}^{(i)},\boldsymbol{\nu'}^{(i)})=(\{{\mu'}^i_t\}_{[0,T'_i]},\{{\nu'}^i_t\}_{[0,T'_i]})\in\Sigma$, $i=1,2$, such that
\[c(\mu^{(1)},\mu^{(2)},(\boldsymbol\mu,\boldsymbol\nu))\ge c(\mu^{(1)},\mu',(\boldsymbol{\mu'}^{(1)},\boldsymbol{\nu'}^{(1)}))+c(\mu',\mu^{(2)},(\boldsymbol{\mu'}^{(2)},\boldsymbol{\nu'}^{(2)})).\]
Indeed, we have that if $(\boldsymbol\mu,\boldsymbol\nu)\not\in\mathscr F^\infty_{[0,T]}(\mu^{(1)})$ or $\mu_{|t=T}\neq\mu^{(2)}$, the result is trivial since the left hand side is $+\infty$. Thus we assume that $(\boldsymbol\mu,\boldsymbol\nu)\in\mathscr F^\infty_{[0,T]}(\mu^{(1)})$ with $\mu_{|t=T}=\mu^{(2)}$. If we take any $\tau\in[0,T]$ and define $\mu'=\mu_{\tau}$, and $(\boldsymbol\mu'^{(1)},\boldsymbol\nu'^{(1)})=(\boldsymbol\mu,\boldsymbol\nu)_{|[0,\tau]}$ and $(\boldsymbol\mu'^{(2)},\boldsymbol\nu'^{(2)})=(\boldsymbol\mu,\boldsymbol\nu)_{|[\tau,T]}$, this provides
the desired inequality.
\item Condition $\boldsymbol{(C_3)}$ holds immediately by definition of $c(\cdot,\cdot,\cdot)$.
\item Condition $\boldsymbol{(C_4)}$ follows by additivity of the cost function $c(\cdot,\cdot,\cdot)$, defined above, along feasible trajectories, together with the well posedness of the restriction operation for feasible trajectories.
\item Condition $\boldsymbol{(C_5)}$ holds with the constant function $h\equiv 1$.
\end{itemize}

We leave to future research the study of conditions assuring some regularity for the minimum time function in this framework. We refer the reader to \cite{Cav} for some discussions concerning the case with no feasibility constraints. The study of higher order controllability conditions in this setting is still an open research direction. We address the reader to the recent issue \cite{CMlie} for a consistent definition of Lie brackets in a measure-theoretic setting which could possibly be used for the study of second order controllability results.

\begin{remark}
We observe that, along an optimal trajectory, the minimum time function previously outlined gives the mass' time of last entry into the target set. 
Alternatively, keeping the same definition of $c_f$, we can give another definition of the cost $c$ leading to a value function that is an \emph{averaged minimum time function}. This is inspired by \cite{CMPproceed}, where the authors provide also a possible example of application in the case with no interactions. Let $S\subseteq \mathbb R^d$ be non empty and closed, and $\tilde S=\{\mu\in\mathscr P_2(\mathbb R^d)\,:\,\mathrm{supp}\,\mu\subseteq S\}$, we define
\begin{equation*}
c(\mu^{(1)},\mu^{(2)},(\boldsymbol\mu,\boldsymbol\nu))=\begin{cases}\displaystyle\int_0^T\int_{\mathbb R^d}\chi_{\mathbb R^d\setminus S}(x)d\mu_s(x)\,ds,&\textrm{if }(\boldsymbol\mu,\boldsymbol\nu)\in\mathcal F^\infty_{[0,T]}(\mu^{(1)}),\\&\textrm{and }\mu_{|t=T}=\mu^{(2)},\\
+\infty,&\textrm{ otherwise}.\end{cases}
\end{equation*}
We can easily check that properties $\boldsymbol{(C_1)}-\boldsymbol{(C_5)}$ holds true with $h(\bar\mu)=\bar\mu(\mathbb R^d\setminus S)$. Thus, if the initial datum $\bar\mu$ is such that $\mathrm{supp}\,\bar\mu\subseteq \mathbb R^d\setminus S$, then the Hamiltonians for the minimum time function and for the averaged minimum time function coincide. This was also observed in \cite{CMP} dealing with an averaged minimum-time problem with no feasibility constraints.
\end{remark}

\section{Application to an advertising campaign strategy model}\label{sec:appl}
An example of application of the theory developed in the present paper can be provided by a very simplified model of allocation of resources in an advertising campaign of a political party.
Indeed, in this case we assume that each point $P(x_1,...,x_d)\in\mathbb R^d$ represents the \emph{attitude} of each voter towards a list of $d$ statements, gathering the
main themes of the public debate. For instance, we may assume that $x_i=0$ means neutrality towards the $i$-th statement, while a value $x_i>0$ (resp. $x_i<0$) 
denotes a positive (resp. negative) attitude towards it. The strength of such positive or negative attitude is given by $|x_i|$.
\par\medskip\par
The interval time $[0,T]$ represents the time period of the campaign. 
Since the total number of voters can be assumed constant during the campaign, we can represent the evolution of the voters' distribution in the space of opinions  $\mathbb R^d$ by a 
time-dependent probability measure $\{\mu_t\}_{t\in[0,T]}\subseteq\mathscr P(\mathbb R^d)$. In this case, given $d$ real intervals $I_1,\dots,I_d$, 
the number $\mu_t(I_1\times\dots\times I_d)$ will represent the fraction of the voters which have attitude toward the $i$-th statement belonging to the range $I_i$, $i=1,\dots, d$.
The initial configuration $\bar\mu$ of the voter's attitude can be approximately deduced with polls, interviews, analyzing the trend topics in the social networks, and so on. 
\par\medskip\par
Due to the conservation of the mass, the voters' distribution evolves according to the (controlled) continuity equation 
\[\begin{cases}
\partial_t\mu_t+\mathrm{div}(v_t\mu_t)=0,\\
\mu_{0}=\bar \mu.
\end{cases}\]
where $v_t(x)\in F(x)$ for $\mu_t$-a.e. $x\in\mathbb R^d$ and a.e. $t\in [0,T]$.
The set-valued map $F(x)$ represents the possible infinitesimal evolutions starting from the opinion $x$, 
which are affected both by the advertising effort, represented by the control $u(\cdot)$, and by the general public opinion trend, 
represented by the drift term $f_0(x)$. A convenient choice can be to take $f_0(x)=a x$ for some $a\in\mathbb R$. In the case $a>0$ each voter - in absence of advertising - will tend 
to radicalize his/her position towards extreme positions as the time passes, while for $a<0$ the voters will spontaneously mitigate their opinion in time, tending to neutrality
towards the selected political themes (or, in other terms, the selected political themes will go out of fashion). 
In this simplified model we do not assume any interaction among the voters. 
Models including these effects have been studied by \cite{ACFK,Carrillo2018} but only numerical results are available at the present moment.
\par\medskip\par
Aim of the political party is to minimize the cost of the campaign needed to steer the voters near to a desired configuration, representing the party's political positions.
The $W_2$-distance between the final configuration and the set of desired ones can be taken as exit cost, while the running cost will be proportional to the minimum amount of advertising effort $u(\cdot)$
needed to generate the desired evolution.
Thus we are naturally led to a cost of type
\[\mathscr J(\boldsymbol\mu,\boldsymbol\nu)=\int_0^T\int_{\mathbb R^d} \left[\Psi\left(x,\dfrac{\nu_t}{\mu_t}(x)\right)+I_{F(x)}\left(\dfrac{\nu_t}{\mu_t}(x)\right)\right]\,d\mu_t+\mathscr G(\mu_T),\]
where, once represented by $\mathscr S$ the desidered opinion ditributions, we set
\[\mathscr G(\mu)=\inf\{W_2(\mu,\theta):\,\theta\in\mathscr S\},\]
and we want to minize it over the pairs $(\boldsymbol\mu=\{\mu_t\}_{t\in[0,T]},\boldsymbol\nu=\{\nu_t\}_{t\in[0,T]})\subseteq \mathscr P(\mathbb R^d)\times\mathscr M(\mathbb R^d;\mathbb R^d)$
satisfying $\partial_t\mu_t+\mathrm{div}\,\nu_t=0$ and $|\nu_t|\ll \mu_t$ for a.e. $t\in [0,T]$.
We notice that, due to the possible presence of different opinions in the party, in general the final configuration set $\mathscr S$ 
cannot be reduced neither to the concentration of the whole mass into a single point 
(total \emph{unanimity} on a specific position), nor to uniformly distribute it on a certain set of opinions in $\mathbb R^d$ (which represents \emph{uniform distribution} in some range of opinions), 
but in general will exhibit an internal non-uniform distribution among the opinions.
\par\medskip\par
In this framework, we considered two additional constraints to the advertising campaign: indeed, we can consider a $L^\infty([0,T])$-bound on the instantaneous cost
\[t\mapsto\int_{\mathbb R^d}\Psi\left(x,\dfrac{\nu_t}{\mu_t}(x)\right)\,d\mu_t(x),\]
representing a constraint on the instantaneous amount of the resources that could be allocated for the advertising.
\par\medskip\par
Usually a bound of this kind is enforced by some antitrust authority to ensure a fair competition among the parties. For instance, this would like to avoid   
a concentration of the advertising on the media for the parties that can rely on wealthy supporting lobbies, which will translate into an unfair advantage 
on the other parties.
\par\medskip\par
The other kind of constraint is instead on the $L^1([0,T])$-norm of the same map, which represents the total cost of the advertising campaign. This is motivated by 
the financial resources raised that can be invested.
\par\medskip\par
One of the main problem in defining an optimal strategy is whether to concentrate the available resources trying to convince a particular segment of the public opinion (usually already not too far from the
party's position), or to launch a general campaign all over the population. In some cases these two strategies can lead to quite different outcomes, depending also on the electoral system,
in other cases the outcome will be almost the same. Our analysis can be seen as first step in an analytical analysis of such situations.
Future direction of this study will be to include an effect of interactions among the voters, which can lead to the possibility to control the voters only by concentrating the advertising effort
on some \emph{influencers}. From a numerical point of view, such kind of models have been treated by \cite{Carrillo2018,CFPT,FPR}.

\appendix
\section{A sparsity constraint in Lagrangian formulation}\label{sec:appendix}
We will provide here a sparsity constraint in a \emph{Lagrangian} formulation, while the setting presented in Section \ref{sec:rainbow} concerned a sparsity constraint given in an \emph{Eulerian} point of view.
To this aim, we state a feasibility constraint on the Carath\'eodory solutions of the differential inclusion $\dot\gamma(t)\in F(\gamma(t))$ as follows
\begin{itemize}
\item we consider a notion of \emph{extended characteristic}, by coupling each characteristic $\gamma$ with a time-dependent curve $\zeta(\cdot)$ related to the amount of control needed to generate it; 
\item we put a feasibility constraint on the probability measures concentrated on this extended notion of characteristics, i.e., 
we select the measures concentrated on extended curves satisfying a control sparsity constraint.
\end{itemize}
This microscopic sparsity constraint allows also to select a (possibly not unique) probabilistic representation for feasible trajectories, since it can be used to prescribe the paths 
to be followed by the microscopic particles. This makes it particulary suitable for applications in irrigation problems or dynamics set on networks \cites{CdMT17,CdMT18}. 
We call this constraint \emph{the $L^\infty$-extended curve-based feasibility condition}.

\medskip

We first introduce the following sets and operators to deal with continuous curves in the extended space $\mathbb R^d\times\mathbb R$ and which are used to outline the problem considered in this section. We consider only compact and nonempty intervals $I$ of $\mathbb R$.
We define
\begin{enumerate}
\item the extended space $\tilde\Gamma_I:=C^0(I;\mathbb R^d\times\mathbb R)$;
\item the extended evaluation operator $\tilde e_{I,t}:\mathbb R^{d+1}\times\tilde\Gamma_I\to\mathbb R^{d+1}, \,(\tilde x,\tilde\gamma)\mapsto \tilde\gamma(t)$, where we omit the subscript $I$ when it is clear from the context;
\item given $\hat I\subseteq I$ compact and nonempty interval, we define the continuous map called \emph{restriction operator}
\begin{align*}
R_{I\to \hat I}:\mathbb R^{d+1}\times\tilde\Gamma_{I}&\to\mathbb R^{d+1}\times\tilde\Gamma_{\hat I},\\
(\tilde x,\tilde\gamma)&\mapsto(\tilde\gamma_{|\hat I}(\min \hat I),\tilde\gamma_{|\hat I}),
\end{align*}
where $\tilde\gamma_{|\hat I}(t):=\tilde\gamma(t)$ for all $t\in \hat I$.
\item given $I_i=[a_i,b_i]\subseteq\mathbb R$, $a_2=b_1$, $I=I_1\cup I_2$, we set
\[D_{I_1,I_2}:=\{(\tilde x_1,\tilde\gamma_1,\tilde x_2,\tilde\gamma_2)\in \mathbb R^{d+1}\times\tilde\Gamma_{I_1}\times\mathbb R^{d+1}\times\tilde\Gamma_{I_2}:\,\tilde x_i=\tilde\gamma_i(a_i),\,\tilde\gamma_1(b_1)=\tilde\gamma_2(a_2)\},\]
and we define the continuous map called \emph{merge operator} 
\begin{align*}
M_{I_1,I_2}:D_{I_1,I_2}&\to\mathbb R^{d+1}\times\tilde\Gamma_{I},\\
(\tilde x_1,\tilde\gamma_1,\tilde x_2,\tilde\gamma_2)&\mapsto(\tilde\gamma_1\star\tilde\gamma_2(a_1),\tilde\gamma_1\star\tilde\gamma_2),
\end{align*}
where $\tilde\gamma_1\star\tilde\gamma_2\in \tilde\Gamma_{I}$ is defined as $\tilde\gamma_1\star\tilde\gamma_2(t)=\tilde\gamma_i(t)$ for all $t\in I_i$, $\tilde\gamma_i\in \tilde\Gamma_{I_i}$,
satisfying $\tilde\gamma_1(b_1)=\tilde\gamma_2(a_2)$, $i=1,2$.

We notice that $R_{I\to I_i}\circ M_{I_1,I_2}=\tilde\pi_i$, $i=1,2$ and $M_{I_1,I_2} (R_{I\to I_1}(\tilde x,\tilde\gamma),R_{I\to I_2}(\tilde x,\tilde\gamma))=(\tilde\gamma(a_1),\tilde\gamma)$, where
\[\tilde\pi_i:\mathbb R^{d+1}\times\tilde\Gamma_{I_1}\times\mathbb R^{d+1}\times\tilde\Gamma_{I_2}\to\mathbb R^{d+1}\times\tilde\Gamma_{I_i},\quad (\tilde x_1,\tilde\gamma_1,\tilde x_2,\tilde\gamma_2)\mapsto(\tilde x_i,\tilde\gamma_i).\]
\item Given $\tilde x\in\mathbb R^{d+1}$, $I=[a,b]$, we say that the pair $(\tilde x,\tilde\gamma)\in\mathbb R^{d+1}\times\tilde\Gamma_I$ satisfies the extended dynamical system \textbf{(ES)} in the time interval $I$, if and only if $\tilde\gamma$ is an absolutely continuous solution of the following Cauchy problem
\begin{equation*}
\begin{cases}
\dot\gamma(t)\in F(\gamma(t)), &\textrm{for a.e. }t\in I,\\
\dot\zeta(t)=\Psi(\gamma(t),\dot\gamma(t)), &\textrm{for a.e. }t\in I,\\
\gamma(a)=x\\
\zeta(a)=\omega,
\end{cases}
\end{equation*}
where $\Psi$ is the control magnitude density and we denoted $\tilde x=(x,\omega)\in\mathbb R^d\times\mathbb R$, $\tilde\gamma=(\gamma,\zeta)\in C^0(I;\mathbb R^d)\times C^0(I;\mathbb R)$.
\end{enumerate}

\medskip

Let $\alpha\ge 0$ be fixed. Considering the notation outlined in Section \ref{sec:GF}, we take 
\[X=\mathscr P(\mathbb R^d\times\mathbb R), \,\,\,\textrm{and}\,\,\, \Sigma=\displaystyle\bigcup_{\substack{I\subseteq \mathbb R\\ I\textrm{ compact interval}}}\mathscr P(\mathbb R^{d+1}\times\tilde\Gamma_I).\] 

\begin{definition}[$L^\infty$-extended curve-based feasible set]\label{def:extfeas}
We define the \emph{extended-curve-based admissibility set} by
\begin{equation*}
\mathcal{\tilde A}_I:=\left\{\boldsymbol{\tilde\eta}\in \mathscr P(\mathbb R^{d+1}\times\tilde\Gamma_I)\,:\,\begin{array}{l}\boldsymbol{\tilde\eta} \textrm{ is concentrated on pairs }(\tilde x,\tilde\gamma)\in\mathbb R^{d+1}\times\tilde\Gamma_I \\\textrm{satisfying \textbf{(ES)}}\end{array}\right\}. 
\end{equation*}
Given $\tilde\mu\in\mathscr P(\mathbb R^{d+1})$, we define
\begin{equation*}
\mathcal{\tilde A}_I(\tilde\mu):=\left\{\boldsymbol{\tilde\eta}\in\mathcal{\tilde A}_I\,:\,\tilde e_{\min I}\sharp\boldsymbol{\tilde\eta}=\tilde\mu \right\}. 
\end{equation*}

Given $\boldsymbol{\tilde\eta}\in\mathcal{\tilde A}_I$, we define the map  $\omega_{\boldsymbol{\tilde\eta}}:\mathbb R^{d+1}\times\tilde\Gamma_I\to [0,+\infty]$ 
by setting for $\boldsymbol{\tilde\eta}$-a.e. $(\tilde x,\tilde\gamma)\in \mathbb R^{d+1}\times\tilde\Gamma_I$, with $\tilde x=(x,\omega)\in\mathbb R^d\times\mathbb R$ and $\tilde\gamma=(\gamma,\zeta)\in C^0(I;\mathbb R^d)\times C^0(I;\mathbb R)$,
\begin{equation*}
\omega_{\boldsymbol{\tilde\eta}}(\tilde x,\tilde\gamma):=
\begin{cases}
\omega+\displaystyle\int_{I}\dot\zeta(t)\,dt,&\textrm{ if }\dot\zeta(\cdot)\in L^1(I);\\ \\ 
+\infty,&\textrm{ otherwise}.
\end{cases}
\end{equation*}

We define the \emph{extended curve-based feasibility set} by
\begin{equation*}
\mathcal{\tilde F}_I:=\left\{\boldsymbol{\tilde\eta}\in\mathcal{\tilde A}_I\,:\, \|\omega_{\boldsymbol{\tilde\eta}}\|_{L^{\infty}_{\boldsymbol{\tilde\eta}}}\le \alpha\right\}. 
\end{equation*}
Given $\tilde\mu\in\mathscr P(\mathbb R^{d+1})$, we define
\begin{equation*}
\mathcal{\tilde F}_I(\tilde\mu):=\mathcal{\tilde A}_I(\tilde\mu)\cap\mathcal{\tilde F}_I. 
\end{equation*}
\end{definition}

The following definitions and considerations, that we perform for the \emph{extended curve-based feasible set}, are even more valid if, instead, we consider just the \emph{extended curve-based admissibility set}.

\begin{definition}[Concatenation and restriction]\label{def:concresteta}
\begin{enumerate}
\item[]
\item Let $I_i=[a_i,b_i]\subset\mathbb R$, $i=1,2$, with $b_1=a_2$, and $I:=I_1\cup I_2$. Let $\boldsymbol{\tilde\eta}^i\in\mathcal {\tilde F}_{I_i}$ with $\tilde e_{b_1}\sharp\boldsymbol{\tilde\eta}^1=\tilde e_{a_2}\sharp\boldsymbol{\tilde\eta}^2$.
We define the \emph{concatenation}
\[\boldsymbol{\tilde\eta}^1\star\boldsymbol{\tilde\eta}^2=\tilde\mu\otimes M_{I_1,I_2}\sharp(\tilde\eta^1_{\tilde y}\otimes\tilde\eta^2_{\tilde y}),\]
where $\tilde\mu:=\tilde e_{b_1}\sharp\boldsymbol{\tilde\eta}^1=\tilde e_{a_2}\sharp\boldsymbol{\tilde\eta}^2\in\mathscr P(\mathbb R^d\times\mathbb R)$, and
 $\{\tilde\eta^i_{\tilde y}\}_{\tilde y\in\mathbb R^{d+1}}$ is the measure uniquely defined for $\tilde\mu$-a.e. $\tilde y\in\mathbb R^{d+1}$ by the disintegrations
of $\boldsymbol{\tilde\eta}^1$ and $\boldsymbol{\tilde\eta}^2$ w.r.t. the extended evaluation operators $\tilde e_{b_1}$ and $\tilde e_{a_2}$, respectively.
Notice that for $\tilde\mu$-a.e. $\tilde y\in\mathbb R^{d+1}$ we have $\mathrm{supp}\,(\tilde \eta^1_{\tilde y}\otimes\tilde \eta^2_{\tilde y})\subseteq D_{I_1,I_2}$ by construction and $\boldsymbol{\tilde\eta}^1\star\boldsymbol{\tilde\eta}^2\in\mathcal {\tilde A}_I$.
\item Let $\boldsymbol{\tilde\eta}\in\mathcal {\tilde F}_I$. The \emph{restriction} $\boldsymbol{\tilde\eta}_{|\hat I}$ of $\boldsymbol{\tilde\eta}$ to a compact and nonempty interval $\hat I\subset I$, is defined by setting $\boldsymbol{\tilde\eta}_{|\hat I}:=R_{I\to\hat I}\sharp\boldsymbol{\tilde\eta}$. By construction, $\boldsymbol{\tilde\eta}_{|\hat I}\in\mathcal {\tilde A}_{\hat I}$. 
\end{enumerate}
\end{definition}

We provide here an example showing the result of the concatenation operation in an illustrative situation.
\begin{example}\label{ex:occhio}
Considering the notation of Definition \ref{def:setting}, let $d=2$, $m=1$, $U=[-1,1]$ and  $f_0(x)=(1,0)\in\mathbb R^2$, $f_1(x)=(0,1)\in\mathbb R^2$ for all $x\in\mathbb R^2$. Thus, $F(x)=\{(1,u)\in\mathbb R^2\,:\,u\in U\}$. Let us fix the control sparsity threshold $\alpha=2$.

In $\mathbb R^2$, consider $x_j=(0,j)$, and $\gamma_{x_j}:[0,2]\to\mathbb R^2$, $\gamma_{x_j}(t)=(t,j-t\,\mathrm{sgn}\,j)$, for $j=\pm 1$.
Let $\tilde\gamma_{\tilde x_j}=(\gamma_{x_j},\zeta)\in C^0([0,2];\mathbb R^2)\times C^0([0,2];\mathbb R)$ be a solution of system $\textbf{(ES)}$ in the time-interval $[0,2]$, starting from $\tilde x_j=(x_j,\omega)\in\mathbb R^2\times\mathbb R$, with $\omega=0$. Thus, we have that $\zeta(t)=t$, since $\Psi(x,(1,1))=\Psi(x,(1,-1))=1$ for all $x\in\mathbb R^2$.

Let us define  
\[\boldsymbol{\tilde\eta}^1=\frac{1}{2}\sum_{j\in\{-1,1\}}\delta_{\tilde x_j}\otimes
\delta_{\tilde\gamma_{\tilde x_{j_{|[0,1]}}}}
\in\mathscr P(\mathbb R^{2+1}\times\tilde\Gamma_{[0,1]}),\quad \boldsymbol{\tilde\eta}^2=\frac{1}{2}\sum_{j\in\{-1,1\}}\delta_{\tilde\gamma_{\tilde x_j}(1)}\otimes
\delta_{\tilde\gamma_{\tilde x_{j_{|[1,2]}}}}
\in\mathscr P(\mathbb R^{2+1}\times\tilde\Gamma_{[1,2]}).\]
We see that $\tilde e_1\sharp\boldsymbol{\tilde\eta}^1=\tilde e_1\sharp\boldsymbol{\tilde\eta}^2=\delta_{\tilde\gamma_{\tilde x_j}(1)}=\delta_{((1,0),1)}\in\mathscr P(\mathbb R^{2+1})$.
Now, consider also $\xi_{\tilde x_j}\in C^0([0,2];\mathbb R^{2+1})$,
\begin{equation*}
\xi_{\tilde x_j}(t):=\begin{cases}
\tilde\gamma_{\tilde x_j}(t),&\textrm{for }t\in[0,1],\\
\tilde\gamma_{\tilde x_i}(t),&\textrm{for }t\in[1,2],\,i\in\{-1,1\},\,i\neq j.
\end{cases}
\end{equation*}
We can finally compute the concatenation $\boldsymbol{\tilde\eta}^1\star\boldsymbol{\tilde\eta}^2$ which leads to
\[\boldsymbol{\tilde\eta}:=\boldsymbol{\tilde\eta}^1\star\boldsymbol{\tilde\eta}^2=\frac{1}{4}\displaystyle\sum_{j\in\{-1,1\}}\delta_{\tilde x_j}\otimes\delta_{\tilde\gamma_{\tilde x_j}}+\frac{1}{4}\displaystyle\sum_{j\in\{-1,1\}}\delta_{\tilde x_j}\otimes\delta_{\xi_{\tilde x_j}}\in\mathscr P(\mathbb R^{2+1}\times\tilde\Gamma_{[0,2]}).\]
Moreover, we can prove that $\boldsymbol{\tilde\eta}^1,\boldsymbol{\tilde\eta}^2$ and $\boldsymbol{\tilde\eta}$ are admissible and feasible in the respective time intervals, according to Definition \ref{def:extfeas}.
\end{example}

\begin{proposition}\label{prop:wellposedeta}
Let $I\subset\mathbb R$ be a compact and nonempty interval. Given $\boldsymbol{\tilde\eta}\in\mathcal {\tilde F}_I$, the following restrictions' properties hold.
\begin{enumerate}
\item[$(i)$] For any $\hat I\subseteq I$ compact and nonempty interval, we have $\boldsymbol{\tilde\eta}_{|\hat I}\in\mathcal {\tilde F}_{\hat I}$ and $\tilde e_t\sharp\boldsymbol{\tilde\eta}=\tilde e_t\sharp\boldsymbol{\tilde\eta}_{|\hat I}$ for any $t\in\hat I$.
\item[$(ii)$] For any $I\supseteq I_2\supseteq I_1$ compact and nonempty intervals, we have $\left(\boldsymbol{\tilde\eta}_{|I_2}\right)_{|I_1}\equiv\boldsymbol{\tilde\eta}_{|I_1}$.
\item[$(iii)$] $\boldsymbol{\tilde\eta}\equiv\boldsymbol{\tilde\eta}_{|I}$.
\end{enumerate}
Furthermore, we have the following property for concatenations.
\begin{enumerate}
\item[$(iv)$] Let $I_i=[a_i,b_i]$, $i=1,2$, with $b_1=a_2$ and $I_1\cup I_2=I$ and $\boldsymbol{\tilde\eta}^i\in\mathcal{\tilde F}_{I_i}$, $i=1,2$, with $\tilde e_{b_1}\sharp\boldsymbol{\tilde\eta}^1=\tilde e_{a_2}\sharp\boldsymbol{\tilde\eta}^2$. Then $\boldsymbol{\tilde\eta}:=\boldsymbol{\tilde\eta}^1\star\boldsymbol{\tilde\eta}^2\in\mathcal{\tilde F}_I$ and we have
\begin{equation*}
\tilde e_t\sharp\boldsymbol{\tilde\eta}=
\begin{cases}
\tilde e_t\sharp\boldsymbol{\tilde\eta}^1, &\textrm{if }t\in I_1,\\
\tilde e_t\sharp\boldsymbol{\tilde\eta}^2, &\textrm{if }t\in I_2.
\end{cases}
\end{equation*}
\end{enumerate}
\end{proposition}
\begin{proof}
By construction, observe that $\boldsymbol{\tilde\eta}_{|\hat I}\in\mathcal {\tilde A}_{\hat I}$.\\
\emph{Proof of $(i)$.} First, we want to prove that $\boldsymbol{\tilde\eta}_{|\hat I}\in\mathcal {\tilde F}_{\hat I}$.
By contradiction, let $A\subseteq\mathbb R^{d+1}\times\tilde\Gamma_{\hat I}$ be a Borel set of strictly positive $\boldsymbol{\tilde\eta}_{|\hat I}$-measure such that the feasibility condition is violated. In particular, we have
\[\int_{\mathbb R^{d+1}\times\tilde\Gamma_{\hat I}}\chi_A(\tilde x_1,\tilde\gamma_1)\left[\omega_1+\int_{\hat I}\Psi(\gamma_1(t),\dot\gamma_1(t))\,dt\right]\,d\boldsymbol{\tilde\eta}_{|\hat I}(\tilde x_1,\tilde\gamma_1)>\alpha\boldsymbol{\tilde\eta}_{|\hat I}(A),\]
where we denote with $\tilde x_1=(x_1,\omega_1)$, $\tilde\gamma_1=(\gamma_1,\zeta_1)$. Now by definition of $\boldsymbol{\tilde\eta}_{|\hat I}$, the left-hand side can be rewritten as follows
\begin{equation*}
\int_{\mathbb R^{d+1}\times\tilde\Gamma_I}\chi_A(\tilde\gamma_{|\hat I}(\min \hat I),\tilde\gamma_{|\hat I})\left[\zeta_{|\hat I}(\min\hat I)+\int_{\hat I}\Psi(\gamma_{|\hat I}(t),\dot\gamma_{|\hat I}(t))\,dt\right]\,d\boldsymbol{\tilde\eta}(\tilde x,\tilde\gamma)\le\alpha\boldsymbol{\tilde\eta}_{|\hat I}(A),
\end{equation*}
where we denoted with $\tilde x=(x,\omega)$, $\tilde\gamma=(\gamma,\zeta)$, hence the contradiction. Indeed, we have that for $\boldsymbol{\tilde\eta}$-a.e. $(\tilde x,\tilde\gamma)$
\[\zeta_{|\hat I}(\min \hat I)+\int_{\hat I}\Psi(\gamma_{|\hat I}(t),\dot\gamma_{|\hat I}(t))\,dt=\omega+\int_{\min I}^{\min \hat I}\Psi(\gamma(t),\dot\gamma(t))\,dt+\int_{\hat I}\Psi(\gamma(t),\dot\gamma(t))\,dt\le\alpha,\]
by feasibility of $\boldsymbol{\tilde\eta}$.

\smallskip

Let us now prove that for any $t\in\hat I$, $\tilde e_t\sharp\boldsymbol{\tilde\eta}=\tilde e_t\sharp\boldsymbol{\tilde\eta}_{|\hat I}$. Consider any test function $\varphi\in C^0_b(\mathbb R^{d+1};\mathbb R)$, $t\in\hat I$. We have
\begin{align*}
\int_{\mathbb R^{d+1}}\varphi(\tilde x_1)\,d(\tilde e_t\sharp\boldsymbol{\tilde\eta}_{|\hat I})(\tilde x_1)&=\int_{\mathbb R^{d+1}\times\tilde\Gamma_{\hat I}}\varphi(\tilde\gamma_1(t))\,d\boldsymbol{\tilde\eta}_{|\hat I}(\tilde x_1,\tilde\gamma_1)\\
&=\int_{\mathbb R^{d+1}\times\tilde\Gamma_I}\varphi(\tilde\gamma_{|\hat I}(t))\,d\boldsymbol{\tilde\eta}(\tilde x,\tilde\gamma)\\
&=\int_{\mathbb R^{d+1}\times\tilde\Gamma_I}\varphi(\tilde\gamma(t))\,d\boldsymbol{\tilde\eta}(\tilde x,\tilde\gamma)\\
&=\int_{\mathbb R^{d+1}}\,d(\tilde e_t\sharp\boldsymbol{\tilde\eta}) (\tilde x).
\end{align*}

\smallskip

\emph{Proof of $(ii)$.} For any test function $\varphi\in C^0_b(\mathbb R^{d+1}\times\tilde\Gamma_{I_1})$,
\begin{align*}
\int_{\mathbb R^{d+1}\times\tilde\Gamma_{I_1}}\varphi(\tilde x_1,\tilde\gamma_1)\,d\left(\boldsymbol{\tilde\eta}_{|I_2}\right)_{|I_1}(\tilde x_1,\tilde\gamma_1)&=\int_{\mathbb R^{d+1}\times\tilde\Gamma_{I_2}}\varphi(\tilde\gamma_{|I_1}(\min I_1),\tilde\gamma_{|I_1})\,d\boldsymbol{\tilde\eta}_{|I_2}(\tilde x,\tilde\gamma)\\
&=\int_{\mathbb R^{d+1}\times\tilde\Gamma_I}\varphi(\tilde\gamma_{|{I_1}_{|I_2}}(\min I_1),\tilde\gamma_{|{I_1}_{|I_2}})\,d\boldsymbol{\tilde\eta}(\tilde x,\tilde\gamma)\\
&=\int_{\mathbb R^{d+1}\times\tilde\Gamma_I}\varphi(\tilde\gamma_{|I_1}(\min I_1),\tilde\gamma_{|I_1})\,d\boldsymbol{\tilde\eta}(\tilde x,\tilde\gamma)\\
&=\int_{\mathbb R^{d+1}\times\tilde\Gamma_{I_1}}\varphi(\tilde x_1,\tilde\gamma_1)\,d\boldsymbol{\tilde\eta}_{|I_1}(\tilde x_1,\tilde\gamma_1).
\end{align*}

\smallskip

\emph{Proof of $(iii)$.} By definition of $\boldsymbol{\tilde\eta}_{|I}$, for any test function $\varphi\in C^0_b(\mathbb R^{d+1}\times\tilde\Gamma_I;\mathbb R)$, we have
\begin{align*}
&\int_{\mathbb R^{d+1}\times\tilde\Gamma_{I}}\varphi(\tilde x_1,\tilde\gamma_1)\,\boldsymbol{\tilde\eta}_{|I}(\tilde x_1,\tilde\gamma_1)=\\
&=\int_{\mathbb R^{d+1}\times\tilde\Gamma_I}\varphi(\tilde\gamma_{|I}(\min I),\tilde\gamma_{|I})\,d\boldsymbol{\tilde\eta}(\tilde x,\tilde\gamma)\\
&=\int_{\mathbb R^{d+1}\times\tilde\Gamma_I}\varphi(\tilde x,\tilde\gamma)\,d\boldsymbol{\tilde\eta}(\tilde x,\tilde\gamma).
\end{align*}

\smallskip

\emph{Proof of $(iv)$.} First, we prove that $\boldsymbol{\tilde\eta}^1\star\boldsymbol{\tilde\eta}^2\in\mathcal {\tilde F}_I$. Let us denote with $\tilde\mu:=\tilde e_{b_1}\sharp\boldsymbol{\tilde\eta}^1=\tilde e_{a_2}\sharp\boldsymbol{\tilde\eta}^2$.
By contradiction, let $A\subseteq\mathbb R^{d+1}\times\tilde\Gamma_I$ be a Borel set of strictly positive measure w.r.t. $\boldsymbol{\tilde\eta}^1\star\boldsymbol{\tilde\eta}^2$ such that the feasibility condition is not respected, i.e.
\[\int_{\mathbb R^{d+1}\times\tilde\Gamma_I}\chi_A(\tilde x,\tilde\gamma)\left[\omega+\int_I\Psi(\gamma(t),\dot\gamma(t))\,dt\right]\,d(\boldsymbol{\tilde\eta}^1\star\boldsymbol{\tilde\eta}^2)(\tilde x,\tilde\gamma)>\alpha\cdot\boldsymbol{\tilde\eta}^1\star\boldsymbol{\tilde\eta}^2(A),\]
where we denote with $\tilde x=(x,\omega)$, $\tilde\gamma=(\gamma,\zeta)$. By definition of concatenation, the left-hand side is equivalent to
\begin{align*}
&\int_{\mathbb R^{d+1}}\int_{\tilde e^{-1}_{I_2,t}(\tilde y)}\int_{\tilde e^{-1}_{I_1,t}(\tilde y)} \chi_A(\tilde\gamma_1\star\tilde\gamma_2(a_1),\tilde\gamma_1\star\tilde\gamma_2)\cdot\\
&\quad\qquad\cdot\left[\zeta_1\star\zeta_2(a_1)+\int_I\Psi\left(\gamma_1\star\gamma_2(t),\dfrac{d}{dt}(\gamma_1\star\gamma_2)(t)\right)\,dt\right]\,d\tilde\eta^1_{\tilde y}(\tilde x_1,\tilde\gamma_1)\,d\tilde\eta^2_{\tilde y}(\tilde x_2,\tilde\gamma_2)\,d\tilde\mu(\tilde y)\\
&=\int_{\mathbb R^{d+1}}\int_{\tilde e^{-1}_{I_2,t}(\tilde y)}\int_{\tilde e^{-1}_{I_1,t}(\tilde y)} \chi_A(\tilde\gamma_1\star\tilde\gamma_2(a_1),\tilde\gamma_1\star\tilde\gamma_2)\cdot\\
&\quad\qquad\cdot\left[\omega_1+\int_{I_1}\Psi(\gamma_1(t),\dot\gamma_1(t))\,dt+\int_{I_2}\Psi(\gamma_2(t),\dot\gamma_2(t))\,dt\right]\,d\tilde\eta^1_{\tilde y}(\tilde x_1,\tilde\gamma_1)\,d\tilde\eta^2_{\tilde y}(\tilde x_2,\tilde\gamma_2)\,d\tilde\mu(\tilde y)\\
&=\int_{\mathbb R^{d+1}}\int_{\tilde e^{-1}_{I_2,t}(\tilde y)}\left[\omega_2+\int_{I_2}\Psi(\gamma_2(t),\dot\gamma_2(t))\,dt\right]\cdot\\
&\quad\qquad\cdot\int_{\tilde e^{-1}_{I_1,t}(\tilde y)} \chi_A(\tilde\gamma_1\star\tilde\gamma_2(a_1),\tilde\gamma_1\star\tilde\gamma_2)\,d\tilde\eta^1_{\tilde y}(\tilde x_1,\tilde\gamma_1)\,d\tilde\eta^2_{\tilde y}(\tilde x_2,\tilde\gamma_2)\,d\tilde\mu(\tilde y)\\
&\le\alpha\cdot\boldsymbol{\tilde\eta}^1\star\boldsymbol{\tilde\eta}^2(A),
\end{align*}
by feasibility of $\boldsymbol{\tilde\eta}^2$, where we denoted with $\tilde x_i=(x_i,\omega_i)$, $\tilde\gamma_i=(\gamma_i,\zeta_i)$. Hence the contradiction.

\smallskip

With the same notation as before, let us now prove the other assertion. For any test function $\varphi\in C^0_b(\mathbb R^{d+1};\mathbb R)$, we have
\begin{align*}
&\int_{\mathbb R^{d+1}}\varphi(\tilde x)\,d(\tilde e_t\sharp(\boldsymbol{\tilde\eta}^1\star\boldsymbol{\tilde\eta}^2))(\tilde x)=\\
&=\int_{\mathbb R^{d+1}\times\tilde\Gamma_I}\varphi(\tilde\gamma(t))\,d(\boldsymbol{\tilde\eta}^1\star\boldsymbol{\tilde\eta}^2)(\tilde x,\tilde\gamma)\\
&=\int_{\mathbb R^{d+1}}\int_{\mathbb R^{d+1}\times\tilde\Gamma_I}\varphi(\tilde\gamma(t))\,d[M_{I_1,I_2}\sharp(\tilde\eta^1_{\tilde y}\otimes\tilde\eta^2_{\tilde y})](\tilde x,\tilde\gamma)\,d\tilde\mu(\tilde y)\\
&=\int_{\mathbb R^{d+1}}\int_{\tilde e^{-1}_{I_1,t}(\tilde y)\times \tilde e^{-1}_{I_2,t}(\tilde y)} \varphi(\tilde\gamma_1\star\tilde\gamma_2(t))\,d(\tilde\eta^1_{\tilde y}\otimes\tilde\eta^2_{\tilde y})(\tilde x_1,\tilde\gamma_1,\tilde x_2,\tilde\gamma_2)\,d\tilde\mu(\tilde y)\\
&=\begin{cases}
\tilde e_t\sharp\boldsymbol{\tilde\eta}^1, &\textrm{if }t\in I_1,\\
\tilde e_t\sharp\boldsymbol{\tilde\eta}^2, &\textrm{if }t\in I_2.
\end{cases}
\end{align*}
\end{proof}

\medskip

In the sequel, we give a further property verified by the given definitions of concatenation and restriction, which turn out to be weakly compatible.

\begin{proposition}
Let $I,I_1,I_2\subset\mathbb R$ be compact and nonempty intervals, with $I=I_1\cup I_2$, $I_i=[a_i,b_i]$, $i=1,2$ and $b_1=a_2$. Then the restriction and concatenation operators are weakly compatible, i.e. given $\boldsymbol{\tilde\eta}^i\in\mathcal{\tilde F}_{I_i}$, $i=1,2$ such that $\tilde e_{b_1}\sharp\boldsymbol{\tilde\eta}^1=\tilde e_{a_2}\sharp\boldsymbol{\tilde\eta}^2$, we have $(\boldsymbol{\tilde\eta}^1\star\boldsymbol{\tilde\eta}^2)_{|I_i}=\boldsymbol{\tilde\eta}^i$, $i=1,2$.
\end{proposition}
\begin{proof}
By definition of restriction and concatenation operators, for any $i=1,2$ and for all $\varphi\in C^0(\mathbb R^{d+1}\times\tilde\Gamma_{I_i};\mathbb R)$, we have
\begin{align*}
&\int_{\mathbb R^{d+1}\tilde\Gamma_{I_i}}\varphi(\tilde x_i,\tilde\gamma_i)\,d(\boldsymbol{\tilde\eta}^1\star\boldsymbol{\tilde\eta}^2)_{|I_i}(\tilde x_i,\tilde\gamma_i)=\\
&=\int_{\mathbb R^{d+1}\times\tilde\Gamma_I}\varphi(\tilde\gamma_{|I_i}(a_i),\tilde\gamma_{|I_i})\,d(\boldsymbol{\tilde\eta}^1\star\boldsymbol{\tilde\eta}^2)(\tilde x,\tilde\gamma)\\
&=\int_{\mathbb R^{d+1}}\int_{\tilde e^{-1}_{I_1,t}(\tilde y)\times\tilde e^{-1}_{I_2,t}(\tilde y)}\varphi\left((\tilde\gamma_1\star\tilde\gamma_2)_{|I_i}(a_i),(\tilde\gamma_1\star\tilde\gamma_2)_{|I_i}\right)\,d(\tilde\eta^1_{\tilde y}\otimes\tilde\eta^2_{\tilde y})(\tilde x_1,\tilde\gamma_1,\tilde x_2,\tilde\gamma_2)\,d\tilde\mu(\tilde y)\\
&=\int_{\mathbb R^{d+1}\times\tilde\Gamma_{I_i}}\varphi(\tilde\gamma_i(a_i),\tilde\gamma_i)\,d\boldsymbol{\tilde\eta}^i(\tilde x_i,\tilde\gamma_i).
\end{align*}
\end{proof}

We notice that the properties just proved hold also for the settings described in Sections \ref{sec:linftymu} and \ref{sec:l1mu} in a straightforward way.

\begin{remark}
As the following example shows, a stronger compatibility relation between restriction and concatenation operators is not true in general in this context. Indeed, let $I,I_1,I_2\subset\mathbb R$ be compact and nonempty intervals, with $I=I_1\cup I_2$, $I_i=[a_i,b_i]$, $i=1,2$ and $b_1=a_2$, and $\boldsymbol{\tilde\eta}\in\mathcal {\tilde A}_I$. Then, in general, we cannot write $\boldsymbol{\tilde\eta}=\boldsymbol{\tilde\eta}_{|I_1}\star\boldsymbol{\tilde\eta}_{|I_2}$.
\end{remark}
\begin{example}
Consider the same framework outlined in Example \ref{ex:occhio}, with $I=I_1\cup I_2$, $I_1=[0,1]$, $I_2=[1,2]$. Let us define
\[\boldsymbol\rho:=\frac{1}{2}\displaystyle\sum_{j\in\{-1,1\}}\delta_{\tilde x_j}\otimes\delta_{\tilde\gamma_{\tilde x_j}}\in\mathcal{\tilde A}_{[0,2]}, \quad\boldsymbol\xi:=\frac{1}{2}\displaystyle\sum_{j\in\{-1,1\}}\delta_{\tilde x_j}\otimes\delta_{\xi_{\tilde x_j}}\in\mathcal{\tilde A}_{[0,2]}.\]
We have that $\boldsymbol\rho_{|I_i}=\boldsymbol\xi_{|I_i}=\boldsymbol{\tilde\eta}_{|I_i}=\boldsymbol{\tilde\eta}^i$, $i=1,2$. Nevertheless, both $\boldsymbol\rho$ and $\boldsymbol\xi$ are different from $\boldsymbol{\tilde\eta}^1\star\boldsymbol{\tilde\eta}^2=\boldsymbol{\tilde\eta}$.
\end{example}

\medskip

This construction provides another situation where the abstract Dynamic Programming Principle of Theorem \ref{thm:DPP} holds true (see Theorem \ref{thm:DPPeta}).
Furthermore, with the following definition we can reconduct these objects to that of Definition \ref{def:adm}, thus gaining the validity of the general results proved in Section \ref{sec:GF}.

\begin{definition}\label{def:etaGF}
Let $c$ and $c_f$ be as in Definition \ref{def:GFobjects}, satisfying the following additional properties
\begin{enumerate}
\item[$\boldsymbol{(C''_3)}$] $c(\tilde\mu^{(1)},\tilde\mu^{(2)},\boldsymbol{\tilde\eta})<+\infty$ if and only if $\boldsymbol{\tilde\eta}\in\mathcal {\tilde F}_I(\tilde\mu^{(1)})$, with $\tilde e_{\max I}\sharp\boldsymbol{\tilde\eta}=\tilde\mu^{(2)}$ for some compact and nonempty interval $I\subset\mathbb R$;
\item[$\boldsymbol{(C''_4)}$] let $0\le a\le b\le c$, $\boldsymbol{\tilde\eta}\in\mathcal {\tilde F}_{[a,c]}$. Then $c:X\times X\times \Sigma\to[0,+\infty]$ is superadditive by restrictions, i.e.
\[c(\tilde e_a\sharp\boldsymbol{\tilde\eta},\tilde e_c\sharp\boldsymbol{\tilde\eta},\boldsymbol{\tilde\eta})\ge c(\tilde e_a\sharp\boldsymbol{\tilde\eta},\tilde e_b\sharp\boldsymbol{\tilde\eta},\boldsymbol{\tilde\eta}_{|[a,b]})+c(\tilde e_b\sharp\boldsymbol{\tilde\eta},\tilde e_c\sharp\boldsymbol{\tilde\eta},\boldsymbol{\tilde\eta}_{|[b,c]}).\]
The finiteness of each member follows from item $(i)$ in Proposition \ref{prop:wellposedeta}.
\end{enumerate}
Let $\tilde\mu\in\mathscr P(\mathbb R^{d+1})$, $I\subset\mathbb R$ nonempty and compact interval, and $\boldsymbol{\tilde\eta}\in\mathcal {\tilde F}_I(\tilde\mu)$. We define the set $\mathscr G^{\mathcal {\tilde F}}_{\boldsymbol{\tilde\eta}}$ made of the pairs $(\xi,\sigma)$ defined as follows
\begin{enumerate}
\item $\xi:I\to X$, $\xi(t):=\tilde e_t\sharp\boldsymbol{\tilde\eta}$ for all $t\in I$;
\item $\sigma:I\to\Sigma$, $\sigma(t):=\boldsymbol{\tilde\eta}_{|[\min I,t]}$ for all $t\in I$.
\end{enumerate}
Finally, we define the set
\[\mathscr G_I^{\mathcal {\tilde F}}(\tilde\mu):=\left\{(\xi,\sigma)\in\mathscr G^{\mathcal {\tilde F}}_{\boldsymbol{\tilde\eta}}\,:\,\boldsymbol{\tilde\eta}\in\mathcal {\tilde F}_I(\tilde\mu)\right\}.\]
\end{definition}

\begin{theorem}[DPP for the Lagrangian sparsity case]\label{thm:DPPeta}
Let $V:\mathscr P(\mathbb R^d\times \mathbb R)\to [0,+\infty]$ be as in Definition \ref{def:GFobjects}. For any $\tilde\mu_0\in\mathscr P(\mathbb R^d\times \mathbb R)$ we have
\[V(\tilde\mu_0)=\inf_{\substack{\boldsymbol{\tilde\eta}\in\mathcal{\tilde F}_I(\tilde\mu_0)\\ I\subseteq \mathbb R \textrm{ compact interval}}}\left\{c(\tilde\mu_0,\tilde e_{\max I}\sharp\boldsymbol{\tilde\eta},\boldsymbol{\tilde\eta})+V(\tilde e_{\max I}\sharp\boldsymbol{\tilde\eta})\right\}.\]
\end{theorem}
\begin{proof}
The proof follows by Theorem \ref{thm:DPP} and $\boldsymbol{(C''_3)}$.
\end{proof}

The following remark, mentioned also for the feasibility cases analyzed in Section \ref{sec:sparse}, holds also in this setting.
\begin{remark}\label{rem:etaGF}
Any $(\xi,\sigma)\in\mathscr G_{\boldsymbol{\tilde\eta}}^{\mathcal {\tilde F}}$, s.t. $\boldsymbol{\tilde\eta}\in\mathcal {\tilde F}_I(\tilde\mu)$, is generalized admissible from $\tilde\mu$, according to Definition \ref{def:adm}. Indeed, item $(1)$ in Definition \ref{def:adm} is obviously verified, while item $(2)$ comes from item $(i)$ of Proposition \ref{prop:wellposedeta}.
Finally, by properties $(i)-(ii)$ proved in Proposition \ref{prop:wellposedeta}, we have that $\boldsymbol{(C''_4)}$ implies item $(3)$ in Definition \ref{def:adm} by taking for instance $\sigma_{t_1\to t_2}:=\boldsymbol{\tilde\eta}_{|[t_1,t_2]}$ for any $0\le t_1\le t_2\le T$. Hence, Corollaries \ref{cor:DPPadm}, \ref{cor:DPPadm2} hold in this setting. Furthermore, by $\boldsymbol{(C''_4)}$ we have that $c(\xi(t),\xi(t),\sigma_{t\to t})=0$ for all $t\in I$. Indeed, by item $(iii)$ in Proposition \ref{prop:wellposedeta} we can indentify any $\hat\rho\in\mathcal F^\infty_{[t,t]}$ with its restriction $\hat\rho_{|[t,t]}$.
\end{remark}

\section{Basic estimates}\label{sec:moments}
In this appendix section, we recall some estimates used throughout the paper, borrowed from \cite{CMNP,Cav} (Lemma 2 and Lemma 3.1 respectively).

\begin{proposition}\label{prop:moments}
Let $T>0$, $p\ge 1$, $\mu_0\in\mathscr P_p(\mathbb R^d)$. Let $(\boldsymbol\mu,\boldsymbol\nu)\in\mathcal A_{[0,T]}(\mu_0)$, and let $\boldsymbol\eta$ be a representation for the pair $(\boldsymbol\mu,\boldsymbol\nu)$. Then, for $t\in]0,T]$, we have
\begin{itemize}
\item[$(i)$] $e_0\in L^p_{\boldsymbol\eta}$;
\item[$(ii)$] $\displaystyle\left\|\frac{e_t-e_o}{t}\right\|^p_{L^p_{\boldsymbol\eta}}\le 2^{p-1}De^{DT}\,(1+\mathrm{m}_p(\mu_0))<+\infty$;
\item[$(iii)$] $\mathrm{m}_p(\mu_t)\le K\, (1+\mathrm{m}_p(\mu_0))$,
\end{itemize}
where $D>0$ is coming from Lemma \ref{lemma:bprop} and $K=K(p,D,T)>0$.
\end{proposition}
\begin{proof}
First, recall that $(a+b)^p\le 2^{p-1}(a^p+b^p)$ for any $a,b\ge0$.\\
Item $(i)$ is immediate, indeed $\|e_0\|^p_{L^2_p}=\int_{\mathbb R^d\times\Gamma_T}|\gamma(0)|^p\,d\boldsymbol\eta=\int_{\mathbb R^d}|x|^p\,d(e_0\sharp\boldsymbol\eta)(x)=\mathrm{m}_p(\mu_0)<+\infty$.
Let us prove $(ii)$. By admissibility of $(\boldsymbol\mu,\boldsymbol\nu)$, we have that $v_t(x)=\frac{\nu_t}{\mu_t}(x)\in F(x)$. By Lemma \ref{lemma:bprop}, there exists $D>0$ such that $|v_t(y)|\le D(1+|y|)$, $y\in\mathbb R^d$. Thus, for $\boldsymbol\eta$-a.e. $(x,\gamma)$ we have
\begin{align*}
|\gamma(t)-\gamma(0)|&\le \left|\int_0^t|\dot\gamma(s)|\,ds\right|\le Dt+D\left|\int_0^t|\gamma(s)|\,ds\right|\\
&\le Dt(1+|\gamma(0)|)+D\left|\int_0^t|\gamma(s)-\gamma(0)|\,ds\right|.
\end{align*}
By Gronwall inequality,
\[|\gamma(t)-\gamma(0)|\le Dt\,e^{Dt}\,(1+|\gamma(0)|).\]
Dividing by $t\in]0,T]$ and taking the $L^p_{\boldsymbol\eta}$-norm, we have
\begin{align*}
\left\|\frac{e_t-e_0}{t}\right\|^p_{L^p_{\boldsymbol\eta}}&\le De^{DT}\left(\int_{\mathbb R^d\times\Gamma_T}(1+|\gamma(0)|)^p\,d\boldsymbol\eta\right)\\
&\le2^{p-1}D\,e^{DT}\,(1+\mathrm{m}_p(\mu_0))<+\infty.
\end{align*}

To prove $(iii)$, notice that
\[\mathrm{m}_p(\mu_t)=\int_{\mathbb R^d}|x|^p\,d\mu_t=\int_{\mathbb R^d\times\Gamma_T}|e_t(x,\gamma)|^p\,d\boldsymbol\eta=\|e_t\|^p_{L^p_{\boldsymbol\eta}},\]
and
\begin{align*}
\|e_t\|^p_{L^p_{\boldsymbol\eta}}&\le\left(\|e_0\|_{L^p_{\boldsymbol\eta}}+\|e_t-e_0\|_{L^p_{\boldsymbol\eta}}\right)^p\\
&\le2^{p-1}\left(\|e_0\|^p_{L^p_{\boldsymbol\eta}}+\|e_t-e_0\|^p_{L^p_{\boldsymbol\eta}}\right).
\end{align*}
We conclude by using the estimate in $(ii)$.
\end{proof}

\section*{Acknowledgments} 
The authors acknowledge the partial support of the NSF Project \emph{Kinetic description of emerging challenges in multiscale problems of natural sciences}, DMS Grant \# 1107444 and the endowment fund of the Joseph and Loretta Lopez Chair.\\
This work has been partially supported by the project of the Italian Ministry of Education, Universities and Research (MIUR) ``Dipartimenti di Eccellenza 2018-2022''.\\ G.C. has been supported by Cariplo foundation and Regione Lombardia through the project 2016-2018 ``Variational evolution problems and optimal transport''.
G.C. thanks the Department of Mathematical Sciences of Rutgers University - Camden (U.S.A.).

%

\end{document}